\documentclass[reqno,a4paper,oneside]{amsart}
\usepackage[T1]{fontenc}
\usepackage[utf8]{inputenc}
\usepackage[english]{babel}
\usepackage{microtype}
\usepackage{layaureo}
\usepackage{emptypage}
\usepackage{bm}
\usepackage{enumerate}
\usepackage{amsmath}
\usepackage{amsthm}
\usepackage{amssymb}
\usepackage{mathtools}
\usepackage{mathrsfs}
\usepackage{dsfont}
\usepackage{graphicx}
\usepackage{subfig}
\usepackage{caption}
\usepackage{cases}
\usepackage{lipsum}
\usepackage{afterpage}
\usepackage[dvipsnames]{xcolor}
\usepackage{listings}
\usepackage{comment}
\usepackage{hyperref}
\DeclarePairedDelimiter{\abs}{\lvert}{\rvert}
\DeclarePairedDelimiter{\norm}{\lVert}{\rVert}

\newcommand{\real}{\mathbb{R}}
\newcommand{\R}{\mathbb{R}}
\newcommand{\nat}{\mathbb{N}}
\newcommand{\Q}{\mathbb{Q}}
\newcommand{\prob}{\mathbb{P}}
\newcommand{\E}{\mathbb{E}}

\newcommand{\closure}[2][3]{{}\mkern#1mu\overline{\mkern-#1mu#2}}
\newcommand{\barr}[2][3]{{}\mkern#1mu\overline{\mkern-#1mu#2}}

\newcommand{\de}{\mathrm{d}}
\newcommand{\ev}{\mathrm{ev}}
\newcommand{\Lip}{\mathrm{Lip}}
\newcommand{\BL}{\mathrm{BL}}
\newcommand{\ind}{\mathds{1}}

\newcommand{\bx}{\bm{x}}
\newcommand{\by}{\bm{y}}
\newcommand{\blambda}{\bm{\lambda}}
\newcommand{\bX}{\mathbf{X}}
\newcommand{\bY}{\mathbf{Y}}
\newcommand{\bLambda}{\mathbf{\Lambda}}
\newcommand{\bv}{\bm{v}}
\newcommand{\bsigma}{\bm{\sigma}}
\newcommand{\bcT}{\bm{\cT}}
\newcommand{\bcG}{\bm{\cG}}
\newcommand{\bB}{\mathbf{B}}
\newcommand{\bh}{\bm{h}}

\newcommand{\barX}{\barr{X}}
\newcommand{\barLambda}{\barr{\Lambda}}
\newcommand{\barY}{\barr{Y}}
\newcommand{\barB}{\barr{B}}
\newcommand{\barSigma}{\barr{\Sigma}}
\newcommand{\barmu}{\barr{\mu}}
\newcommand{\barA}{\barr{A}}

\newcommand{\cC}{\mathcal{C}}
\newcommand{\cM}{\mathcal{M}}
\newcommand{\cX}{\mathcal{X}}
\newcommand{\cY}{\mathcal{Y}}
\newcommand{\cP}{\mathcal{P}}
\newcommand{\cT}{\mathcal{T}}
\newcommand{\cG}{\mathcal{G}}
\newcommand{\cR}{\mathcal{R}}

\newcommand{\cK}{\mathcal{K}}
\newcommand{\cZ}{\mathcal{Z}}
\newcommand{\cN}{\mathcal{N}}
\newcommand{\cS}{\mathcal{S}}
\newcommand{\cA}{\mathcal{A}}

\newcommand{\barcR}{\barr{\cR}}

\newcommand{\sX}{\mathscr{X}}
\newcommand{\sY}{\mathscr{Y}}
\newcommand{\sF}{\mathscr{F}}
\newcommand{\sB}{\mathscr{B}}
\newcommand{\sN}{\mathscr{N}}
\newcommand{\sA}{\mathscr{A}}
\newcommand{\barsA}{\barr{\sA}}
\newcommand{\sE}{\mathscr{E}}
\newcommand{\sC}{\mathscr{C}}

\newcommand{\Law}{\operatorname{Law}}

\newcommand{\diag}{\operatorname{diag}}

\newcommand{\red}[1]{\textcolor{red}{#1}}
\newcommand{\blue}[1]{\textcolor{blue}{#1}}
\newcommand{\green}[1]{\textcolor{ForestGreen}{#1}}

\numberwithin{equation}{section}

\theoremstyle{plain}
\newtheorem{thm}{Theorem}[section]

\newtheorem{lemma}[thm]{Lemma}
\newtheorem{prop}[thm]{Proposition}
\theoremstyle{definition}
\newtheorem{defn}[thm]{Definition}
\theoremstyle{remark}
\newtheorem{rem}[thm]{Remark}

\newtheorem*{notation}{Notational warning}

\title[Multi-agent models with strategies and diffusive effects]{Well-posedness and propagation of chaos\\for multi-agent models with strategies\\and diffusive effects}\thanks{\emph{Acknowledgments.}
AB and MM are members of the GNAMPA group of INdAM (Istituto Nazionale di Alta Matematica).
This research fits within the scopes of the GNAMPA Project 2024 \emph{Analisi asintotica di modelli evolutivi di interazione} and of the MUR grant \emph{Geometric Analytic Methods for PDEs and Applications} (2022SLTHCE cup E53D23005880006). This manuscript reflects only the authors’ views and opinions and the Italian Ministry cannot be considered responsible for them.}
\date{\today}

\author[A. Baldi]{Alessandro Baldi}
\address[Alessandro Baldi]{Dipartimento di Scienze Matematiche ``G. L. Lagrange'', Politecnico di Torino, Corso Duca degli Abruzzi, 24, 10129 Torino, Italy. ORCID: 0009-0006-4778-8431.}
\email{alessandro.baldi@polito.it}

\author[M. Morandotti]{Marco Morandotti}
\address[Marco Morandotti]{Dipartimento di Scienze Matematiche ``G. L. Lagrange'', Politecnico di Torino, Corso Duca degli Abruzzi, 24, 10129 Torino, Italy. ORCID: 0000-0003-3528-6152.}
\email{marco.morandotti@polito.it}


\begin{document}

\begin{abstract}
A multi-agent model for individuals endowed with strategies and subject to diffusive effects is proposed. 
The microscopic state of each agent is described by a spatial position and a probability measure, interpreted as a mixed strategy, over a compact metric space.
The evolution is governed by a non-local interaction mechanism and by stochastic effects acting on the spatial component of the state.
The well-posedness of the multi-agent system and that of a certain McKean--Vlasov stochastic differential equation are proved.
Eventually, a propagation of chaos result is obtained, which guarantees that the former model converges to the latter as the number of agents goes to infinity.
\end{abstract}

\maketitle

\allowdisplaybreaks

\tableofcontents

\section{Introduction}
Multi-agent models offer a paradigm to describe a huge variety of complex systems driven by interactions among the constituting individuals. 
Since the number of agents is typically very large, following the evolution of each of them can be a demanding, often unfeasible, task, and both analytical approaches and the algorithmic complexity in simulating such large ensembles call
for reduced models that can capture the macroscopic, global behaviour of these systems. 
The idea of providing macroscopic descriptions emerging from the microscopic behaviour of particles dates back to Boltzmann's work \cite{Boltzmann1970}, where one of the fundamental assumptions is the so-called \emph{Stosszahlansatz}, namely the independence of the bahaviour of distinct particles when their number is very large. 
One of the most successful attempts to justify this approach can be found in Kac's seminal paper \cite{Kac}, where he introduces the notion of \emph{propagation of chaos}, which, shortly afterwards, was shown by McKean \cite{McKean1967} to be satisfied by a certain class of diffusion models.

Systems that can be studied with this approach model, for instance, 
opinion formation \cite{AMS33,Toscani-Opinioni}, 
wealth distribution \cite{AMS32,AMS34,Par-Tosc-Ricchezza}, 
traffic or pedestrian flow \cite{AMS2,AMS31,AMS49,AMS50,AMS55}, 
herding problems \cite{AMS1,AMS2,AMS9,AMS20,AMS43,AMS51}, 
flocking and swarming \cite{AMS11,ToscaniRosado,CFTV,CuckerSmale},
consensus-based optimisation \cite{ABMS2025,AMS17,AMS26,AMS36,FKR2024,HoffmannCBO,CBO-Originale,AMS56},
including applications to pattern formation and chemical reaction networks \cite{AMS18,Bibbona}, often stimulating the development of novel rigorous analytical techniques, see, \emph{e.g.}, \cite{AMS28,AMS47,AMS48}.
We refer the reader to \cite{PareschiToscani_book} for a comprehensive overview of modelling through multi-agent kinetic equations, and to \cite{ChaosReviewI,ChaosReviewII} for a thorough review of propagation of chaos techniques.
In recent years, much attention has been drawn to two enhancements of multi-agent systems, with the aim of devising models that can better describe the above-mentioned complex phenomena: the addition, to the microscopic state, of more variables that may represent the agents' strategies or their belonging to a certain population 
\cite{AFMS,Ascione-Castorina-Solombrino,CristianiLoyTosin,Loy-Tosin,Morandotti-Solombrino,thai2015birthdeathprocessmean}, and the introduction of stochastic terms which alter the deterministic character of the dynamics, to account for randomness.

\smallskip

We study here multi-agent models expanding those in \cite{AFMS,Morandotti-Solombrino} by introducing a diffusive term to the spatial evolution of agents in addition to the classical drift term that describes the spatial velocity.
The microscopic state features also the presence of a variable, referred to as \emph{mixed strategy}, describing the behavioural pattern of agents.
The evolution of this microscopic state is regulated by non-local interactions.

Given $N\in\nat^+$, a time horizon $T>0$, and a compact metric space $U$ of \emph{pure strategies}, we consider a population of $N$ agents, each of which is identified by a label $i \in \{1,\dots,N\}$; the microscopic state of the $i$-th agent at time $t \in [0,T]$ is represented through the random variable $Y^i_t = (X^i_t,\Lambda^i_t)$ taking values in the space $\real^d \times \cP(U)$, where $\cP(U)$ is the convex set of Borel probability measures over $U$.
The $\real^d$-valued random variable $X^i_t$ can be interpreted as the spatial position of the agent $i$, while the random variable $\Lambda_t^i$ can be viewed, borrowing the language of game theory, as the agent's \emph{mixed strategy}.\footnote{Despite the fact that our model includes the case of~$U$ containing a continuum of pure strategies, the concept is perhaps more clear in the simpler and very peculiar case in which $U=\{u_1,\ldots,u_M\}$ contains a finite number of pure strategies.
In particular, the simplest case is instructive: if $M=2$, then we can consider $U=\{u_1,u_2\}=\{F,L\}$, so that $\Lambda_t^i\in\cP(U)$ can be identified with a parameter in $[0,1]$ describing the degree of leadership of agent~$i$ at time~$t$, ranging from being a leader to being a follower. 
If $M>2$, then~$U$ can represent, for instance, the different assets in a financial portfolio, or different populations an individual can belong to, see, \emph{e.g.}, \cite{Morandotti-Solombrino} for a few examples; in this case, $\Lambda_t^i\in\cP(U)$ can be identified with a point in the $(M-1)$-dimensional simplex $\Delta^{M-1}\coloneqq\big\{\lambda\in\real^M: \text{$\sum_{j=1}^M \lambda_j=1$ and $\lambda_j\ge0$ for every $j=1,\ldots,M$}\big\}$.}

The states of the 
agents 
evolve 
according to the stochastic differential equation (SDE)
\begin{equation}
\begin{dcases}
\de X^i_t = v_{\Sigma^N_t}(X^i_t,\Lambda^i_t)\,\de t + \sigma_{\Sigma^N_t}(X^i_t,\Lambda^i_t)\,\de B^i_t\,, \\
\de \Lambda^i_t = \cT_{\Sigma^N_t}(X^i_t,\Lambda^i_t)\,\de t,
\end{dcases}
\qquad \text{for $t \in [0,T]$ and $i = 1,\dots,N$,}
\label{eq:NDicsreteProblem}
\end{equation}
where $\Sigma^N_t \coloneqq \frac{1}{N} \sum_{j=1}^N \delta_{(X^j_t,\Lambda^j_t)}$ denotes the \emph{empirical measure} of the system, starting from given initial states $Y_0^i = (X_0^i,\Lambda_0^i)$ ($i = 1,\dots,N$).
In \eqref{eq:NDicsreteProblem}, the field~$v$ models drift, whereas the diffusive effects are accounted for the $m$-dimensional standard Brownian motions $B^1,\dots,B^N$, whose effects are modulated by the $\real^{d\times m}$-valued field~$\sigma$; the measure-valued field~$\cT$ in the second equation regulates the evolution of the mixed strategies.
The interaction mechanism is encoded in the dependence of these three fields on the empirical measure $\Sigma^N_t$\, which gives \eqref{eq:NDicsreteProblem} a non-linear and non-local character.

When the number $N$ of agents becomes very large, the rigorous study of the coupled SDE \eqref{eq:NDicsreteProblem} may be hindered, due to the complex dependencies that may arise among the random variables $Y^i_t$\,.
To overcome this difficulty, one seeks to approximate the $N$-particle system by a single SDE describing the dynamics of a representative agent, by performing a mean-field limit: we forgo the description of the evolution of every single agent of the population in favour of a representation of the behaviour of a generic member of the population. 
However, to properly perform this limiting procedure, it is necessary to prove a \emph{propagation of chaos} result, that is, when $N\to\infty$, the microscopic states of the agents $\{Y^i_t\}_{i= 1,\dots,N}$ become independent and identically distributed, provided that the i.i.d.~property holds at the initial time $t = 0$.
Intuitively, if propagation of chaos holds, owing to the asymptotic i.i.d.~property of the random variables $Y_t^i$, we can approximate, for large values of $N$, the empirical measure $\Sigma_t^N$ by the (deterministic) law $\Sigma_t\coloneqq \Law(Y_t^1)$ of the representative individual.
Therefore, by formally plugging the common law $\Sigma_t$ in \eqref{eq:NDicsreteProblem}, the dynamics decouples and it is reasonable to postulate that the mean-field description as $N\to\infty$ is provided by the \emph{McKean--Vlasov}~SDE
\begin{subequations}\label{eq:MeanField}
\begin{equation}\label{eq:MeanFieldProblem}
\begin{dcases}
\de\barr{X}_t = v_{\Sigma_t}(\barr{X}_t,\barr{\Lambda}_t)\,\de t + \sigma_{\Sigma_t}(\barr{X}_t,\barr{\Lambda}_t)\,\de\barr{B}_t\,, \\
\de\barr{\Lambda}_t = \cT_{\Sigma_t}(\barr{X}_t,\barr{\Lambda}_t)\,\de t,
\end{dcases}
\end{equation}
subject to the requirement that
\begin{equation}\label{eq:MeanFieldSelfReferential}
    \Law(\barY_t) = \Sigma_t, \qquad \text{for all } t \in [0,T]. 
\end{equation}
\end{subequations}
In \eqref{eq:MeanField}, $\barY_t=(\barX_t\,,\barLambda_t)$ denotes the state of the representative individual.

One of the peculiarities of systems \eqref{eq:NDicsreteProblem} and \eqref{eq:MeanField} is that the state space of the agents is $\real^d \times \cP(U)$, which lacks the vector space structure and is, in general, infinite-dimensional.
By embedding it into a suitable Banach space (see Section~\ref{sec:F(U)}), we will be able to take advantage of the techniques to study evolution problems in closed convex subspaces of Banach spaces, see \cite[Chapitre~I.3]{Brezis-Operateur-Maximaux} (see also \cite{AFMS}).

\subsection{Main results and technical details}
We now proceed to describe the three main contributions of this paper, together with the relevant technical details.

Our first main result, Theorem~\ref{thm:WellPosednessNParticleSystem} below, concerns the well-posedness of the $N$-particle systems \eqref{eq:NDicsreteProblem}. 
We prove the existence of a pathwise unique strong solution (see Definition~\ref{def:StrongSolutionAndUniqueness} below), under the hypothesis of square integrability on the spatial initial data~$\{X_0^i\}_{i=1}^N$\,. 
Moreover, we obtain that the integrability of the initial data is inherited by the solution at all future times.
To show the existence of the solutions to \eqref{eq:NDicsreteProblem}, we combine a Picard iteration scheme with Br\'{e}zis's strategy \cite[Chapitre~I.3]{Brezis-Operateur-Maximaux} to treat the infinite-dimensional components~$\Lambda^i_t$ belonging to the convex set~$\cP(U)$.
As in \cite{AFMS}, the crucial hypothesis on the fields is the global Lipschitz continuity (see \eqref{eq_fieldsprop} below), together with the geometric property~\eqref{eq:T_prop_geometry} of the field~$\cT$, which ensures that the evolution of the $\Lambda_t^i$'s does not escape $\cP(U)$.
We stress that system \eqref{eq:NDicsreteProblem} is a generalisation of \cite[equation (3.5)]{AFMS}, with respect to which a state-dependent stochastic component is added to the spatial dynamics (we mention that, in the deterministic case, a generalisation of \cite[equation (3.5)]{AFMS} has been studied in \cite[equation (1.3)]{Morandotti-Solombrino}).

Our second main result, Theorem~\ref{thm:WellPosednessMeanField} below, concerns the well-posedness of the McKean--Vlasov SDE~\eqref{eq:MeanField}.
We prove the existence of a strong solution (see Definition~\ref{def:StrongSolutionMF} below), under the hypothesis of square integrability on the spatial initial datum~$\barX_0$\,. 
Also in this case, the integrability of the initial datum is inherited by the solution uniformly in time.
Uniqueness is obtained in the class of strong solutions whose spatial first moment is uniformly bounded in time.
To show the existence of solutions to~\eqref{eq:MeanField}, we resort to an auxiliary SDE (see~\eqref{eq:AuxiliaryMeanField} below), which allows us to construct a map~$\cS$ (see Proposition~\ref{prop:DefS}) in a way that condition~\eqref{eq:MeanFieldSelfReferential} can be interpreted as a fixed-point problem for~$\cS$.
The latter is solved by applying the Banach--Caccioppoli fixed-point Theorem.
This strategy is an adaptation of that of Sznitman \cite{Sznitman} to the case in which the state space is the more general $\real^d \times \cP(U)$.

Our third main result, Theorem~\ref{thm:propagation_of_chaos} below, concerns the propagation of chaos.
This is achieved via \emph{synchronous coupling} \cite[Section~3]{ChaosReviewII} (see also \cite{Sznitman}) between the trajectories of the $N$-particle system~\eqref{eq:NDicsreteProblem} and the trajectories of $N$ independent copies of the solution to the mean-field equation.
Thanks to this result, the McKean--Vlasov SDE~\eqref{eq:MeanField} is indeed the mean-field limit of the $N$-particle system~\eqref{eq:DicsreteProblem} as $N\to\infty$.

\smallskip

The paper is structured as follows: in Section~\ref{sec:prel}, we set the notation, collect some preliminaries on measure theory, and fix the functional setting of the problem. 
Section~\ref{sec:NParticleModel} is devoted to studying the $N$-particle system~\eqref{eq:NDicsreteProblem}, Section~\ref{sec:MeanField} tackles the well-posedness of the McKean--Vlasov SDE~\eqref{eq:MeanField}, and propagation of chaos is proved in Section~\ref{sec:propagation_of_chaos}.
Appendix~\ref{app} contains the proofs of some technical probabilistic results.


\section{Notation and preliminaries}\label{sec:prel}
We collect here some preliminaries on measure theory, probability measures, stochastic processes, and present the functional setting of our problem.

\subsection{Distances on the space on probability measures} If $(\cX,d_{\cX})$ is a metric space, we denote by $\cM(\cX)$ the vector space of signed Borel measures over $\cX$ with finite total variation; we also indicate by $\cP(\cX)$ the convex subset of probability measures, by $\cM_+(\cX)$ the convex cone of non-negative measures and by $\cM_0(\cX)\coloneqq \R(\cP(\cX) - \cP(\cX))$ the vector subspace of measures with zero mass.

We recall the notion of \emph{push-forward} of measures.
\begin{defn} Let $(\cX,\sX,\mu)$ be a measure space (with $\mu \in \cM_+(\cX)$) and $(\cY,\sY)$ a measurable space. Given a measurable function $f \colon \cX \longrightarrow \cY$, we define the measure $f_\sharp\mu$ on $(\cY,\sY)$ as
\begin{equation}
f_\sharp\mu(B) \coloneq \mu(f^{-1}(B)),\qquad \text{for all $B \in \sY$.}
\end{equation}
The measure $f_\sharp\mu$ is called \emph{push-forward} (or \emph{image measure}) of $\mu$ through $f$.
\end{defn}
By construction, the image measure $f_\sharp\mu$ is non-negative and it has the same total mass as $\mu$; in particular, if $\mu \in \cP(\cX)$ then also $f_\sharp\mu \in \cP(\cX)$. 
The integrals with respect to $\mu$ and $f_\sharp\mu$ are related by the following \emph{change of variable formula}.
\begin{thm} Let $(\cX,\sX,\mu)$ be a measure space (with $\mu \in \cM(\cX)$), $(\cY,\sY)$ a measurable space and $f \colon \cX \longrightarrow \cY$ a measurable map. Then, for every measurable $\varphi \colon \cY \longrightarrow \real$, the following relation holds
\begin{equation}
    \int_{\cX} \varphi(f(x))\,\de\mu(x) = \int_{\cY} \varphi(y)\,\de(f_\sharp\mu)(y).
\end{equation}
\label{thm:IntegrazioneMisImg}
\end{thm}
If $(\Omega,\sF,\prob)$ is a probability space and $X \colon \Omega \longrightarrow \sY$ is a random variable, the probability measure $X_\sharp\prob$ over $\sY$ is called the \emph{law} of $X$, which we shall denote by $\Law(X)$.

In the sequel we will frequently employ a notable class of metrics defined on spaces of probability measures: the \emph{Wasserstein distances}. We start by defining the set of probability measures with finite $p$-th moment.
\begin{defn}  Let $(\cX,d_\cX)$ be a metric space and $p \in [1,+\infty)$. The set of \emph{probability measures with finite $p$-th moment} $\cP_p(\cX)$ is defined as
\begin{equation}
    \cP_p(\cX) \coloneqq \biggl\{\mu \in \cP(\cX) : \int_\cX d_\cX(x,x_0)^p\,\de\mu(x) < \infty \textrm{ for some } x_0 \in \cX \biggr\}.
\label{eq:PpDef}
\end{equation}
\end{defn}
We note that the definition of $\cP_p(\cX)$ is actually independent of the choice of the point $x_0$ and that if $(\cX, d_\cX)$ is a compact metric space then $\cP_p(\cX) = \cP(\cX)$ for every $p \in [1,\infty)$.

If $(\cX,d_\cX)$ is a complete and separable metric space, the sets $\cP_p(\cX)$ can be endowed with a metric structure by equipping them with the \emph{Wasserstein distances}, whose definition and main properties we recall thereafter (for a detailed treatment see, \emph{e.g.}, \cite{AGS,Optimal-Transport-Villani}).
\begin{defn} Let $(\cX,d_\cX)$ be a complete and separable metric space and let $p \in [1,+\infty)$. For every $\mu, \nu \in \cP_p(\cX)$ the \emph{Wasserstein distance of order $p$} between $\mu$ and $\nu$ is defined as
\begin{equation}
W_p^p(\mu,\nu) \coloneqq \inf_{\pi \in \Gamma(\mu,\nu)}\int_{\cX\times\cX} d_\cX(x_1,x_2)^p\,\de \pi(x_1,x_2),
\label{eq:WasserDef}
\end{equation}
where $\Gamma(\mu,\nu)$ denotes the sets of \emph{transport plans} between $\mu$ and $\nu$, given by
\begin{equation*}
\Gamma(\mu,\nu) \coloneqq \bigl\{ \pi \in \cP(\cX\times\cX) : (p_1)_\sharp \pi = \mu, \;(p_2)_\sharp \pi = \nu\bigr\},
\label{eq:TransportPlan}
\end{equation*}
where $p_1, p_2 \colon \cX\times\cX \longrightarrow \cX$ denote the canonical projections on the first and second factor, respectively.
\label{def:WasserDef}
\end{defn}
The general theory of Optimal Transport (see, e.g., \cite{Optimal-Transport-Ambrosio,AGS,Optimal-Transport-Santambrogio,Optimal-Transport-Villani}) ensures that the infimum in (\ref{eq:WasserDef}) is actually attained, so that the definition of $W_p$ is well-posed. Furthermore, it can be shown that $W_p$ satisfies all the properties of a metric.
\begin{defn}[Wasserstein space] Let $(\cX,d_\cX)$ be a complete and separable metric space and $p \in [1,+\infty)$. The function $W_p$ defines a metric over the set $\cP_p(\cX)$ of probability measures with finite $p$-th moment. The metric space $(\cP_p(\cX),W_p)$ is called the \emph{Wasserstein space} of order $p$.
\end{defn}
Wasserstein spaces inherit form $(\cX,d_\cX)$ the properties of completeness and separability.
\begin{thm} Let $(\cX,d_\cX)$ be a complete and separable metric space and $p \in [1,+\infty)$. Then the Wasserstein space $(\cP_p(\cX),W_p)$ is complete and separable.
\label{thm:WasserPolish}
\end{thm}
We conclude this paragraph by recalling the definition and fundamental properties of the spaces of Lipschitz continuous functions. Let $(\cX,d_\cX)$ be a metric space and $\varphi \colon \cX \longrightarrow \real$ a Lipschitz continuous function. We define \emph{Lipschitz constant} of $\varphi$ the quantity
\begin{equation}
    \Lip(\varphi) \coloneqq \sup_{\substack{x_1,x_2 \in \cX \\ x_1 \neq x_2}}\frac{\abs{\varphi(x_1)-\varphi(x_2)}}{d_\cX(x_1,x_2)}.
\end{equation}
We denote by $\Lip_b(\cX)$ the set of functions $\varphi \colon \cX \longrightarrow \real$ that are bounded and Lipschitz continuous over $\cX$. $\Lip_b(\cX)$ is a real vector space and can be endowed with the norm $\norm{\cdot}_\Lip$ defined as
\begin{equation}
    \norm{\varphi}_\Lip \coloneqq \sup_{x \in \cX}\abs{\varphi(x)} + \Lip(\varphi)
\end{equation}
for every $\varphi \in \Lip_b(\cX)$. If $(\cX,d_\cX)$ is a compact metric space, by Weierstrass Theorem, $\Lip_b(\cX)$ coincide with the vector space of Lipschitz function over $\cX$, which we simply denote by $\Lip(\cX)$.

\subsection{The Banach space $F(U)$}\label{sec:F(U)}
One of the peculiarities of the system \eqref{eq:NDicsreteProblem} lies in the fact that the dynamics of the mixed strategies $\Lambda^i_t$ takes place in the convex set $\cP(U)$, which inherently does not posses the structure of a normed vector space. Therefore, it is useful to preliminarily embed the set of probability measures $\cP(U)$ into a suitable Banach space, thereby taking advantage of its vector space and metric structures, and of some analytical tools such as Bochner integration. 
The construction that we adopt follows that of~ \cite[Section 2.1]{AFMS}; here, we provide an outline of the procedure, referring to the work of Ambrosio \emph{et al.}~and to \cite{Ambrosio-Puglisi, Arens-Eells, WeaverLip}  for further details.

Let us consider the metric space $(U,d_U)$, which we recall to be compact by hypothesis; we consider the Banach space $(\Lip(U),\norm{\cdot}_\Lip)$, from which we can construct the 
topological dual $(\Lip(U))^*$ endowed with the customary dual norm, referred to as the \emph{bounded Lipschitz norm} (or $\BL$ norm for brevity), defined as
\begin{equation}
\norm{\ell}_\BL \coloneqq \sup\Bigl\{\abs{\langle\ell,\varphi\rangle} \colon \varphi \in \Lip(U),\;\norm{\varphi}_\Lip \le 1\Bigr\}
\label{eq:DefNormaF(U)}
\end{equation}
for every  $\ell \in (\Lip(U))^*$. 
Now, by identifying each probability measure on $U$ with a linear and continuous functional on $\Lip(U)$, we can embed $\cP(U)$ into a Banach space, which we define as follows.
\begin{defn} We call $F(U)$ the closed subspace of the Banach space $\bigl((\Lip(U))^*,\norm{\cdot}_\BL\bigr)$ defined as
\begin{equation}
    F(U) \coloneqq \closure[-1]{\mathrm{span}(\cP(U))}^{\norm{\cdot}_\BL},
\end{equation}
where $\mathrm{span}(\cP(U))$ denotes the set of all finite linear combinations of elements of $\cP(U)$; here, each measure $\lambda \in \cP(U)$ is identified with the element of $(\Lip(U))^*$ defined by
\begin{equation}
\varphi \longmapsto \int_U \varphi(u)\,\de\lambda(u),\qquad\text{for every $\varphi\in \Lip(U)$.}
\end{equation}
\label{def:F(U)}
\end{defn}
\begin{rem}\label{remark1}
For notational consistency, we shall use the symbol $\norm{\cdot}_{F(U)}$ to indicate the $\BL$ norm restricted to $F(U)$.
The space $F(U)$ is known in the literature as \emph{Arens-Eells space}, and it can be shown to be a separable Banach space containing $\cM(U)$. Additionally, $\cP(U)$ 
turns out to be a compact subset of $F(U)$ with respect to the topology induced by the $\BL$ norm. 
\end{rem}
In view of the importance of these properties throughout the present article, we highlight them in the following proposition. 
\begin{prop}[properties of $F(U)$] Let $(U,d_U)$ be a compact metric space. Then the normed vector space $(F(U),\norm{\cdot}_{F(U)})$ introduced in Definition~\ref{def:F(U)} is a separable Banach space. Furthermore, the convex subset $\cP(U)$ is compact in $F(U)$.
\label{prop:ProprF(U)P(U)}
\end{prop}
It can be shown that the convergence of probability measures with respect to the $\BL$ norm is closely related to the convergence in the $1$-Wasserstein distance $W_1$, which, in turn, is related to the weak convergence of probability measures (that is, in the duality with continuous and bounded functions). For details, we refer again to \cite[Section~2.1]{AFMS}; we highlight that 
\begin{equation}\label{eq_triangolino}
\norm{\lambda}_{F(U)} \le 1,\qquad \text{for all $\lambda \in \cP(U)$.}
\end{equation}
\begin{rem}
We point out that all integrals of $F(U)$-valued functions are to be intended as Bochner integrals, see \cite[Appendix A]{AFMS} and \cite{Yosida}.
\end{rem}

\subsection{Functional setting of the problem}
We can now formulate problem \eqref{eq:NDicsreteProblem} in a suitable analytical framework. 
We recall that the  microscopic states $Y^i_t = (X^i_t,\Lambda^i_t)$ belong to $\real^d \times \cP(U)$ for all $i \in \{1,\dots, N\}$ and for all times $t \in [0,T]$.
In view of Proposition~\ref{prop:ProprF(U)P(U)},
we shall always consider $\real^d \times \cP(U)$ as a subset of 
$\real^d \times F(U)$, which
can naturally be endowed with a Banach space structure by considering the product norm $\norm{\cdot}_{\real^d \times F(U)}$, defined as
\begin{equation}\label{eq_normay}
    \norm{y}_{\real^d \times F(U)} \coloneqq \norm{x}_{\real^d} + \norm{\lambda}_{F(U)},\qquad\text{for $y=(x,\lambda)\in \real^d\times F(U)$.}
\end{equation}
Here, $\norm{\cdot}_{\real^d}$ denotes the customary Euclidean norm on $\real^d$, and $\norm{\cdot}_{F(U)}$ denotes the norm on $F(U)$ (see Remark~\ref{remark1}). 

It will be often useful to consider the product spaces $(\real^d)^N$, $F(U)^N$, and $(\real^d \times F(U))^N$ (with $N \in \nat^+$), which we also equip with the customary product norms, given by
\begin{equation*}
\norm{\bx}_{(\real^d)^N} \coloneqq \sum_{i=1}^N\norm{x_i}_{\real^d}\,, \qquad \norm{\blambda}_{(F(U))^N} \coloneqq \sum_{i=1}^N\norm{\lambda_i}_{F(U)}\,, 
\end{equation*}
and
\begin{equation*}
\norm{\by}_{(\real^d\times F(U))^N} \coloneqq \sum_{i=1}^N\norm{x_i}_{\real^d} + \sum_{i=1}^N\norm{\lambda_i}_{F(U)} = \norm{\bx}_{(\real^d)^N} + \norm{\blambda}_{(F(U))^N}\,,
\end{equation*}
respectively, for all $\bx = (x_1,\dots,x_N) \in (\real^d)^N$, $\blambda = (\lambda_1,\dots,\lambda_N) \in (F(U))^N$, and $\by = (x_1,\lambda_1,\dots,x_N,\lambda_N) \in (\real^d \times F(U))^N$. Since $(\real^d,\norm{\cdot}_{\real^d})$ and $(F(U),\norm{\cdot}_{F(U)})$ are separable Banach spaces, so are $((\real^d)^N,\norm{\cdot}_{(\real^d)^N})$, $((F(U))^N,\norm{\cdot}_{(F(U))^N})$, and $((\real^d \times F(U))^N,\norm{\cdot}_{(\real^d \times F(U))^N})$. 
By \eqref{eq_triangolino}, it follows that 
\begin{equation}\label{eq_doppiotriangolino}
\norm{\blambda}_{(F(U))^N} \le N,\qquad \text{for all $\blambda \in (\cP(U))^N$.}
\end{equation}

On the vector space $\real^{d\times m}$ 
we shall consider the Frobenius norm
\begin{equation*}
    \norm{\sigma}_{\real^{d\times m}} \coloneqq \sqrt{\sum_{i=1}^d \sum_{j=1}^m  \sigma_{ij}^2}\,,\qquad\textup{ for all } \sigma \in \real^{d\times m},
\end{equation*}
while for real matrices $\bsigma$ of dimension $dN\times mN$ obtained by the juxtaposition of $N^2$ matrices of dimension $d \times m$ we stipulate that their norm be given by the sum of the norms of the constituting blocks: more explicitly, we define
\begin{equation*}
\norm{\bsigma}_{\real^{dN\times mN}} \coloneqq \sum_{k=1}^N\sum_{\ell=1}^N \norm{\sigma_{k\ell}}_{\real^{d\times m}},\quad\textup{ for all } 
\bsigma =
\begin{pmatrix}
    \sigma_{11} & \cdots & \sigma_{1N} \\
    \vdots & \ddots & \vdots \\
    \sigma_{N1} & \cdots & \sigma_{NN}
\end{pmatrix}
\in \real^{dN\times mN}.
\end{equation*}

When considering the space $\cC([0,T],E)$ of continuous functions taking values in a Banach space $(E,\norm{\cdot}_E)$, we always endow it with the uniform norm
$$\norm{f}_\infty \coloneqq \max_{t\in[0,T]} \norm{f(t)}_E\,,\qquad \text{for every $f\in \cC([0,T],E)$,}$$
which makes it a Banach space. We recall that if $(E,\norm{\cdot}_E)$ is separable, so is $(\cC([0,T],E),\norm{\cdot}_\infty)$.


As a last item, we state an inequality for stochastic integrals. 
\begin{thm}[Burkholder--Davis--Gundy inequality] 
Let $\bigl(\Omega, \sF, (\mathscr{F}_t)_{t\in[0,T]},\prob\bigr)$ be a filtered probability space and let $(B_t)_{t\in[0,T]}$ be a standard $\real^m$-valued Brownian motion and $(G_t)_{t\in[0,T]}$ be a progressively measurable, $\real^{d\times m}$-valued stochastic process, both defined on $\bigl(\Omega, \sF, (\mathscr{F}_t)_{t\in[0,T]},\prob\bigr)$.
Let $p \ge 2$ and suppose that $\int_0^T\,\norm{G_t}_{\real^{d\times m}}^p\,\de t < +\infty$, $\prob$-almost surely. 
Then for all $t \in [0,T]$, the following inequality holds:
\begin{equation}\label{eq_BDG}
    \E\biggl[\sup_{u\in[0,t]}\norm*{\int_0^u G_s\,\de B_s}_{\real^d}^p\biggr] \le c_p\, t^\frac{p-2}{2}\, \E\biggl[\int_0^t\norm{G_s}_{\real^{d\times m}}^p\,\de s\biggr]
\end{equation}
where 
$c_p = \bigl[\frac{1}{2}\, p^{p+1} (p-1)^{1-p}\bigr]^\frac{p}{2}$.
\end{thm}
We refer the reader to \cite[Proposition 8.4]{Stochastic-Calculus} and \cite[page 116]{SV} for further details.

\subsection{Structural hypotheses on the fields $v$, $\sigma$, and $\cT$}
We list here the standing assumptions on the fields
\begin{subequations}\label{eq_hpfields}
\begin{align}
v &\colon \cP_1(\real^d \times \cP(U)) \times \real^d \times \cP(U) \longrightarrow \real^d \\
\sigma &\colon \cP_1(\real^d \times \cP(U)) \times \real^d \times \cP(U) \longrightarrow \real^{d\times m} \\
\cT &\colon \cP_1(\real^d \times \cP(U)) \times \real^d \times \cP(U) \longrightarrow \cM_0(U) \subseteq F(U) %
\end{align}
\end{subequations}
in the right-hand side of \eqref{eq:NDicsreteProblem}.
We assume that there exist constants $L_v\,, L_\sigma\,, L_\cT\,, \theta > 0$ such that the following properties hold:
for all $(\Sigma_1,x_1,\lambda_1),(\Sigma_2,x_2,\lambda_2)\in \cP_1(\real^d \times \cP(U)) \times \real^d \times \cP(U)$,
\begin{subequations}\label{eq_fieldsprop}
\begin{align}
&\hspace{-3mm}\norm{v_{\Sigma_1}(x_1,\lambda_1) - v_{\Sigma_2}(x_2,\lambda_2)}_{\real^d} \!\le\! L_v \big( \norm{x_1 - x_2}_{\real^d} + \norm{\lambda_1 - \lambda_2}_{F(U)}+W_1(\Sigma_1,\Sigma_2) \big), \label{eq:v_prop}\\
&\hspace{-3mm}\norm{\sigma_{\Sigma_1}(x_1,\lambda_1) - \sigma_{\Sigma_2}(x_2,\lambda_2)}_{\real^{d\times m}} \!\le\! L_\sigma \big( \norm{x_1 - x_2}_{\real^d} + \norm{\lambda_1 - \lambda_2}_{F(U)} + W_1(\Sigma_1,\Sigma_2) \big), \label{eq:sigma_prop}\\
&\hspace{-3mm}\norm{\cT_{\Sigma_1}(x_1,\lambda_1) - \cT_{\Sigma_2}(x_2,\lambda_2)}_{F(U)} \!\le\! L_\cT \big( \norm{x_1 - x_2}_{\real^d} + \norm{\lambda_1 - \lambda_2}_{F(U)} + W_1(\Sigma_1,\Sigma_2)), \label{eq:T_prop}
\end{align}
\end{subequations}
and, for all $(\Sigma,x,\lambda)\in \cP_1(\real^d \times \cP(U)) \times \real^d \times \cP(U)$, 
\begin{equation}\label{eq:T_prop_geometry}
\lambda + \theta \cT_\Sigma(x,\lambda) \in \cP(U).
\end{equation}

Requirements 
\eqref{eq_fieldsprop} correspond to global Lipschitz continuity of the field with respect to all their arguments. Instead, assumption 
\eqref{eq:T_prop_geometry} is geometric in nature: intuitively, it expresses the fact that the field $\cT$ always points inside the convex set $\cP(U)$, thereby preventing the trajectories of the mixed strategy components from escaping the set of probability measures.

\section{The $N$-particle model}\label{sec:NParticleModel}
Let us consider a time horizon $T > 0$ and a fixed number of particles $N\in\nat^+$. We begin studying the well posedness of \eqref{eq:NDicsreteProblem} by introducing a suitable notion of solution.

\begin{defn}\label{def:StrongSolutionAndUniqueness}
Let  $B^1,\ldots,B^N \colon \Omega \longrightarrow \cC([0,T],\real^m)$ be $m$-dimensional standard Brownian motions defined on a filtered probability space $\bigl(\Omega, \sF, (\mathscr{F}_t)_{t\in[0,T]},\prob\bigr)$; let $X_0^1,\ldots,X_0^N\colon\Omega\longrightarrow\real^d$ be $\mathscr{F}_0$-measurable random variables, and let $\Lambda_0^1,\dots,\Lambda_0^N\colon\Omega\longrightarrow\mathcal{P}(U)$ be 
$\mathscr{F}_0$-measurable random variables.
We define a \emph{strong solution} to problem \eqref{eq:NDicsreteProblem} a continuous, $(\mathscr{F}_t)_t$-adapted stochastic process $\bY = (X^1,\Lambda^1,\dots,X^N,\Lambda^N) \colon \Omega \longrightarrow \cC\bigl([0,T],(\real^d\times\cP(U))^N\bigr)$ satisfying, $\prob$-almost surely, for all $t \in [0,T]$ and $i = 1,\dots,N$,
\begin{equation}\label{eq_NparticleIntegral}
\begin{dcases}
X^i_t = X^i_0 + \int_0^t v_{\Sigma^N_s}(X^i_s,\Lambda^i_s)\,\de s + \int_0^t\sigma_{\Sigma^N_s}(X^i_s,\Lambda^i_s)\,\de B^i_s \,,\\
\Lambda^i_t = \Lambda^i_0 + \int_0^t\cT_{\Sigma^N_s}(X^i_s,\Lambda^i_s)\,\de s \,,
\end{dcases}
\end{equation}
where $\Sigma^N_t = \frac{1}{N} \sum_{j=1}^N \delta_{(X^j_t,\Lambda^j_t)}\in \mathcal{P}(\real^d\times\mathcal{P}(U))$, for every $t\in[0,T]$.
We say that such a strong solution is \emph{pathwise unique} if, given two strong solutions $\bY_1$, $\bY_2$ of \eqref{eq:NDicsreteProblem} (with the same Brownian motions and initial data), we have
\begin{equation}\label{eq_sarapathwiseuniqueness}
    \prob\Bigl(\bY_{1,t} = \bY_{2,t}, \textup{ for every } t\in[0,T]\Bigr) = 1.
\end{equation}
\end{defn}

The main result of this section will be the following well-posedness theorem.
\begin{thm}[well-posedness of the $N$-particle system~\eqref{eq:NDicsreteProblem}]\label{thm:WellPosednessNParticleSystem} 
Let us assume that \eqref{eq_fieldsprop} and \eqref{eq:T_prop_geometry} are satisfied and that the spatial initial data $X_0^1,\dots,X_0^N$ belong to $L^2(\Omega,\mathscr{F},\prob)$. Then problem \eqref{eq:NDicsreteProblem} admits a pathwise unique strong solution $\bY$. 
Moreover, the integrability of the spatial initial data is inherited, uniformly in time, by the solution, namely, if  $X_0^1,\dots,X_0^N \in L^p(\Omega,\mathscr{F},\prob)$ with $p \ge 2$, the process $\bY$ satisfies 
\begin{equation}
    \E\biggl[\sup_{t\in[0,T]} \norm{\bY_t}_{(\real^d \times F(U))^N}^p\biggr] \le C\Bigl(1+\E\bigl[\norm{\bX_0}_{(\real^d)^N}^p\bigr]\Bigr), 
\end{equation}
for some constant $C>0$ depending on $p$, $N$, $T$, $M_v$\,, and $M_\sigma$ (the last two constants are introduced in the statement of Proposition~\ref{prop:DiscreteEmpLipSub-v-sigma-T}, see estimates \eqref{eq:DiscreteEmpSub} below).
\end{thm}

\begin{proof}
The proof of the theorem is articulated in various steps, which are addressed in the rest of this section.
In particular, the first step is to recast the dynamics~\eqref{eq:NDicsreteProblem} in the equivalent form \eqref{eq:DicsreteProblem} below, thanks to the introduction of suitable vector fields $\bv$, $\bsigma$, and $\bcT$ (see \eqref{eq_boldfields}).
In Section~\ref{sec:structuralbold}, we prove some structural properties of these fields (see Propositions~\ref{prop:DiscreteLip} and~\ref{prop:DiscreteEmpLipSub-v-sigma-T}), as well as of the field $\bcG$ defined in~\eqref{eq:G_Def} below (see Proposition~\ref{prop:DiscreteEmpLipSub-G}).

In Section \eqref{sec_existenceN}, we prove that the $N$-particle problem in \eqref{eq:NDicsreteProblem} admits a strong solution according to Definition~\ref{def:StrongSolutionAndUniqueness}: to do so, we rely on Picard iterations to solve \eqref{eq:DicsreteProblem}, whose convergence is a consequence of the estimates proved in Proposition~\ref{prop:QuattroStime} below. 
The proofs that the Picard iterations and the solution to \eqref{eq:NDicsreteProblem} are adapted to the filtration $(\sF_t)_{t\in[0,T]}$ (see Propositions~\ref{prop:DiscreteIterationsContAdapt} and~\ref{prop:SolContAdatt}, respectively) are postponed to Section~\ref{app_adapted}.
In Section~\eqref{sec_aprioriestimateN}, we prove an \emph{a priori} estimate on the $p$-th moments of solutions to \eqref{eq:NDicsreteProblem}.
Finally, in Section~\ref{sec_pathwiseuniquenessN}, we prove pathwise uniqueness.
\end{proof}

It is convenient to rewrite problem (\ref{eq:NDicsreteProblem}) in a more compact and synthetic form, by gathering all spatial variables together in the process $\bX \colon \Omega \longrightarrow \cC\bigl([0,T],(\real^d)^N\bigr)$ defined by $\bX_t \coloneqq (X^1_t,\dots,X^N_t)$, and the mixed strategies in the process $\bm{\Lambda} \colon \Omega \longrightarrow \cC\bigl([0,T],(\cP(U))^N\bigr)$ defined by $\bLambda_t \coloneqq (\Lambda^1_t,\dots,\Lambda^N_t)$, for all $t \in [0,T]$.
We also introduce the random variables $\bX_0 \coloneqq (X_0^1,\dots,X_0^N) \in L^2(\Omega,\mathscr{F},\prob)$ and $\bLambda_0 \coloneqq (\Lambda_0^1,\dots,\Lambda_0^N)$, taking values in $(\real^d)^N$ and $(\cP(U))^N$, respectively.

We now introduce the fields $\bv \colon \cP_1(\real^d \times \cP(U)) \times (\real^d)^N \times (\cP(U))^N \longrightarrow (\real^d)^N$, $\bsigma \colon \cP_1(\real^d \times \cP(U)) \times (\real^d)^N \times (\cP(U))^N \longrightarrow \real^{dN\times mN}$, and $\bcT \colon \cP_1(\real^d \times \cP(U)) \times (\real^d)^N \times (\cP(U))^N \longrightarrow (F(U))^N$ defined as
\begin{subequations}\label{eq_boldfields}
\begin{align}
\bv\colon(\Sigma,x_1,\dots,x_N,\lambda_1,\dots,\lambda_N) & \longmapsto 
\begin{pmatrix}
    v_\Sigma(x_1,\lambda_1) \\ \vdots \\ v_\Sigma(x_N,\lambda_N)
\end{pmatrix}, \label{eq:b_vDef}\\[2mm]
\bsigma\colon(\Sigma,x_1,\dots,x_N,\lambda_1,\dots,\lambda_N) & \longmapsto
\diag
\begin{pmatrix}
\sigma_\Sigma(x_1,\lambda_1),\dots,\sigma_\Sigma(x_N,\lambda_N)
\end{pmatrix}, \label{eq:b_sigmaDef}\\[2mm]
\bcT\colon(\Sigma,x_1,\dots,x_N,\lambda_1,\dots,\lambda_N) & \longmapsto 
\begin{pmatrix}
    \cT_\Sigma(x_1,\lambda_1) \\ \vdots \\ \cT_\Sigma(x_N,\lambda_N)
\end{pmatrix}, \label{eq:b_lambdaDef}
\end{align}
\end{subequations}
respectively, so that we can rewrite problem (\ref{eq:NDicsreteProblem}) in integral form as
\begin{equation}\label{eq:DicsreteProblem}
\begin{dcases}
\bX_t = \bX_0 + \int_0^t \bv_{\Sigma^N_s}(\bX_s,\bLambda_s)\,\de s + \int_0^t \bsigma_{\Sigma^N_s}(\bX_s,\bLambda_s)\,\de\bB_s\,, \\ 
\bLambda_t = \bm{\Lambda}_0 + \int_0^t \bcT_{\Sigma^N_s}(\bX_s,\bLambda_s)\,\de s\,,
\end{dcases}
\end{equation}
where $\bB_t \coloneqq (B^1_t,\ldots,B^N_t)$, for $t\in[0,T]$. 
The more concise formulation \eqref{eq:DicsreteProblem} will be convenient for the analytical arguments that follow, which require a different treatment of the spatial and mixed-strategy components due to their different nature. 


\subsection{\texorpdfstring{Structural properties of the fields $\bv$, $\bsigma$, and $\bcT$}{Structural properties of the boldface fields}}
\label{sec:structuralbold}
Before starting with the proof of Theorem \ref{thm:WellPosednessNParticleSystem}, we deduce some Lipschitz continuity and sublinearity estimates for the fields $\bv$, $\bsigma$, and $\bcT$ introduced in \eqref{eq_boldfields}.

\begin{prop}\label{prop:DiscreteLip}
Let $(\Sigma_1,\bx_1,\blambda_1),(\Sigma_2,\bx_2,\blambda_2) \in \cP_1(\real^d \times \cP(U))\times (\real^d)^N\times (\cP(U))^N$. 
Then the following estimates hold
\begin{subequations}
\begin{align}
\label{eq:DiscreteLip-v1}
&\begin{split}
&\, \norm{\bv_{\Sigma_1}(\bx_1,\blambda_1) - \bm{v}_{\Sigma_2}(\bx_2,\blambda_2)}_{(\real^d)^N}\\
\le &\, L_v\bigl(\norm{\bx_1 - \bx_2}_{(\real^d)^N} + \norm{\blambda_1 - \blambda_2}_{(F(U))^N} + N W_1(\Sigma_1, \Sigma_2) \bigr),
\end{split} \\
\label{eq:DiscreteLip-sigma1} 
&\begin{split}
&\, \norm{\bsigma_{\Sigma_1}(\bx_1,\blambda_1) - \bsigma_{\Sigma_2}(\bx_2,\blambda_2)}_{\real^{dN\times mN}} \\
\le&\, L_\sigma \bigl(\norm{\bx_1 - \bx_2}_{(\real^d)^N} + \norm{\blambda_1 - \blambda_2}_{(F(U))^N} + N W_1(\Sigma_1, \Sigma_2)\bigr),
\end{split} \\
\label{eq:DiscreteLip-T1}
&\begin{split}
&\, \norm{\bcT_{\Sigma_1}(\bx_1,\blambda_1) - \bcT_{\Sigma_2}(\bx_2,\blambda_2)}_{(F(U))^N} \\
\le &\, L_{\cT}\bigl(\norm{\bx_1 - \bx_2}_{(\real^d)^N} + \norm{\blambda_1 - \blambda_2}_{(F(U))^N} + N W_1(\Sigma_1, \Sigma_2)\bigr),
\end{split}
\end{align}
\end{subequations}
where the constants $L_v$\,, $L_\sigma$\,, and $L_\cT$ are those appearing in \eqref{eq_fieldsprop}.
\end{prop}

\begin{proof}
We only prove estimate \eqref{eq:DiscreteLip-v1}, since the others are analogous. 
By the definition of the norm in $(\real^d)^N$ and \eqref{eq:v_prop}, we have
\begin{equation*}
\begin{split}
&\,\norm{\bm{v}_{\Sigma_1}(\bx_1,\blambda_1) - \bm{v}_{\Sigma_2}(\bx_2,\blambda_2)}_{(\real^d)^N} 
= \sum_{j=1}^N \norm{v_{\Sigma_1}(x_{1,j},\lambda_{1,j}) - v_{\Sigma_2}(x_{2,j},\lambda_{2,j})}_{\real^d} \\
\le &\, L_v\sum_{j=1}^N \big(\norm{x_{1,j} - x_{2,j}}_{\real^d} +  \norm{\lambda_{1,j} - \lambda_{2,j}}_{F(U)} + W_1(\Sigma_1, \Sigma_2) \big), \\
\end{split}
\end{equation*}
which yields \eqref{eq:DiscreteLip-v1}.
\end{proof}

\begin{prop}\label{prop:DiscreteEmpLipSub-v-sigma-T}
Let $\bx_i=(x_{i,1},\ldots,x_{i,N})\in(\real^d)^N$ and $\blambda_i=(\lambda_{i,1},\ldots,\lambda_{i,N})\in(\cP(U))^N$, and let $\by_i=(x_{i,1},\lambda_{i,1},\ldots,x_{i,N},\lambda_{i,N})\eqqcolon(y_{i,1},\ldots,y_{i,N})\in(\real^d\times\cP(U))^N$, for $i=1,2$, and let 
\begin{equation}\label{eq_empirical}
\barr{\Sigma}^i \coloneqq \frac{1}{N} \sum_{j = 1}^N \delta_{y_{i,j}}\in \cP_1(\real^d \times \cP(U)),\quad\text{for $i=1,2$.}
\end{equation}
Then
\begin{subequations}\label{eq:DiscreteEmpLip}
\begin{align}
\label{eq:DiscreteEmpLip-v}
\norm{\bv_{\barr{\Sigma}^1}(\bx_1,\blambda_1) - \bv_{\barr{\Sigma}^2}(\bx_2,\blambda_2)}_{(\real^d)^N} &\le 2L_v \bigl(\norm{\bx_1 - \bx_2}_{(\real^d)^N} + \norm{\blambda_1 - \blambda_2}_{(F(U))^N}\bigr), \\
\label{eq:DiscreteEmpLip-sigma}
\norm{\bsigma_{\barr{\Sigma}^1}(\bx_1,\blambda_1) - \bsigma_{\barr{\Sigma}^2}(\bx_2,\blambda_2)}_{\real^{dN\times mN}} &\le 2L_\sigma \bigl(\norm{\bx_1 - \bx_2}_{(\real^d)^N} + \norm{\blambda_1 - \blambda_2}_{(F(U))^N}\bigr), \\
\label{eq:DiscreteEmpLip-T}
\!\!\!
\norm{\bcT_{\barr{\Sigma}^1}(\bx_1,\blambda_1) - \bcT_{\barr{\Sigma}^2}(\bx_2,\blambda_2)}_{(F(U))^N} &\le 2L_{\cT} \bigl(\norm{\bx_1 - \bx_2}_{(\real^d)^N} + \norm{\blambda_1 - \blambda_2}_{(F(U))^N}\bigr).
\end{align}
\end{subequations}
Furthermore there exist $M_v\,, M_\sigma\,, M_\cT> 0$ such that for every $\by =(x_{1},\lambda_{1},\ldots,x_{N},\lambda_{N})\in (\real^d \times \cP(U))^N$, letting $\barr{\Sigma} \coloneqq \frac{1}{N} \sum_{j = 1}^N \delta_{y_j}$\,, the following sublinearity estimates hold 
\begin{subequations}\label{eq:DiscreteEmpSub}
\begin{align}
\label{eq:DiscreteEmpSub-v}
\norm{\bm{v}_{\barr{\Sigma}}(\bx,\blambda)}_{(\real^d)^N} \le M_v \bigl(1 + \norm{\bx}_{(\real^d)^N}\bigr), \\
\label{eq:DiscreteEmpSub-sigma}
\norm{\bsigma_{\barr{\Sigma}}(\bx,\blambda)}_{\real^{dN\times mN}} \le M_\sigma\bigl(1 + \norm{\bx}_{(\real^d)^N}\bigr), \\
\label{eq:DiscreteEmpSub-T}
\norm{\bcT_{\barr{\Sigma}}(\bx,\blambda)}_{(F(U))^N} \le M_\cT \bigl(1 + \norm{\bx}_{(\real^d)^N}\bigr).
\end{align}
\end{subequations}
\end{prop}
\begin{proof} 
We start by observing that, since $\frac{1}{N} \sum_{j=1}^N \delta_{(y_{1,j},y_{2,j})} \in \Gamma\big(\barr{\Sigma}^1,\barr{\Sigma}^2\big)$, by definition of Wasserstein distance, we have the estimate
\begin{equation}
\begin{split}
W_1(\barr{\Sigma}^1,\barr{\Sigma}^2) &\le \int_{(\real^d \times \cP(U))^2} \norm{y - y'}_{\real^d \times F(U)}\, \de \Biggl(\frac{1}{N} \sum_{j=1}^N \delta_{(y_{1,j},y_{2,j})} \Biggr)(y,y') \\
&\le \frac{1}{N} \sum_{j=1}^N \norm{y_{1,j} - y_{2,j}}_{\real^d \times F(U)} = \frac{1}{N} \norm{\by_1 - \by_2}_{(\real^d \times F(U))^N}\,.
\end{split}
\end{equation}
This, in combination with \eqref{eq:DiscreteLip-v1}, gives \eqref{eq:DiscreteEmpLip-v}; the remaining estimates are proved in an analogous fashion.
Finally, by choosing a point $\by_0\in (\real^d\times\mathcal{P}(U))^N$ and letting $\barr{\Sigma}^0$ be the associated empirical measure defined as in \eqref{eq_empirical},  by means of the triangle inequality and \eqref{eq:DiscreteEmpLip-v}, we obtain
\begin{equation}
\begin{split}
\norm{\bm{v}_{\barr{\Sigma}}(\bx,\blambda)}_{(\real^d)^N} \le &\norm{\bm{v}_{\barr{\Sigma}_0}(\bx_0,\blambda_0)}_{(\real^d)^N} + \norm{\bm{v}_{\barr{\Sigma}}(\bx,\blambda) - \bm{v}_{\barr{\Sigma}_0}(\bx_0,\blambda_0)}_{(\real^d)^N} \\
\le & \norm{\bm{v}_{\barr{\Sigma}_0}(\bx_0,\blambda_0)}_{(\real^d)^N} +  2L_v(\norm{\bx - \bx_0}_{(\real^d)^N} + \norm{\blambda - \blambda_0}_{(F(U))^N}) \\
\le & 
M_v(1 + \norm{\bx}_{(\real^d)^N}),
\end{split}
\end{equation}
where we have used \eqref{eq_doppiotriangolino} for the last inequality. Estimate \eqref{eq:DiscreteEmpSub-v} is proved; the remaining estimates are proved in an analosgous fashion.
\end{proof}

Let us now introduce an auxiliary vector field that will be useful in the proof of existence. We fix $\theta>0$ such that \eqref{eq:T_prop_geometry} is satisfied, thereby we can define a vector field $\bcG \colon \cP_1(\real^d \times \cP(U)) \times (\real^d)^N \times (\cP(U))^N \longrightarrow (\cP(U))^N$ given by
\begin{equation}
(\Sigma,\bx,\blambda) \longmapsto \bcG_\Sigma(\bx,\blambda) \coloneqq \blambda + \theta \bcT_\Sigma(\bx,\blambda),
\label{eq:G_Def}
\end{equation}
which allows us to 
write
\begin{equation*}
\bcT_\Sigma(\bx,\blambda) = \frac{\bcG_\Sigma(\bx,\blambda) - \blambda}{\theta},
\end{equation*}
so as to obtain a field analogous to the one studied in \cite[Corollaire~1.1, Pag.~11]{Brezis-Operateur-Maximaux}.
The field $\bcG$ inherits from $\bcT$ its properties of Lipschitz continuity and sublinearity: in fact, as an easy corollary of Proposition \ref{prop:DiscreteEmpLipSub-v-sigma-T}, we obtain the following estimate.
\begin{prop}\label{prop:DiscreteEmpLipSub-G} 
Let $\barr{\Sigma}^1$, $\barr{\Sigma}^2$, $\barr{\Sigma}$ be the empirical measures of Proposition \ref{prop:DiscreteEmpLipSub-v-sigma-T}. 
Then there exists $L_\cG>0$, depending on $L_\cT$ and $\theta$, and there exists $M_\cG > 0$, depending on $M_\cT$ and $\theta$, such that the following inequalities hold:
\begin{subequations}\label{eq:DiscreteEmpLipSubG}
\begin{eqnarray}
\!\!\!\! \!\!\!\! \!\!\!\! \!\!\!\!
\norm{\bcG_{\barr{\Sigma}^1}(\bx_1,\blambda_1) - \bcG_{\barr{\Sigma}^2}(\bx_2,\blambda_2)}_{(F(U))^N} &\!\!\!\! \le &\!\!\!\! 2L_{\cG} \bigl(\norm{\bx_1 - \bx_2}_{(\real^d)^N} + \norm{\blambda_1 - \blambda_2}_{(F(U))^N}\bigr),\label{eq:DiscreteEmpLip-G}\\
\norm{\bcG_{\barr{\Sigma}}(\bx,\blambda)}_{(F(U))^N} &\!\!\!\! \le &\!\!\!\! M_\cG \bigl(1 + \norm{\bx}_{(\real^d)^N}\bigr).\label{eq:DiscreteEmpSub-G}
\end{eqnarray}
\end{subequations}
\end{prop}

\subsection{Existence of a solution to the $N$-particle model} \label{sec_existenceN}
We start by showing existence of a strong solution to \eqref{eq:NDicsreteProblem}. We make use of the classical \emph{method of successive approximation}: having fixed $\theta>0$ so as to satisfy condition \eqref{eq:T_prop_geometry}, we define the Picard iterations as follows:
\begin{equation}
\begin{dcases}
\bX_{n+1,t} \coloneqq \bX_0 + \int_0^t \bv_{\Sigma^N_{n,s}}(\bX_{n,s},\bLambda_{n,s})\,\de s + \int_0^t \bsigma_{\Sigma^N_{n,s}}(\bX_{n,s},\bLambda_{n,s})\,\de\bB_s\,, \\ 
\bLambda_{n+1,t} \coloneqq e^{-\frac{t}{\theta}}\bLambda_0 + \frac{1}{\theta}\int_0^t e^\frac{s-t}{\theta} \bcG_{\Sigma^N_{n,s}}(\bX_{n,s},\bLambda_{n,s})\,\de s\,,
\end{dcases}
\label{eq:DiscreteIterations}
\end{equation}
for all $t \in [0,T]$ and $n = 0,1,2,\dots$, where $\Sigma^N_{n,t} \coloneqq \frac{1}{N}\sum_{j=1}^N \delta_{(X^j_{n,t},\Lambda^j_{n,t})}$\,; 
the initial iteration is given by
\begin{equation}
\begin{dcases}
\bX_{0,t} \coloneqq \bX_0\,,\\
\bLambda_{0,t} \coloneqq \bLambda_0\,,
\end{dcases}
\qquad \text{for all $t \in [0,T]$.}
\label{eq:DiscreteIterations0}
\end{equation}

The proof of the following property (which is more probabilistic in nature) is postponed to Section~\ref{app_adapted}.
\begin{prop} \label{prop:DiscreteIterationsContAdapt}
The stochastic processes $\bX_n$ and $\bLambda_n$, defined by \eqref{eq:DiscreteIterations} and \eqref{eq:DiscreteIterations0}, are continuous and adapted to the filtration $(\sF_t)_{t \in [0,T]}$ for all $n = 0,1,2,\dots$
\end{prop}

\begin{rem}\label{rem:trajectories are confined in P(U)} We observe that for all $n = 0,1,2,\dots$, and for all $t \in [0,T]$, we have that $\bLambda_{n,t} \in (\cP(U))^N$. In fact, by hypothesis, $\bLambda_{0,t} \equiv \bLambda_0 \in (\cP(U))^N$; additionally, if $\bLambda_{n,t} \in (\cP(U))^N$ for all $t \in [0,T]$, then the same holds for $\bLambda_{n+1,t}$. Indeed, since $\bcG$ take values in $(\cP(U))^N$, we have
\begin{equation*}
\bcG_{\Sigma^N_{n,s}}(\bX_{n,s},\bLambda_{n,s}) \in (\cP(U))^N, \quad \text{for all $s \in [0,T]$,}
\end{equation*}
whence
\begin{equation*}
\frac{\displaystyle\frac{1}{\theta}\int_0^t e^{\frac{s-t}{\theta}}\bcG_{\Sigma^N_{n,s}}(\bX
_{n,s},\bLambda_{n,s})\,\de s}{\displaystyle\frac{1}{\theta}\int_0^t e^{\frac{s-t}{\theta}}\,\de s} \in (\cP(U))^N, \qquad\text{for every $t\in[0,T]$,}
\end{equation*}
by virtue of the convexity of $(\cP(U))^N$. Noticing that $\frac{1}{\theta}\int_0^t e^{\frac{s-t}{\theta}}\,\de s = 1 - e^{-\frac{t}{\theta}}$, we have that
\begin{equation*}
\frac{1}{\theta}\int_0^t e^{\frac{s-t}{\theta}} \bm{\bcG}_{\Sigma^N_{n,s}}(\bX_{n,s},\bLambda_{n,s})\,\de s \in (1 - e^{-\frac{t}{\theta}})(\cP(U))^N
\end{equation*}
and again, by convexity of $(\cP(U))^N$, 
\begin{equation*}
\bLambda_{n+1,t} = e^{-\frac{t}{\theta}}\bLambda_0 + \frac{1}{\theta}\int_0^t e^{\frac{s-t}{\theta}} \bcG_{\Sigma^N_{n,s}}(\bX_{n,s},\bLambda_{n,s})\,\de s \in (\cP(U))^N
\end{equation*}
and the statement follows by induction on $n$.

Invoking Proposition~\ref{prop:DiscreteIterationsContAdapt}, this ensures that, for all fixed $\omega \in \Omega$, the map $t \longmapsto \bLambda_{n,t}(\omega)$ describes continuous trajectories that remain confined inside the convex set $(\cP(U))^N$, where the Lipchitz continuity and sublinearity inequalities of Propositions \ref{prop:DiscreteEmpLipSub-v-sigma-T} and \ref{prop:DiscreteEmpLipSub-G} hold.
\end{rem}

Let us now deduce the fundamental estimates that will be crucial in the proof of the existence of solutions for the $N$-particle system (\ref{eq:NDicsreteProblem}).

\begin{prop}\label{prop:QuattroStime} 
Under the hypotheses of Theorem~\ref{thm:WellPosednessNParticleSystem}, the following estimates hold:
\begin{subequations}\label{eq:12QuattroStime}
\begin{align}\label{eq:disuguaglianza(a)}
\begin{split}
\!\!\!\!
\E\bigg[\sup_{u\in[0,t]} \norm{\bX_{1,u} - \bX_0}_{(\real^d)^N}^2\bigg] \le &\, \bigl(6 M_v^2 t^2 + 24 N M_\sigma^2 t \bigr)\\
&\, \cdot \Bigl(1 + \E\bigl[\norm{\bX_0}_{(\real^d)^N}^2\bigr] + \E\bigl[\norm{\bLambda_0}_{(F(U))^N}^2\bigr]\Bigr),
\end{split}\\
\begin{split}\label{eq:disuguaglianza(b)}
\!\!\!\! 
\E\bigg[\sup_{u\in[0,t]} \norm{\bLambda_{1,u} - \bLambda_0}_{(F(U))^N}^2\bigg] \le &\, \frac{2 t^2}{\theta^2} \E\bigl[\norm{\bLambda_0}_{(F(U))^N}^2\bigr]\\
&\, + \frac{6M_\cG^2 t^2}{\theta^2} \Bigl(1 + \E\bigl[\norm{\bX_0}_{(\real^d)^N}^2\bigr] + \E\bigl[\norm{\bLambda_0}_{(F(U))^N}^2\bigr]\Bigr) 
\end{split}
\end{align}
\end{subequations}
for all $t \in [0,T]$, and
\begin{subequations}\label{eq:34QuattroStime}
\begin{align}
\begin{split}\label{eq:disuguaglianza(c)}
\E\bigg[\sup_{u\in[0,t]} \norm{\bX_{n+1,u} &\,- \bX_{n,u}}_{(\real^d)^N}^2\bigg] \le \bigl(16L_v^2t + 64 N L_\sigma^2 \bigr) \\ 
&\, \cdot\E\biggl[\int_0^t \bigl(\norm{\bX_{n,s} - \bX_{n-1,s}}_{(\real^d)^N}^2+ \norm{\bLambda_{n,s} - \bLambda_{n-1,s}}_{(F(U))^N}^2\bigr)\,\de s \biggr],
\end{split}\\
\begin{split}\label{eq:disuguaglianza(d)}
\E\bigg[\sup_{u\in[0,t]} \norm{\bLambda_{n+1,u} &\,- \bLambda_{n,u}}_{(\real^d)^N}^2\bigg] \le \frac{8L_\cG^2 t}{\theta^2} \E\biggl[\int_0^t \bigl(\norm{\bX_{n,s} - \bX_{n-1,s}}_{(\real^d)^N}^2 \\
&\,\hspace{4.7cm}+ \norm{\bLambda_{n,s} - \bLambda_{n-1,s}}_{(F(U))^N}^2\bigr)\,\de s \biggr]
\end{split}
\end{align}
\end{subequations}
for all $n = 1,2,\dots$ and $t \in [0,T]$, where $L_v\,,L_\sigma\,,L_\cG$ and $M_v\,, M_\sigma\,, M_\cG$ are the constants in \eqref{eq_fieldsprop}, \eqref{eq:DiscreteEmpSub}, and \eqref{eq:DiscreteEmpLipSubG}.
\end{prop}
\begin{proof}
We start by proving \eqref{eq:disuguaglianza(a)}. 
Let $t \in [0,T]$ be fixed and $u \in [0,t]$. We have that
\begin{equation*}
\begin{split}
&\,\norm{\bX_{1,u} - \bX_0}_{(\real^d)^N}^2 = \norm*{\int_0^u \bv_{\Sigma^N_{0,s}}(\bX_0,\bLambda_0)\,\de s + \int_0^u \bsigma_{\Sigma^N_{0,s}}(\bX_0,\bLambda_0)\,\de \bB_s}_{(\real^d)^N}^2 \\
\le&\, 2\Biggl(\norm*{\int_0^u \bv_{\Sigma^N_{0,s}}(\bX_0,\bLambda_0)\,\de s}_{(\real^d)^N}^2 + \norm*{\int_0^u \bsigma_{\Sigma^N_{0,s}}(\bX_0,\bLambda_0)\,\de \bB_s}_{(\real^d)^N}^2\Biggr) \\
\le&\, 2\Biggl(\int_0^u \norm{\bv_{\Sigma^N_{0,s}}(\bX_0,\bLambda_0)}_{(\real^d)^N}\,\de s\Biggr)^2 + 2 \norm*{\int_0^u \bsigma_{\Sigma^N_{0,s}}(\bX_0,\bLambda_0)\,\de \bB_s}_{(\real^d)^N}^2 \\
\le&\, 2u\int_0^u \norm{\bv_{\Sigma^N_{0,s}}(\bX_0,\bLambda_0)}_{(\real^d)^N}^2\,\de s + 2 \norm*{\int_0^u \bsigma_{\Sigma^N_{0,s}}(\bX_0,\bLambda_0)\,\de \bB_s}_{(\real^d)^N}^2,
\end{split}
\end{equation*}
where we have applied H\"{o}lder inequality in the last line; now, since $ \Sigma^N_{0,s}$ is an empirical measure, by \eqref{eq:DiscreteEmpSub-v}, we have that, for all $s \in [0,u]$,
\begin{equation*}
\begin{split}
    \norm{\bv_{\Sigma^N_{0,s}}(\bX_0,\bLambda_0)}_{(\real^d)^N}^2 &\le\, \bigl[M_v\bigl(1 + \norm{\bX_0}_{(\real^d)^N} + \norm{\bLambda_0}_{(F(U))^N} \bigr)\bigr]^2 \\
    &\le\, 3M_v^2\bigl(1 + \norm{\bX_0}_{(\real^d)^N}^2 + \norm{\bLambda_0}_{(F(U))^N}^2 \bigr),
\end{split}
\end{equation*}
hence, for all $u \in [0,t]$,
\begin{equation*}
\norm{\bX_{1,u} - \bX_0}_{(\real^d)^N}^2 \le 6u^2M_v^2\bigl(1 + \norm{\bX_0}_{(\real^d)^N}^2 + \norm{\bLambda_0}_{(F(U))^N}^2 \bigr) + 2 \norm*{\int_0^u \!\! \bsigma_{\Sigma^N_{0,s}}(\bX_0,\bLambda_0)\,\de \bB_s}_{(\real^d)^N}^2.
\end{equation*}
By taking the supremum over the interval $[0,t]$ and by subsequently taking the expectations, we obtain 
\begin{equation*}
\begin{split}
\E\biggl[\sup_{u\in[0,t]}\norm{\bX_{1,u} - \bX_0}_{(\real^d)^N}^2\biggr] \le &\, 6 t^2 M_v^2 \Bigl(1 + \E\bigl[\norm{\bX_0}_{(\real^d)^N}^2\bigl] + \E\bigl[\norm{\bLambda_0}_{(F(U))^N}^2\bigl]\Bigl) \\ 
&\,+ 2\E \biggl[\sup_{u \in [0,t]} \norm*{\int_0^u \bsigma_{\Sigma^N_{0,s}}(\bX_0,\bLambda_0)\,\de \bB_s}_{(\real^d)^N}^2 \biggr];
\end{split}
\end{equation*}
now, by applying the Burkholder--Davis--Gundy inequality \eqref{eq_BDG} to each component of the last term, and by \eqref{eq:DiscreteEmpSub-sigma}, we get 
\begin{equation*}
\begin{split}
&\,\E \biggl[\sup_{u \in [0,t]} \norm*{\int_0^u \bsigma_{\Sigma^N_{0,s}}(\bX_0,\bLambda_0)\,\de \bB_s}_{(\real^d)^N}^2 \biggr] \le N \sum_{i=1}^N \E \biggl[\sup_{u \in [0,t]} \norm*{\int_0^u \sigma_{\Sigma^N_{0,s}}(X^i_0,\Lambda^i_0)\,\de B^i_s}_{\real^d}^2 \biggr] \\
\le&\, N \sum_{i=1}^N 4\E \biggl[\int_0^t \norm{\sigma_{\Sigma^N_{0,s}}(X^i_0,\Lambda^i_0)}_{\real^{d\times m}}^2 \,\de s\biggr] \le 4N\E \biggl[\int_0^t \norm{\bsigma_{\Sigma^N_{0,s}}(\bX_0,\bLambda_0)}_{\real^{dN\times mN}}^2 \,\de s\biggr]  \\
\le&\, 4N\E \biggl[\int_0^t \bigl[M_\sigma\bigl(1 + \norm{\bX_0}_{(\real^d)^N} + \norm{\bLambda_0}_{(F(U))^N} \bigr)\bigr]^2\,\de s \biggr] \\
\le&\, 12 N M_\sigma^2 t\Bigl(1 + \E\bigl[\norm{\bX_0}_{(\real^d)^N}^2\bigr] + \E\bigl[\norm{\bLambda_0}_{(F(U))^N}^2\bigr]\Bigr).
\end{split}
\end{equation*}
Therefore, estimate \eqref{eq:disuguaglianza(a)} follows by combining the previous estimates.

Let us now turn to the proof of \eqref{eq:disuguaglianza(b)} Let $t \in [0,T]$ be fixed and $u \in [0,t]$. The following chain of inequalities holds:
\begin{equation*}
\begin{split}
&\,\norm{\bLambda_{1,u} - \bLambda_0}_{(F(U))^N}^2 = \norm*{\bigl(e^{-\frac{u}{\theta}}-1\bigr)\bLambda_0 + \int_0^u \frac{1}{\theta} e^{\frac{s-u}{\theta}} \bcG_{\Sigma^N_{0,s}}(\bX_0,\bLambda_0)\,\de s}_{(F(U))^N}^2 \\
\le&\, 2\Biggl(\norm*{\bigl(e^{-\frac{u}{\theta}}-1\bigr)\bLambda_0}_{(F(U))^N}^2 + \norm*{\int_0^u \frac{1}{\theta} e^{\frac{s-u}{\theta}} \bcG_{\Sigma^N_{0,s}}(\bX_0,\bLambda_0)\,\de s}_{(F(U))^N}^2\Biggr) \\
\le&\, 2\abs*{e^{-\frac{u}{\theta}}-1}^2\norm{\bLambda_0}_{(F(U))^N}^2 + 2\biggl(\int_0^u \abs*{\frac{1}{\theta} e^{\frac{s-u}{\theta}}} \norm{\bcG_{\Sigma^N_{0,s}}(\bX_0,\bLambda_0)}_{(F(U))^N}\,\de s\biggr)^2,
\end{split}
\end{equation*}
where we have used the subadditivity property of Bochner integrals (see \cite[formula (A.8)]{AFMS}). 
Recalling the elementary inequalities 
\begin{equation*}
\abs*{e^{-\frac{u}{\theta}}-1}^2\le \frac{u^2}{\theta^2}, \quad \text{for all $u \in [0,t]$},
\qquad\text{and}\qquad
\abs*{\frac{1}{\theta} e^{\frac{s-u}{\theta}}} 
\le \frac{1}{\theta}, \quad\text{for all $s\in[0,u]$,}
\end{equation*}
we conclude that
\begin{equation*}
\begin{split}
\norm{\bLambda_{1,u} - \bLambda_0}_{(F(U))^N}^2 &\le \frac{2}{\theta^2}u^2\norm{\bLambda_0}_{(F(U))^N}^2 + \frac{2}{\theta^2}\Biggl(\int_0^u \norm{\bcG_{\Sigma^N_{0,s}}(\bX_0,\bLambda_0)}_{(F(U))^N}\,\de s\Biggr)^2 \\
&\le \frac{2}{\theta^2}u^2\norm{\bLambda_0}_{(F(U))^N}^2 + \frac{2}{\theta^2} u \int_0^u \norm{\bcG_{\Sigma^N_{0,s}}(\bX_0,\bLambda_0)}_{(F(U))^N}^2\,\de s.
\end{split}
\end{equation*}
Applying 
\eqref{eq:DiscreteEmpSub-G}, to the last term, we obtain that, 
for all $u \in [0,t]$,
\begin{equation*}
\norm{\bLambda_{1,u} - \bLambda_0}_{(F(U))^N}^2 \le \frac{2}{\theta^2}u^2\norm{\bLambda_0}_{(F(U))^N}^2 + \frac{6M_\cG^2}{\theta^2}u^2 (1 + \norm{\bX_0}_{(\real^d)^N}^2 + \norm{\bLambda_0}_{(F(U))^N}^2).
\end{equation*}
By taking the supremum over the interval $[0,t]$ and subsequently the expected values, estimate \eqref{eq:disuguaglianza(b)} follows.

We now tackle estimate \eqref{eq:disuguaglianza(c)}. 
Let $n \in \nat^+$, $t \in [0,T]$, and $u \in [0,t]$. 
We can estimate
\begin{equation*}
\begin{split}
\norm{\bX_{n+1,u} - \bX_{n,u}}_{(\real^d)^N}^2
\!\le&\, 2u\int_0^u \norm{\bv_{\Sigma^N_{n,s}}(\bX_{n,s},\bLambda_{n,s}) - \bv_{\Sigma^N_{n-1,s}}(\bX_{n-1,s},\bLambda_{n-1,s})}_{(\real^d)^N}^2\,\de s\\
&\,+ 2\norm*{\int_0^u \bigl(\bsigma_{\Sigma^N_{n,s}}(\bX_{n,s},\bLambda_{n,s}) - \bsigma_{\Sigma^N_{n-1,s}}(\bX_{n-1,s},\bLambda_{n-1,s})\bigr)\,\de\bB_s}_{(\real^d)^N}^2\\
\le&\, 16L_v^2u\int_0^u \bigl(\norm{\bX_{n,s} - \bX_{n-1,s}}_{(\real^d)^N}^2 + \norm{\bLambda_{n,s} - \bLambda_{n-1,s}}_{(F(U))^N}^2\bigr)\,\de s\\
&\,+ 2\norm*{\int_0^u \!\bigl(\bsigma_{\Sigma^N_{n,s}}(\bX_{n,s},\bLambda_{n,s}) - \bsigma_{\Sigma^N_{n-1,s}}(\bX_{n-1,s},\bLambda_{n-1,s})\bigr)\,\de\bB_s}_{(\real^d)^N}^2,
\end{split}
\end{equation*}
where we have used \eqref{eq:DiscreteEmpLip-v};
hence, by taking the supremum over the interval $[0,t]$ and subsequently the expected values,
by applying Burkholder--Davis--Gundy inequality \eqref{eq_BDG} and 
\eqref{eq:DiscreteEmpLip-sigma}, 
we obtain
\begin{equation*}
\begin{split}
&\E\biggl[\sup_{u\in[0,t]}\norm{\bX_{n+1,u} - \bX_{n,u}}_{(\real^d)^N}^2\biggr] \\
\le&\, \bigl(16L_v^2t + 64 N L_\sigma^2\bigr) \E\biggl[\int_0^t \bigl(\norm{\bX_{n,s} - \bX_{n-1,s}}_{(\real^d)^N}^2 + \norm{\bLambda_{n,s} - \bLambda_{n-1,s}}_{(F(U))^N}^2\bigr)\,\de s \biggr].
\end{split}
\end{equation*}
From the arbitrariness of $t \in [0,T]$ and $n \in \nat^+$, \eqref{eq:disuguaglianza(c)} follows .

We now turn to \eqref{eq:disuguaglianza(d)}. 
As before, we fix $n \in \nat^+$, $t \in [0,T]$, and $u \in [0,t]$. 
We have that
\begin{equation*}
\begin{split}
&\,\norm{\bLambda_{n+1,u} - \bLambda_{n,u}}_{(F(U))^N}^2 \\
\le&\, \biggl( \int_0^u \biggl\lvert\frac{e^{\frac{s-u}{\theta}}}{\theta}\biggr\rvert \norm{\bcG_{\Sigma^N_{n,s}}(\bX_{n,s},\bLambda_{n,s}) -  \bcG_{\Sigma^N_{n-1,s}}(\bX_{n-1,s},\bLambda_{n-1,s})}_{(F(U))^N}\,\de s \biggr)^2 \\
\le&\, \frac{u}{\theta^2}\int_0^u \norm{\bcG_{\Sigma^N_{n,s}}(\bX_{n,s},\bLambda_{n,s}) -  \bcG_{\Sigma^N_{n-1,s}}(\bX_{n-1,s},\bLambda_{n-1,s})}_{(F(U))^N}^2\,\de s \\
\le&\, \frac{8L_\cG^2}{\theta^2}u\int_0^u  \bigl(\norm{\bX_{n,s} - \bX_{n-1,s}}_{(\real^d)^N}^2 + \norm{\bLambda_{n,s} - \bLambda_{n-1,s}}_{(F(U))^N}^2\bigr)\,\de s,
\end{split}
\end{equation*}
where we have used \eqref{eq:DiscreteEmpLip-G} in the third inequality.
By taking the supremum over $[0,t]$ and the expected values of both sides of the inequality, we conclude that
\begin{equation*}
\begin{split}
&\,\E\biggl[\sup_{u\in[0,t]} \norm{\bLambda_{n+1,u} - \bLambda_{n,u}}_{(\real^d)^N}^2\biggr] \\ 
\le&\, \frac{8L_\cG^2 t}{\theta^2} \E\biggl[\int_0^t \bigl(\norm{\bX_{n,s} - \bX_{n-1,s}}_{(\real^d)^N}^2 + \norm{\bLambda_{n,s} - \bLambda_{n-1,s}}_{(F(U))^N}^2\bigr)\,\de s \biggr],
\end{split}
\end{equation*}
and from the arbitrariness of $t \in [0,T]$ and $n \in \nat^+$ we obtain \eqref{eq:disuguaglianza(d)}.
This concludes the proof.
\end{proof}

\begin{rem}
We notice that the term $\E\bigl[\norm{\bX_0}_{(\real^d)^N}^2\bigr]$ is finite due to the assumed square integrability of the spatial initial data, whereas $\E\bigl[\norm{\bLambda_0}_{(F(U))^N}^2\bigr]\le N^2$ by \eqref{eq_doppiotriangolino}. 
\end{rem}

Now, let us introduce the sequence of continuous $(\real^d \times \cP(U))^N$-valued stochastic processes $\{\bY_n\}_{n=0}^{\infty}$ defined by
\begin{equation}\label{eq:YnDef}
\bY_{n,t} \coloneqq \bigl(X^1_{n,t}\,,\Lambda^1_{n,t}\,,\dots,X^N_{n,t}\,,\Lambda^N_{n,t}\bigr) 
\end{equation}
for all $t\in [0,T]$, where the processes $\{X^i_n\}_{i=0}^N$ and $\{\Lambda^i_n\}_{i=0}^N$ are, for all $n = 0,1,2\dots$, the components of the processes $\bX_n$ and $\bLambda_n$ defined through the iterations (\ref{eq:DiscreteIterations}). In the following, we will construct a solution to the $N$-particle problem (\ref{eq:DicsreteProblem}) as a suitable limit of the processes $\{\bY_n\}_{n=0}^{\infty}$\,.

\begin{prop}\label{prop:factInequality_t} Let $\{\bY_n\}_{n=0}^{\infty}$ be the sequence of continuous $(\real^d \times \cP(U))^N$-valued stochastic processes $\{\bY_n\}_{n=0}^{\infty}$ defined by \eqref{eq:YnDef}. Then, for all $n = 0,1,2,\dots$ and for all $t \in [0,T]$, the following inequality holds
\begin{equation}\label{eq:factInequality_t}
\E\biggl[\sup_{u\in[0,t]} \norm{\bY_{n+1,u} - \bY_{n,u}}_{(\real^d \times F(U))^N}^2\biggr] \le \frac{(\cR t)^{n+1}}{(n+1)!}\,,
\end{equation}
where $\mathcal{R}>0$ depends on $\theta,N,T$, on the constants $L_{v}\,, L_{\sigma}\,, L_{\cG}$\,, $M_{v}\,, M_{\sigma}\,, M_{\cG}$\, and on the second moments of the initial data $X_0^1\,,\dots,X_0^N$\,. 
\end{prop}
\begin{proof} 
We proceed by induction on $n$. We first prove the statement for $n=0$: recalling estimates \eqref{eq:12QuattroStime}, we obtain
\begin{equation*}
\begin{split}
&\,\E\biggl[\sup_{u\in[0,t]} \norm{\bY_{1,u} - \bY_0}_{(\real^d \times F(U))^N}^2\biggr] \\
\le&\, 2 \E\biggl[\sup_{u\in[0,t]}\norm{\bX_{1,u} - \bX_0}_{(\real^d)^N}^2\biggr] + 2 \E\biggl[\sup_{u\in[0,t]}\norm{\bLambda_{1,u} - \bLambda_0}_{(F(U))^N}^2\biggr] \\
\le&\, \bigl(12 M_v^2 t^2 + 48 N M_\sigma^2 t \bigr)
\Bigl(1 + \E\bigl[\norm{\bX_0}_{(\real^d)^N}^2\bigr] + \E\bigl[\norm{\bLambda_0}_{(F(U))^N}^2\bigr]\Bigr) \\
&\,+\frac{4t^2}{\theta^2} \E\bigl[\norm{\bLambda_0}_{(F(U))^N}^2\bigr] + \frac{12M_\cG^2 t^2}{\theta^2}  \Bigl(1 + \E\bigl[\norm{\bX_0}_{(\real^d)^N}^2\bigr] + \E\bigl[\norm{\bLambda_0}_{(F(U))^N}^2\bigr]\Bigr) \\
\le&\, \biggl(12M_v^2 T + 48 N M_{\sigma}^2 + \frac{12M_{\cG}^2 T}{\theta^2}  \biggr)\Bigl(N^2+ 1 + \E\bigl[\norm{\bX_0}_{(\real^d)^N}^2\bigr] 
\Bigr)t + \frac{4N^2 T}{\theta^2} t \eqqcolon \cR t,
\end{split}
\end{equation*}
where we have bounded some occurrences of
$t$ with $T$ and used \eqref{eq_doppiotriangolino} in the last inequality; this is 
\eqref{eq:factInequality_t} for $n=0$.
Now, let $n \in \nat$\,; by \eqref{eq:34QuattroStime}, we can estimate
\begin{equation*}
\begin{split}
&\E\biggl[\sup_{u\in[0,t]} \norm{\bY_{n+2,u} - \bY_{n+1,u}}_{(\real^d \times F(U))^N}^2\biggr] \\
\le&\, 2 \E\biggl[\sup_{u\in[0,t]}\norm{\bX_{n+2,u} - \bX_{n+1,u}}_{(\real^d)^N}^2\biggr] + 2 \E\biggl[\sup_{u\in[0,t]}\norm{\bLambda_{n+2,u} - \bLambda_{n+1,u}}_{(F(U))^N}^2\biggr] \\
\le&\,\bigl(32L_v^2t + 128 N L_\sigma^2 \bigr)\E\biggl[\int_0^t \bigl(\norm{\bX_{n+1,s} - \bX_{n,s}}_{(\real^d)^N}^2+ \norm{\bLambda_{n+1,s} - \bLambda_{n,s}}_{(F(U))^N}^2\bigr)\,\de s \biggr]\\
&\,+ \frac{16L_\cG^2 t}{\theta^2} \E\biggl[\int_0^t \bigl(\norm{\bX_{n+1,s} - \bX_{n,s}}_{(\real^d)^N}^2 + \norm{\bLambda_{n+1,s} - \bLambda_{n,s}}_{(F(U))^N}^2\bigr)\,\de s \biggr] \\
\le&\, 32\biggl(2L_v^2T + 8 N L_\sigma^2+ \frac{L_\cG^2 T}{\theta^2}\biggr)\E\biggl[\int_0^t \norm{\bY_{n+1,s} - \bY_{n,s}}_{(\real^d\times\cP(U))^N}^2\,\de s \biggr] \\
\le&\, 32\biggl(2L_v^2T + 8L_\sigma^2+ \frac{L_\cG^2 T}{\theta^2}\biggr) \int_0^t\frac{(\cR s)^{n+1}}{(n+1)!} \,\de s \eqqcolon \barcR\frac{\cR^{n+1} t^{n+2}}{(n+2)!},
\end{split}
\end{equation*}
where, again, we have estimated some occurrences of~$t$ by~$T$, and we have used the induction hypothesis in the last inequality.
By possibly redefining~$\cR$ as $\max\{\barcR,\cR\}$, 
we finally obtain the instance of \eqref{eq:factInequality_t} with $n$ replaced by $n+1$.
\end{proof}
In what follows, we will construct a strong solution to problem \eqref{eq:DicsreteProblem} as the uniform limit of the processes $\{\bY_n\}_{n=0}^{\infty}$\,.

\begin{prop}\label{prop:IterateCauchy}
There exists a $\prob$-negligible set $\cZ\in\sF$ such that for all $\omega\in\cZ^c$,
$\{\bY_n(\omega)\}_{n=0}^{\infty} \subseteq \cC([0,T],(\real^d \times \cP(U))^N)$
is a Cauchy sequence in the space $(\cC([0,T],(\real^d \times F(U))^N), \norm{\cdot}_\infty)$.
Therefore, letting $\tilde{\by}$ be an arbitrary element of $(\real^d \times \cP(U))^N$, the process $\bY \colon \Omega \longrightarrow \cC([0,T],(\real^d \times \cP(U))^N)$ given by
\begin{equation}\label{eq:DefSoluzione}
\bY(\omega) \coloneqq
\begin{dcases}
\lim_{n\to \infty} \bY_{n}(\omega), &\omega \in \cZ^c \\
\tilde{\by}, &\omega \in \cZ 
\end{dcases}
\end{equation}
is well defined.
\end{prop}
\begin{proof} 
By Markov inequality and \eqref{eq:factInequality_t} with $t=T$,  
we obtain, for all $n = 0,1,2,\dots$,
\begin{equation}\label{eq:Summable?}
\begin{split}
&\,\prob\biggl(\sup_{t\in[0,T]} \norm{\bY_{n+1,t} - \bY_{n,t}}_{(\real^d \times F(U))^N} > \frac{1}{2^n}\biggr) \\
\le&\, 4^n\E\biggl[\sup_{t\in[0,T]} \norm{\bY_{n+1,t} - \bY_{n,t}}_{(\real^d \times F(U))^N}^2\biggr] \le  4^n\frac{(\mathcal{R}T)^{n+1}}{(n+1)!}\,,
\end{split}
\end{equation}
so that, by 
the Borel--Cantelli Lemma, we deduce that 
\begin{equation*}
\begin{split}
&\,\prob\Biggl(\limsup_n \biggl\{\sup_{t\in[0,T]} \norm{\bY_{n+1,t} - \bY_{n,t}}_{(\real^d \times F(U))^N} > \frac{1}{2^n}\biggr\}\Biggr)=0. 
\end{split}
\end{equation*}
Therefore, we have found a $\prob$-negligible set $\cZ \in \mathscr{F}$ such
that, for all $\omega \in \cZ^c$, it holds
\begin{equation}\label{eq_stellina}
\sup_{t\in[0,T]} \norm{\bY_{n+1,t}(\omega) - \bY_{n,t}(\omega)}_{(\real^d \times F(U))^N} \le \frac{1}{2^n},\quad \text{eventually.}
\end{equation}
Now we show that, for all $\omega \in \cZ^c$, the sequence of continuous functions $\{\bY_n(\omega)\}_{n=0}^{\infty}$ is a Cauchy sequence in the Banach space $(\cC([0,T],(\real^d \times F(U))^N), \norm{\cdot}_\infty)$.
Let us begin by presenting 
the function $\bY_n(\omega)$ as a telescoping sum
\begin{equation*}
\bY_n(\omega) = \bY_0(\omega) + \sum_{k = 0}^{n-1} \bigl(\bY_{k+1}(\omega) - \bY_{k}(\omega)\bigr), \qquad\text{for all $n=0,1,2,\ldots$,}
\end{equation*}
so that, for 
$m > n$, we can write 
\begin{equation*}
\bY_{m,t}(\omega) - \bY_{n,t}(\omega) = \sum_{k = n}^{m-1} \bigl(\bY_{k+1,t}(\omega) - \bY_{k,t}(\omega)\bigr)\qquad\text{for all $t \in [0,T]$.}
\end{equation*}
Therefore, for $n$ large enough, by \eqref{eq_stellina},
\begin{equation}
\label{eq:CauchyCondition}
\begin{split}
&\,\sup_{t \in [0,T]} \norm{\bY_{m,t}(\omega) - \bY_{n,t}(\omega)}_{(\real^d \times F(U))^N} \\
\le&\, \sum_{k = n}^{m-1} \biggl(\sup_{t \in [0,T]}\norm{\bY_{k+1,t}(\omega) - \bY_{k,t}(\omega)}_{(\real^d \times F(U))^N}\biggr) \le \sum_{k = n}^{m-1} \frac{1}{2^k} \le \sum_{k = n}^{\infty} \frac{1}{2^k}\,,
\end{split}
\end{equation}
which vanishes as $n\to\infty$.
Now, letting $\tilde{\by}$ be an arbitrary element of $(\real^d \times \cP(U))^N$, \eqref{eq:DefSoluzione} is well defined and $\bY_t(\omega)\in (\real^d\times\cP(U))^N$ owing to the closedness of $\cP(U)$ in $F(U)$.
Continuity of the trajectories follows from uniform convergence.
This concludes the proof.
\end{proof}
%

The proof of the following property, which is more probabilistic in nature, is postponed to Section~\ref{app_adapted}.
\begin{prop}\label{prop:SolContAdatt} 
The limiting stochastic process $\bY$ defined in Proposition~\ref{prop:IterateCauchy} is adapted to the filtration $(\mathscr{F}_t)_{t \in [0,T]}$\,.
\end{prop}


In the following, we will show that $\bY$ is a strong solution to the $N$-particle problem (\ref{eq:NDicsreteProblem}). The key idea is to pass to the limit as $n\to\infty$ in the definition of the iterations $\eqref{eq:DiscreteIterations}$. In line with the previous sections, we write $\bY$ as
\begin{equation*}
\bY_t = (X^1_t,\Lambda^1_t,\dots,X^N_t,\Lambda^N_t), \quad \text{for all $t \in [0,T]$,}
\end{equation*}
and we introduce the $(\real^d)^N$- and $(\cP(U))^N$-valued processes $\bX$ and $\bLambda$ defined as
\begin{equation*}
\bX_t \coloneqq (X^1_t,\dots,X^N_t), \quad \bLambda_t \coloneqq (\Lambda^1_t,\dots,\Lambda^N_t), \quad \text{for all $t \in [0,T]$;}
\end{equation*}
we also introduce the associated empirical measures
\begin{equation}
    \Sigma^N \coloneqq \frac{1}{N}\sum_{j=1}^N \delta_{(X^j,\Lambda^j)} \qquad\text{and}\qquad     \Sigma^N_t \coloneqq (\ev_t)_\sharp\Sigma^N = \frac{1}{N}\sum_{j=1}^N \delta_{(X^j_t,\Lambda^j_t)}.
\end{equation}
By construction, we have that, for all $\omega\in \cZ^c$, the sequences of continuous functions $\{\bX_n(\omega)\}_{n=0}^{\infty}$ and $\{\bLambda_n(\omega)\}_{n=0}^{\infty}$ uniformly converge to $\bX(\omega)$ and $\bLambda(\omega)$, respectively.
We first observe that, as a consequence of inequalities \eqref{eq:DiscreteEmpLip-v}, \eqref{eq:DiscreteEmpLip-sigma} and \eqref{eq:DiscreteEmpLip-G}, we have that
\begin{align*}
\sup_{t \in [0,T]}\norm{\bv_{\Sigma^N_{n,t}}(\bX_{n,t},\bLambda_{n,t}) - \bv_{\Sigma^N_t}(\bX_t,\bLambda_t)}_{(\real^d)^N} \le&\, 2L_v\sup_{t \in [0,T]}\norm{\bY_{n,t} - \bY_t}_{(\real^d\times F(U))^N}\,, \\
\sup_{t \in [0,T]}\norm{\bsigma_{\Sigma^N_{n,t}}(\bX_{n,t},\bLambda_{n,t}) - \bsigma_{\Sigma^N_t}(\bX_t,\bLambda_t)}_{\real^{dN\times mN}} \le&\, 2L_\sigma\sup_{t \in [0,T]}\norm{\bY_{n,t} - \bY_t}_{(\real^d\times F(U))^N}\,, \\
\sup_{t \in [0,T]}\norm{\bcG_{\Sigma^N_{n,t}}(\bX_{n,t},\bLambda_{n,t}) - \bcG_{\Sigma^N_t}(\bX_t,\bLambda_t)}_{(F(U))^N} \le&\, 2L_\cG\sup_{t \in [0,T]}\norm{\bY_{n,t} - \bY_t}_{(\real^d\times F(U))^N}\,.
\end{align*}
By Proposition~\ref{prop:IterateCauchy} (see \eqref{eq:DefSoluzione}), for every $\omega\in\cZ^c$ 
we deduce that the sequences of continuous functions
\begin{equation*}
\begin{aligned}
&[0,T]\ni t \longmapsto \bv_{\Sigma^N_{n,t}}(\bX_{n,t}(\omega),\bLambda_{n,t}(\omega))\in(\real^d)^N \\
&[0,T]\ni t \longmapsto \bcG_{\Sigma^N_{n,t}}(\bX_{n,t}(\omega),\bLambda_{n,t}(\omega))\in(\cP(U))^N
\end{aligned}\qquad n = 0,1,2.\dots
\end{equation*}
uniformly converge to
\begin{equation*}
[0,T] \ni t \longmapsto \bv_{\Sigma^N_t}(\bX_t(\omega),\bLambda_t(\omega))
\quad\text{and}\quad
[0,T]\ni t \longmapsto \bcG_{\Sigma^N_t}(\bX_t(\omega),\bLambda_t(\omega)),
\end{equation*}
respectively, and that the latter are in turn continuous. Therefore we deduce that on $\cZ^c$

\begin{equation*}
\begin{aligned}
\lim_{n\to\infty} \int_0^t \bv_{\Sigma^N_{n,s}}(\bX_{n,s},\bLambda_{n,s})\,\de s &= \int_0^t \bv_{\Sigma^N_s}(\bX_s,\bLambda_s)\,\de s\,, \\
\lim_{n\to\infty}\frac{1}{\theta}\int_0^t e^{\frac{s-t}{\theta}} \bcG_{\Sigma^N_{n,s}}(\bX_{n,s},\bLambda_{n,s})\,\de s &= \frac{1}{\theta}\int_0^t e^{\frac{s-t}{\theta}} \bcG_{\Sigma^N_s}(\bX_s,\bLambda_s)\,\de s\,.
\end{aligned}
\end{equation*}
for all $t \in [0,T]$.
Furthermore, we have that on $\cZ^c$ the sequence of continuous functions
\begin{equation*}
[0,T]\ni t \longmapsto \bsigma_{\Sigma^N_{n,t}}(\bX_{n,t}(\omega),\bLambda_{n,t}(\omega))\in \real^{dN\times mN}, \quad n = 0,1,2,\dots
\end{equation*}
uniformly converges to the function
$[0,T]\ni t \longmapsto \bsigma_{\Sigma^N_t}(\bX_t(\omega),\bLambda_t(\omega))$;
therefore 
\begin{equation*}
    \lim_{n\to\infty}\int_0^t\norm{\bsigma_{\Sigma^N_{n,s}}(\bX_{n,s},\bLambda_{n,s}) - \bsigma_{\Sigma^N_s}(\bX_s,\bLambda_s)}^2_{\real^{dN\times mN}}\,\de s = 0
\end{equation*}
$\prob$-almost surely. This implies (see \cite[Proposition~7.3]{Stochastic-Calculus}) the convergence of the random variables
\begin{equation*}
    \int_0^t\bsigma_{\Sigma^N_{n,s}}(\bX_{n,s},\bLambda_{n,s})\,\de\bB_s \overset{\prob}{\longrightarrow}\int_0^t\bsigma_{\Sigma^N_s}(\bX_s,\bLambda_s)\,\de\bB_s \qquad\text{for all $t \in [0,T]$.}
\end{equation*}
Hence, taking the limit in $\prob$-probability as $n \to \infty$, we have that, for all $t \in [0,T]$, 
$$
\bX_{n,t} \overset{\prob}{\longrightarrow} \bX_t\,, \qquad \bLambda_{n,t} \overset{\prob}{\longrightarrow} \bLambda_t\,,
$$
\begin{equation*}
\begin{split}
& \bX_0 + \int_0^t \bv_{\Sigma^N_{n,s}}(\bX_{n,s},\bLambda_{n,s})\,\de s + \int_0^t \bsigma_{\Sigma^N_{n,s}}(\bX_{n,s},\bLambda_{n,s})\,\de\bB_s \\
&\overset{\prob}{\longrightarrow}\, \bX_0 + \int_0^t \bv_{\Sigma^N_s}(\bX_s,\bLambda_s)\,\de s + \int_0^t \bsigma_{\Sigma^N_s}(\bX_s,\bLambda_s)\,\de\bB_s\,,
\end{split}
\end{equation*}
and
$$e^{-\frac{t}{\theta}}\bLambda_0 + \frac{1}{\theta}\int_0^t e^{\frac{s-t}{\theta}} \bcG_{\Sigma^N_{n,s}}(\bX_{n,s},\bLambda_{n,s})\,\de s \overset{\prob}{\longrightarrow} e^{-\frac{t}{\theta}}\bLambda_0 + \frac{1}{\theta}\int_0^t e^{\frac{s-t}{\theta}} \bcG_{\Sigma^N_s}(\bX_s,\bLambda_s)\,\de s\,.$$
In particular, we have that for all $t \in [0,T]\cap\Q$ there exists a $\prob$-negligible set $\cZ_t \in \sF$ such that, on $(\cZ_t\cup\cZ)^c$, by the uniqueness almost-everywhere of the limit in probability, we have that
\begin{equation}\label{eq:EqualityOnRationals}
\begin{dcases}
\bX_t = \bX_0 + \int_0^t \bv_{\Sigma^N_s}(\bX_s,\bLambda_s)\,\de s + \int_0^t\bsigma_{\Sigma^N_s}(\bX_s,\bLambda_s)\,\de\bB_s, \\
\bLambda_t = e^{-\frac{t}{\theta}}\bLambda_0 + \frac{1}{\theta}\int_0^t e^{\frac{s-t}{\theta}} \bcG_{\Sigma^N_s}(\bX_s,\bLambda_s)\,\de s.
\end{dcases}
\end{equation}
Now, defining the $\prob$-negligible set $\cN \coloneqq\bigr(\bigcup_{t\in[0,T]\cap\Q} \cZ_t\bigr)\cup\cZ$, we deduce that on $\cN^c$, equality \eqref{eq:EqualityOnRationals} holds for all $t\in[0,T]\cap\Q$. Since the trajectories of all processes are continuous, we can conclude that, by density, \eqref{eq:EqualityOnRationals} holds for all $t \in [0,T]$. 

We note that equality \eqref{eq:EqualityOnRationals} does not immediately imply that $\bY$ is a solution to problem \eqref{eq:DicsreteProblem}. To conclude the proof of existence we need the following result.
\begin{prop}\label{prop:SolDerivate}
The stochastic process $\bY$ defined in Proposition~\ref{prop:IterateCauchy} satisfies, $\prob$-a.s.,
\begin{equation}
\begin{dcases}
\bX_t = \bX_0 + \int_0^t \bv_{\Sigma^N_s}(\bX_s,\bLambda_s)\,\de s + \int_0^t\bsigma_{\Sigma^N_s}(\bX_s,\bLambda_s)\,\de\bB_s, \\
\bLambda_t = \bLambda_0 + \int_0^t \bcT_{\Sigma^N_s}(\bX_s,\bLambda_s)\,\de s,
\end{dcases}\qquad\text{for all $t \in [0,T]$.}
\end{equation}
\end{prop}
\begin{proof}
Let us fix $\omega \in \cN^c$. We define, for all $t \in [0,T]$, the following functions:
\begin{equation*}
\begin{aligned}
&\bx(t) \coloneqq \bX_t(\omega) = \bX_0(\omega) + \int_0^t \bv_{\Sigma^N_s}(\bX_s,\bLambda_s)(\omega)\,\de s + \int_0^t\bsigma_{\Sigma^N_s}(\bX_s,\bLambda_s)\,\de\bB_s(\omega)  \in (\real^d)^N, \\
&\blambda(t) \coloneqq  \bLambda_t(\omega) = e^{-\frac{t}{\theta}}\bLambda_0(\omega) + \frac{1}{\theta}\int_0^t e^{\frac{s-t}{\theta}} \bcG_{\Sigma^N_s}(\bX_s,\bLambda_s)(\omega)\,\de s \in (\cP(U))^N,\\
&\tilde{\blambda}(t) \coloneqq  \bLambda_0(\omega) + \int_0^t \bcT_{\Sigma^N_s}(\bX_s,\bLambda_s)(\omega)\,\de s \in (\cP(U))^N.
\end{aligned}
\end{equation*}
Our goal is to show that $\blambda(t) = \tilde{\blambda}(t)$ for all $t \in [0,T]$. We first note that $\blambda$ e $\tilde{\blambda}$ are of class $C^1$ with respect to the $\norm{\cdot}_{(F(U))^N}$ norm on the interval $[0,T]$, while $\bx$ is of class $C^0$. By defining the function
$\bh \colon [0,T] \longrightarrow (F(U))^N$ as $t \longmapsto \bh(t) \coloneqq \blambda(t) - \tilde{\blambda}(t)$, we can readily show that $\bh(t) = 0$, for all $t \in [0,T]$.
In fact, 
we have 
\begin{equation*}
\begin{split}
\blambda'(t) = & \frac{\de}{\de t}\biggl(e^{-\frac{t}{\theta}}\blambda(0) + \frac{1}{\theta}e^{-\frac{t}{\theta}}\int_0^t e^{\frac{s}{\theta}} \bcG_{\Sigma^N_s}(\bx(s),\blambda(s))\,\de s\biggr)\\
=& -\frac{1}{\theta}e^{-\frac{t}{\theta}}\blambda(0) - \frac{1}{\theta^2}e^{-\frac{t}{\theta}}\int_0^t e^{\frac{s}{\theta}} \bcG_{\Sigma^N_s}(\bx(s),\blambda(s))\,\de s + \frac{1}{\theta} \bcG_{\Sigma^N_t}(\bx(t),\blambda(t)) \\
=& -\frac{1}{\theta}\blambda(t) + \frac{1}{\theta}\Bigl(\blambda(t) + \theta \bcT_{\Sigma^N_t}(\bx(t),\blambda(t))\Bigr) = \bcT_{\Sigma^N_t}(\bx(t),\blambda(t))
\end{split}
\end{equation*}
(the last line follows from the definition \eqref{eq:G_Def} of $\bcG$), 
whence $\bh'(t) = \blambda'(t) - \tilde{\blambda}'(t) = \bm{0}$, for all $t \in [0,T]$.
Therefore we conclude that, for all $t \in [0,T]$, we have $\bh(t) = \bm{c} $, for a certain $\bm{c} \in (F(U))^N$. Since $\bh(0) = \blambda(0) - \tilde{\blambda}(0) = \bm{0}$, we have that  $\bm{c} = \bm{0}$ for all $t \in [0,T]$, which concludes the proof.
\end{proof}
The existence of a strong solution to the $N$-particle problem \eqref{eq:DicsreteProblem} is therefore proved.


\subsection{An \emph{a priori} estimate on the moments of the solutions}\label{sec_aprioriestimateN}
In the proof of uniqueness we will employ an auxiliary result, consisting in an a priori estimate on the moments of the solutions to problem \eqref{eq:NDicsreteProblem}, which is of independent interest. 

\begin{prop}\label{prop:stima a priori}
Let $\bY$ be a strong solution to problem \eqref{eq:NDicsreteProblem}, and let the spatial initial data $\bX_0$ be a random variable in $L^p(\Omega,\mathscr{F},\prob)$ with $p \ge 2$.
Then the following estimate holds
\begin{equation}\label{eq:APrioriEstimate}
\E\biggl[\sup_{t\in[0,T]} \norm{\bY_t}_{(\real^d \times F(U))^N}^p\biggr] \le C\Bigl(1 + \E\bigl[\norm{\bX_0}_{(\real^d)^N}^p\bigr]\Bigr),
\end{equation}
where $C>0$ depends on $p,N,M_v\,,M_\sigma\,,T$.
\end{prop}
\begin{proof}
Our strategy consists in proving estimate \eqref{eq:APrioriEstimate} for a suitable \emph{stopped} process, constructed from $\bY$, and then to extend it to $\bY$ through an approximation procedure. 
Notice that, by \eqref{eq_doppiotriangolino}, we only need to prove the finiteness of the $p$-th moments of the spatial components $\bX$.

Let us fix a constant $R > 0$ and let us consider, for all $\omega \in \Omega$, the set of times
\begin{equation*}
\mathscr{T}_R(\omega) \coloneqq \{t \in [0,T] \colon \norm{\bX_t(\omega)}_{(\real^d)^N} \ge R\},
\end{equation*}
through which we define the random time
\begin{equation*}
\tau_R(\omega) \coloneqq
\begin{dcases}
\inf\mathscr{T}_R(\omega) &\text{if $\mathscr{T}_R(\omega) \neq \emptyset$,} \\
T &\text{if $\mathscr{T}_R(\omega) = \emptyset$,}
\end{dcases}
\end{equation*}
representing the exit time of the process $\bX$ from the open ball of radius $R$ in $(\real^d)^N$. 
The continuity of the process $\bX$ guarantees that $\tau_R$ is a stopping time \cite[Proposition~3.7]{Stochastic-Calculus}; in particular, it is a finite stopping time taking values in the range $[0,T]$. 
Through the random time $\tau_R$ we can define the \emph{stopped} stochastic processes $\bX_R$ and $\bLambda_R$ given by, for all $t \in [0,T]$,
\begin{equation}
\bX^R_t \coloneqq \bX_{t \wedge \tau_R}, \qquad \bLambda^R_t \coloneqq \bLambda_{t \wedge \tau_R}\,.
\end{equation}
The continuity of $\bX$ and $\bLambda$ ensures that $\bX^R$ e $\bLambda^R$ are stochastic processes adapted to the filtration $(\sF_t)_{t\in[0,T]}$\,, see \cite[Proposition~3.6]{Stochastic-Calculus}]. 
We observe that the process $\bX^R$ coincides with $\bX$ on the interval $[0,\tau_R]$, whereas on $(\tau_r,T]$ it takes the constant value $\bX(\tau_R)$; therefore, for all $t \in [0,T]$, we have
\begin{equation*}
\begin{split}
\bX^R_t =&\, \bX_0 + \int_0^{t \wedge \tau_R} \bv_{\Sigma^N_s}(\bX_s,\bLambda_s)\,\de s + \int_0^{t \wedge \tau_R}\bsigma_{\Sigma^N_s}(\bX_s,\bLambda_s)\,\de\bB_s\\
=&\, \bX_0 + \int_0^t \bv_{\Sigma^N_s}(\bX_s,\bLambda_s) \ind_{\{s<\tau_R\}}\,\de s + \int_0^t\bsigma_{\Sigma^N_s}(\bX_s,\bLambda_s)\ind_{\{s<\tau_R\}}\, \de\bB_s \\
=&\, \bX_0 + \int_0^t \bv_{\Sigma^{N,R}_s}(\bX^R_s,\bLambda^R_s) \ind_{\{s<\tau_R\}}\,\de s + \int_0^t\bsigma_{\Sigma^{N,R}_s}(\bX^R_s,\bLambda^R_s)\ind_{\{s<\tau_R\}}\, \de\bB_s, 
\end{split}
\end{equation*}
where
\begin{equation*}
    \Sigma^{N,R}_t \coloneqq \frac{1}{N}\sum_{j=1}^N\delta_{(X^j_{t\wedge\tau_R},\Lambda^j_{t\wedge\tau_R})}
\end{equation*}
is the empirical measure associated with the stopped processes.
Now, if $t \in [0,T]$, 
we have
\begin{equation*}
\begin{split}
\E\biggl[\sup_{u\in[0,t]}\norm{\bX^R_u}_{(\real^d)^N}^p\biggr] \le&\, 3^{p-1}\biggl(\E\bigl[\norm{\bX_0}_{(\real^d)^N}^p\bigr] + \E\biggl[t^{p-1}\int_0^t \norm{\bv_{\Sigma^{N,R}_s}(\bX^R_s,\bLambda^R_s)}^p_{(\real^d)^N}\ind_{\{s<\tau_R\}}\,\de s\biggr] \\ &\,+\E\biggl[\sup_{u\in[0,t]} \norm*{\int_0^u\bsigma_{\Sigma^{N,R}_s}(\bX^R_s,\bLambda^R_s)\ind_{\{s<\tau_R\}}\,\de\bB_s}_{(\real^d)^N}^p\biggl]\biggr) \\
\eqqcolon&\, 3^{p-1}\bigl(\E\bigl[\norm{\bX_0}_{(\real^d)^N}^p\bigr] + \mathrm{I}_t + \mathrm{II}_t \bigr).
\end{split}
\end{equation*}
We now estimate the two terms $\mathrm{I}_t$ and $\mathrm{II}_t$ separately. 
By \eqref{eq:DiscreteEmpSub-v}, we have
\begin{equation*}
\begin{split}
\mathrm{I}_t &\le t^{p-1}\E\biggl[\int_0^t \bigl[M_v\bigl(1 + \norm{\bX^R_s}_{(\real^d)^N} + \norm{\bLambda^R_s}_{(F(U))^N}\bigr)\bigr]^p\ind_{\{s<\tau_R\}}\,\de s\biggr] \\
&\le T^{p-1}3^{p-1}M_v^p \int_0^t \bigl(1 + \E\bigl[\norm{\bX^R_s}_{(\real^d)^N}^p\bigr] + \E\bigl[\norm{\bLambda^R_s}_{(F(U))^N}^p\bigr]\bigr)\,\de s. 
\end{split}
\end{equation*}
Invoking \eqref{eq_BDG} and \eqref{eq:DiscreteEmpSub-sigma}, we have that 
\begin{equation*}
\begin{split}
\mathrm{II}_t &\le c_p N^{p-1} t^\frac{p-2}{2}\E\biggl[\int_0^t \norm{\bsigma_{\Sigma^{N,R}_s}(\bX^R_s,\bLambda^R_s)}_{\real^{dN\times mN}}^p \ind_{\{s<\tau_R\}}\,\de s\biggr] \\
&\le c_p N^{p-1} t^\frac{p-2}{2}\E\biggl[\int_0^t \bigl[M_\sigma\bigl(1 + \norm{\bX^R_s}_{(\real^d)^N} + \norm{\bLambda^R_s}_{(F(U))^N}\bigr)\bigr]^p\,\de s\biggr] \\
&\le c_p N^{p-1} T^\frac{p-2}{2} 3^{p-1}M_\sigma^p \int_0^t \bigl(1 + \E\bigl[\norm{\bX^R_s}_{(\real^d)^N}^p\bigr] + \E\bigl[\norm{\bLambda^R_s}_{(F(U))^N}^p\bigr]\bigr)\,\de s. 
\end{split}
\end{equation*}

Therefore, for all $t\in[0,T]$, recalling \eqref{eq_doppiotriangolino}, we have
\begin{equation*}
\begin{split}
\E\biggl[\sup_{u\in[0,t]}\norm{\bX^R_u}_{(\real^d)^N}^p\biggr] \le &\, 3^{p-1}\E\bigl[\norm{\bX_0}_{(\real^d)^N}^p\bigr] +3^{2(p-1)}\Bigl(T^{p-1}M_v^p + c_p N^{p-1}T^\frac{p-2}{2}M_\sigma^p\Bigr)\cdot \\
&\, \cdot\int_0^t \bigl(1 + N^p+ \E\bigl[\norm{\bX^R_s}_{(\real^d)^N}^p\bigr] \bigr)\,\de s \\
\leq&\, C_1\Bigl(1 + \E\bigl[\norm{\bX_0}_{(\real^d)^N}^p\bigr]\Bigr) + C_2\int_0^t \E\biggl[\sup_{u\in[0,s]}\norm{\bX^R_u}_{(\real^d)^N}^p\biggr]\,\de s,
\end{split}
\end{equation*}
where $C_1$ and $C_2$ are suitable positive constants depending on $p,N,M_v\,,M_\sigma\,,T$, but independent of the parameter $R$.

We can now apply Gr\"{o}nwall inequality, once we prove the boundedness of
$$[0,T]\ni t \longmapsto v(t) \coloneqq \E\biggl[\sup_{u\in[0,t]}\norm{\bX^R_u}_{(\real^d)^N}^p\biggr].$$
Since, $\prob$-almost surely,
\begin{equation*}
\norm{\bX^R_t}_{(\real^d)^N} \le \norm{\bX_0}_{(\real^d)^N} \vee R, \text{ for all } t \in [0,T],
\end{equation*}
we deduce that, for all $t\in[0,T]$,
\begin{equation*}
v(t) \le \E\bigl[\norm{\bX_0}_{(\real^d)^N}^p\bigr] \vee R^p \le \E\bigl[\norm{\bX_0}_{(\real^d)^N}^p\bigr] + R^p
\end{equation*}
thus implying the boundedness of $v$ over $[0,T]$, as desired. 
Consequently, we can conclude that, for all $t \in [0,T]$,
\begin{equation*}
v(t) \le C_1\bigl(1 + \E\bigl[\norm{\bX_0}_{(\real^d)^N}^p\bigr]\bigr)e^{\int_0^t C_2\,\de s}.
\end{equation*}
By substituting $t = T$ in the last inequality we obtain
\begin{equation*}
v(T) \le C_1\bigl(1 + \E\bigl[\norm{\bX_0}_{(\real^d)^N}^p\bigr]\bigr)e^{C_2 T},
\end{equation*}
that is
\begin{equation}\label{eq:intermediateAPrioriR}
\E\bigg[\sup_{t\in[0,T]}\norm{\bX^R_t}_{(\real^d)^N}^p\bigg] \le C_0 \Bigl(1 + \E\bigl[\norm{\bX_0}_{(\real^d)^N}^p\bigr]\Bigr),
\end{equation}
where $C_0 \coloneqq C_1 e^{C_2 T}$ is a positive constant depending on $p,N,M_v\,,M_\sigma$\,, and $T$, but is independent of $R$. Therefore the estimate is proven for the stopped process $\bX_R$; we will now extend it to the process $\bX$ by letting $R \longrightarrow +\infty$ in \eqref{eq:intermediateAPrioriR}. 
We observe that for almost all $\omega \in \Omega$ we have the inequality
\begin{equation*}
\sup_{t\in[0,T]}\norm{\bX^R_t(\omega)}_{(\real^d)^N}^p = \sup_{t\in[0,\tau_R(\omega)]}\norm{\bX_t(\omega)}_{(\real^d)^N}^p, \text{ for all } R \in [0,+\infty),
\end{equation*}
and that the function
$
R \longmapsto \sup_{t\in[0,\tau_R(\omega)]}\norm{\bX_t(\omega)}_{(\real^d)^N}^p
$ 
is monotonically increasing. Moreover, we have that
\begin{equation*}
\lim_{R \to +\infty}\Bigl(\sup_{t\in[0,T]}\norm{\bX^R_t(\omega)}_{(\real^d)^N}^p\Bigr) 
=
\lim_{R \to +\infty}\Bigl(\sup_{t\in[0,\tau_R(\omega)]}\norm{\bX_t(\omega)}_{(\real^d)^N}^p\Bigl) 
=
\sup_{t\in[0,T]}\norm{\bX_t(\omega)}_{(\real^d)^N}^p\,.
\end{equation*}
We can therefore apply Beppo Levi's Theorem on monotone convergence to inequality \eqref{eq:intermediateAPrioriR} to obtain, as desired,
\begin{equation*}
\E\biggl[\sup_{t\in[0,T]}\norm{\bX_t}_{(\real^d)^N}^p\biggr] \le C_0 \Bigl(1 + \E\bigl[\norm{\bX_0}_{(\real^d)^N}^p\bigr]\Bigr).
\end{equation*}
We now conclude the proof by obtaining inequality \eqref{eq:APrioriEstimate}. We have
\begin{equation*}
\begin{split}
\E\bigg[\sup_{t\in[0,T]}\norm{\bY_t}_{(\real^d \times F(U))^N}^p\bigg] \le&\, 
2^{p-1}\biggl(\E\bigg[\sup_{t\in[0,T]}\norm{\bX_t}_{(\real^d)^N}^p\bigg] + \E\bigg[\sup_{t\in[0,T]}\norm{\bLambda_t}_{(F(U))^N}^p\bigg]\biggr) \\
\le&\, 2^{p-1} \Bigl(C_0+N^p + C_0\E\bigl[\norm{\bX_0}_{(\real^d)^N}^p\bigr]\!\Bigr)  \le C \Bigl(1 + \E\bigl[\norm{\bX_0}_{(\real^d)^N}^p\bigr]\!\Bigr),
\end{split}
\end{equation*}
where $C>0$ is a constant depending on $p,N,M_v\,,M_\sigma\,,$ and $T$. The proof is concluded.
\end{proof}

\subsection{Pathwise uniqueness of the solution}\label{sec_pathwiseuniquenessN}
We dedicate this subsection to the proof of the pathwise uniqueness for the solutions to the $N$-particle problem \eqref{eq:DicsreteProblem}. 

Let $\bY_1$ and $\bY_2$ be two strong solutions to problem \eqref{eq:DicsreteProblem} defined on the same filtered probability space $(\Omega,\mathscr{F},(\mathscr{F}_t)_{t \in[0,T]}, \prob)$ and associated with the same collection of Brownian motions $(B^1,\dots,B^N)\eqqcolon \bB$. 
According to Definition~\ref{def:StrongSolutionAndUniqueness}, $\bY_1$ and $\bY_2$ are two continuous, $(\real^d \times \cP(U))^N$-valued stochastic processes satisfying \eqref{eq:DicsreteProblem} $\prob$-almost surely, with $\Sigma_s^N$ replaced by  
\begin{equation*}
    \Sigma^N_{i,s} = \frac{1}{N}\sum_{j=1}^N\delta_{(X^j_{i,s},\Lambda^j_{i,s})}, \qquad\text{for $i = 1,2$.}
\end{equation*}
For $t \in [0,T]$, we can estimate
\begin{equation*}
\begin{split}
&\,\E\biggl[\sup_{u\in[0,t]}\norm{\bX_{1,u}- \bX_{2,u}}_{(\real^d)^N}^2\biggr] \\
\le&\, 2t\, \E\biggl[\int_0^t \norm{\bv_{\Sigma^N_{1,s}}(\bX_{1,s},\bLambda_{1,s}) - \bv_{\Sigma^N_{2,s}}(\bX_{2,s},\bLambda_{2,s})}_{(\real^d)^N}^2\,\de s\biggr] \\
&\,+ 8N\,\E\biggl[\int_0^t \norm{\bsigma_{\Sigma^N_{1,s}}(\bX_{1,s},\bLambda_{1,s}) - \bsigma_{\Sigma^N_{2,s}}(\bX_{2,s},\bLambda_{2,s})}_{\real^{dN\times mN}}^2\,\de s\biggr] \\
\le&\, 2T\E\biggl[\int_0^t\bigl[2L_v\bigl(\norm{\bX_{1,s} - \bX_{2,s}}_{(\real^d)^N} + \norm{\bLambda_{1,s} - \bLambda_{2,s}}_{(F(U))^N}\bigr)\bigr]^2\,\de s\biggr] \\
&\,+ 8N\,\E\biggl[\int_0^t\bigl[2L_\sigma\bigl(\norm{\bX_{1,s} - \bX_{2,s}}_{(\real^d)^N} + \norm{\bLambda_{1,s} - \bLambda_{2,s}}_{(F(U))^N}\bigr)\bigr]^2\,\de s\biggr] \\
\le&\, 16\bigl(T L_v^2 + 4NL_\sigma^2\bigr)\,\E\biggl[\int_0^t \bigl(\norm{\bX_{1,s} - \bX_{2,s}}_{(\real^d)^N}^2 + \norm{\bLambda_{1,s} - \bLambda_{2,s}}_{(F(U))^N}^2\bigr)\,\de s\biggr],
\end{split}
\end{equation*}
where we have applied H\"{o}lder inequality, \eqref{eq_BDG}, \eqref{eq:DiscreteEmpLip-v}, and \eqref{eq:DiscreteEmpLip-sigma}. Analogously, by H\"{o}lder inequality and \eqref{eq:DiscreteEmpLip-T}, we have
\begin{equation*}
\begin{split}
\E\biggl[\sup_{u\in[0,t]} \norm{\bLambda_{1,u} -&\, \bLambda_{2,u}}_{(F(U))^N}^2\biggr]\le t\,\E\biggl[\int_0^t \norm{\bcT_{\Sigma^N_{1,s}}(\bX_{1,s},\bLambda_{1,s}) - \bcT_{\Sigma^N_{2,s}}(\bX_{2,s},\bLambda_{2,s})}_{(\real^d)^N}^2\,\de s \biggr] \\
\le&\, 8T L_\cT^2\, \E\biggl[\int_0^t \bigl(\norm{\bX_{1,s} - \bX_{2,s}}_{(\real^d)^N}^2 + \norm{\bLambda_{1,s} - \bLambda_{2,s}}_{(F(U))^N}^2\bigr)\,\de s\biggr].
\end{split}
\end{equation*}

Therefore, combining the two previous estimates, we conclude that, for all $t \in [0,T]$,
\begin{equation*}
\begin{split}
&\,\E\biggl[\sup_{u\in[0,t]} \norm{\bY_{1,u} - \bY_{2,u}}_{(\real^d \times F(U))^N}^2\biggl] \\
\le&\, 2\E\biggl[\sup_{u\in[0,t]} \norm{\bX_{1,u} - \bX_{2,u}}_{(\real^d)^N}^2\biggr] + 2\E\biggl[\sup_{u\in[0,t]} \norm{\bLambda_{1,u} - \bLambda_{2,u}}_{(F(U))^N}^2\biggr] \\
\le&\, 16\bigl(2T L_v^2 + 8NL_\sigma^2 + TL_\cT^2\bigr)\E\biggl[\int_0^t \bigl(\norm{\bX_{1,s} - \bX_{2,s}}_{(\real^d)^N}^2 + \norm{\bLambda_{1,s} - \bLambda_{2,s}}_{(F(U))^N}^2\bigr)\,\de s\biggr] \\
\le&\, C_0\int_0^t \E\biggl[\sup_{u\in[0,s]}\norm{\bY_{1,u} - \bY_{2,u}}_{(\real^d \times F(U))^N}^2\biggr]\,\de s,
\end{split}
\end{equation*}
where $C_0 \coloneqq 32\bigl(2T L_v^2 + 8NL_\sigma^2 + TL_\cT^2\bigr)$. 
Now, the \emph{a priori} estimate obtained in Proposition~\ref{prop:stima a priori} allows us to prove the boundedness of the function
\begin{equation*}
[0,T]\ni t\longmapsto v(t) \coloneqq \E\biggl[\sup_{u\in[0,t]} \norm{\bY_{1,u} - \bY_{2,u}}_{(\real^d \times F(U))^N}^2\biggr].
\end{equation*}
Indeed, by~\eqref{eq:APrioriEstimate},
\begin{equation*}
\begin{split}
v(t) \leq 
2\sum_{i=1}^2 \E\biggl[\sup_{u\in[0,t]} \norm{\bY_{i,u}}_{(\real^d \times F(U))^N}^2\biggr] 
\le
4C\Bigl(1 + \E\bigl[\norm{\bX_0}_{(\real^d)^N}^2\bigr]\Bigr),
\end{split}
\end{equation*}
which is bounded by virtue of the square integrability of $\bX_0$.
Therefore we can apply Gr\"{o}nwall inequality obtaining that
$
v(t) = 0$, 
for all 
$t \in [0,T]$.
In particular, for $t = T$, we 
deduce that
\begin{equation*}
\prob\biggl(\sup_{t\in[0,T]} \norm{\bY_{1,t} - \bY_{2,t}}_{(\real^d \times F(U))^N}^2 = 0\biggr) = 1,
\end{equation*}
which is \eqref{eq_sarapathwiseuniqueness}.
Pathwise uniqueness is proved. \qed


\section{The mean-field model}\label{sec:MeanField}


In the present section we will address rigorously the problem of well-posedness of the mean-field model \eqref{eq:MeanField}. Hence, we begin with the following

\begin{defn}
Let 
$\barB \colon \Omega \longrightarrow \cC([0,T],\real^m)$ be an
$m$-dimensional standard Brownian motion defined on a filtered probability space $\bigl(\Omega, \sF, (\sF_t)_{t\in[0,T]},\prob\bigr)$, let $\barX_0 \colon\Omega\longrightarrow\real^d$ be an $\sF_0$-measurable random variable, and let $\barLambda_0\colon\Omega\longrightarrow\mathcal{P}(U)$ be an $\sF_0$-measurable random variable. 
We define \emph{strong solution} to problem \eqref{eq:MeanField} a continuous, $(\sF_t)_t$-adapted stochastic process $\barY = (\barX,\barLambda) \colon \Omega \longrightarrow \cC\bigl([0,T],\real^d\times\cP(U)\bigr)$ satisfying, $\prob$-almost surely, for all $t \in [0,T]$,
\begin{equation}
\begin{dcases}
\barr{X}_t = \barr{X}_0 + \int_0^t v_{\Sigma_s}(\barr{X}_s,\barr{\Lambda}_s)\,\de s + \int_0^t \sigma_{\Sigma_s}(\barr{X}_s,\barr{\Lambda}_s)\,\de \barr{B}_s\,, \\
\barr{\Lambda}_t = \barr{\Lambda}_0 +\int_0^t \cT_{\Sigma_s}(\barr{X}_s,\barr{\Lambda}_s)\,\de s, \\
\Law(\barY_t) = \Sigma_t \in \cP_1(\real^d \times\cP(U)).
\end{dcases}
\end{equation}
We say that such a solution is \emph{pathwise unique} if, given two strong solutions $\barY_1$, $\barY_2$ of \eqref{eq:MeanField} (with the same Brownian motion and initial datum), we have
\begin{equation}
    \prob\Bigl(\barr{Y}_{1,t} = \barr{Y}_{2,t}\,, \textup{ for every } t\in[0,T]\Bigr) = 1.
\end{equation}
\label{def:StrongSolutionMF}
\end{defn}

The main result of this section is the following well-posedness theorem.

\begin{thm}[well-posedness of the mean-field system~\eqref{eq:MeanField}] \label{thm:WellPosednessMeanField}
Let us assume that \eqref{eq_fieldsprop} and \eqref{eq:T_prop_geometry} are satisfied and that the spatial initial datum $\barX_0$ belongs to $L^2(\Omega,\sF,\prob)$. Then problem \eqref{eq:MeanField} admits a strong solution $\barY$, 
which is pathwise unique in the class of strong solutions such that the map $[0,T] \ni t \longmapsto \Sigma_t \in \cP_1(\real^d\times\cP(U))$ 
satisfies the condition
\begin{equation}\label{eq:boundedness of moment of solutions}
\sup_{t \in [0,T]} \int_{\real^d \times \cP(U)} \norm{x}_{\real^d}\,\de\Sigma_t(x,\lambda) < +\infty.
\end{equation}
Moreover, the integrability of the spatial initial datum is inherited, uniformly in time, by the solution, namely, if  $\barX_0 \in L^p(\Omega,\mathscr{F},\prob)$ with $p \ge 2$, the process $\barY$ satisfies 
\begin{equation}\label{eq:APrioriEstimateMeanFieldSol}
    \E\biggl[\sup_{t\in[0,T]} \norm{\barY_t}_{\real^d \times F(U)}^p\biggr] \le C^\Sigma\Bigl(1+\E\bigl[\norm{\barX_0}_{\real^d}^p\bigr]\Bigr), 
\end{equation}
for some constant $C^\Sigma>0$ depending on $p$, $T$, $M_v^\Sigma$\,, and $M_\sigma^\Sigma$ (the last two constants are introduced in the statement of Proposition~\ref{prop:MeanFieldLipSub-v-sigma-T}, see estimates \eqref{eq:MeanFieldSub} below, with $t\longmapsto\Psi_t$ replaced by $t\longmapsto\Sigma_t=\Law(\barY_t)$).
\end{thm}

\begin{proof}
The proof of the theorem is articulated in various steps, which are addressed in the rest of this section.
As a preliminary result, in Proposition~\ref{prop:MeanFieldLipSub-v-sigma-T} below, we prove structural properties of the fields $v$, $\sigma$, and $\cT$ in the right-hand side of \eqref{eq:MeanFieldProblem} for a given curve of measures $t\longmapsto\Psi_t$\,; moreover, in Proposition~\ref{prop:a priori mean field}, we prove the \emph{a priori} estimate~\eqref{eq:APrioriEstimateMeanField}, of which estimate~\eqref{eq:APrioriEstimateMeanFieldSol} will be a particular instance satisfied by the solutions to~\eqref{eq:MeanField} which will be constructed.

In Section~\ref{sec:MeanFieldProofOutline}, we outline the main steps of the construction of the solution to~\eqref{eq:MeanField}, which relies on solving the auxiliary SDE presented in~\eqref{eq:AuxiliaryMeanField}.
In Section~\ref{sec:studyauxproblem}, we deal with the well-posedness of this auxiliary problem, which will be instrumental in the successive fixed-point argument needed to prove the existence of a solution to~\eqref{eq:MeanField}, see Section~\ref{sec:MeanFieldFixedPoint}.
Finally, the uniqueness argument is discussed in Section~\ref{sec:MeanFieldConclusion}.
\end{proof}

\begin{prop}\label{prop:MeanFieldLipSub-v-sigma-T} Let us consider a curve of measures $[0,T] \ni t \longmapsto \Psi_t \in \cP_1(\real^d \times \cP(U))$ such that
\begin{equation}\label{eq:boundedness of moment}
    \sup_{t \in [0,T]} \int_{\real^d \times \cP(U)} \norm{x}_{\real^d}\,\de\Psi_t(x,\lambda) < +\infty.
\end{equation} Then, for all $(t,x_1,\lambda_1), (t,x_2,\lambda_2) \in [0,T]\times\real^d\times\cP(U)$, the following estimates hold
\begin{subequations}\label{eq:MeanFieldLip}
\begin{align}
\norm{v_{\Psi_t}(x_1,\lambda_1) - v_{\Psi_t}(x_2,\lambda_2)}_{\real^d} &\le L_v \bigl(\norm{x_1 - x_2}_{\real^d} + \norm{\lambda_1 - \lambda_2}_{F(U)}\bigr), \label{eq:MeanFieldLip-v} \\
\norm{\sigma_{\Psi_t}(x_1,\lambda_1) - \sigma_{\Psi_t}(x_2,\lambda_2)}_{\real^{d\times m}} &\le L_\sigma \bigl(\norm{x_1 - x_2}_{\real^d} + \norm{\lambda_1 - \lambda_2}_{F(U)}\bigr), \label{eq:MeanFieldLip-sigma}\\
\norm{\cT_{\Psi_t}(x_1,\lambda_1) - \cT_{\Psi_t}(x_2,\lambda_2)}_{F(U)} &\le L_\cT \bigl(\norm{x_1 - x_2}_{\real^d} + \norm{\lambda_1 - \lambda_2}_{F(U)}\bigr), \label{eq:MeanFieldLip-T}
\end{align}
\end{subequations}
where the constants $L_v$\,, $L_\sigma$\,, and $L_\cT$ are those appearing in \eqref{eq_fieldsprop}. 
Furthermore, there exist constants $M_v^\Psi$\,, $M_\sigma^\Psi$\,, and $M_\cT^\Psi > 0$, depending on the curve $\Psi$, such that, for all $(t,x,\lambda) \in [0,T]\times\real^d\times\cP(U)$, the following sublinearity estimates hold 
\begin{subequations}\label{eq:MeanFieldSub}
\begin{align}
\norm{v_{\Psi_t}(x,\lambda)}_{\real^d} &\le M_v^\Psi \bigl(1 + \norm{x}_{\real^d}\bigr),\label{eq:MeanFieldSub-v} \\
\norm{\sigma_{\Psi_t}(x,\lambda)}_{\real^{d\times m}} &\le M_\sigma^\Psi\bigl(1 + \norm{x}_{\real^d}\bigr), \label{eq:MeanFieldSub-sigma} \\
\norm{\cT_{\Psi_t}(x,\lambda)}_{F(U)} &\le M_\cT^\Psi \bigl(1 + \norm{x}_{\real^d}\bigr). \label{eq:MeanFieldSub-T}
\end{align}
\end{subequations}
\end{prop}
\begin{proof} Inequalities \eqref{eq:MeanFieldLip} immediately follow from the structural assumptions \eqref{eq_fieldsprop}. 
We shall only prove the first of the estimates \eqref{eq:MeanFieldSub-v}, as the other two follow analogously. 
We fix an arbitrary element $(t_0,x_0,\lambda_0) \in [0,T]\times\real^d\times\cP(U)$, and by making use of inequalities \eqref{eq:v_prop} and \eqref{eq_triangolino} we obtain that, for all $(t,x,\lambda) \in [0,T]\times\real^d\times\cP(U)$,
\begin{equation}\label{eq:intermediate_sub_proof}
\begin{split}
&\,\norm{v_{\Psi_t}(x,\lambda)}_{\real^d} \le \norm{
v_{\Psi_{t_0}}(x_0,\lambda_0)}_{\real^d} + \norm{v_{\Psi_t}(x,\lambda) - 
v_{\Psi_{t_0}}(x_0,\lambda_0)}_{\real^d} \\
\le&\,  \norm{v_{\Psi_{t_0}}(x_0,\lambda_0)}_{\real^d} + L_v\bigl(\norm{x - x_0}_{\real^d} + \norm{\lambda - \lambda_0}_{F(U)} + W_1(\Psi_t,\Psi_{t_0})\bigr) \\
\le&\, \norm{v_{\Psi_{t_0}}(x_0,\lambda_0)}_{\real^d} + L_v\bigl(2 + \norm{x}_{\real^d} + \norm{x_0}_{\real^d} + W_1(\Psi_t,\Psi_{t_0})\bigr).
\end{split}
\end{equation}
Now, letting $y_0 \coloneqq (x_0,\lambda_0)$ and $c \coloneqq 2 + \norm{x_0}_{\real^d} + W_1(\Psi_{t_0}, \delta_{y_0})$, we have that for all $t \in [0,T]$
\begin{equation*}
W_1(\Psi_t, \Psi_{t_0}) \le  c + \int_{\real^d \times \cP(U)} \norm{x}_{\real^d}\,\de\Psi_t(x,\lambda),
\end{equation*}
where we have employed the triangle inequality for the metric $W_1$\,, inequality \eqref{eq_triangolino}, and the fact that
\begin{equation}\label{eq:moments as W distance from delta}
    W_1(\Psi_t, \delta_{y_0}) = \int_{\real^d \times \cP(U)} \norm{y - y_0}_{\real^d \times F(U)}\,\de\Psi_t(y)
\end{equation}
(which is an immediate consequence of $\Gamma(\Psi_t, \delta_{y_0}) = \{\Psi_t \otimes\delta_{y_0}\}$).
Therefore, by virtue of hypothesis \eqref{eq:boundedness of moment}, we conclude that there exists $C^\Psi > 0$, depending on the curve $\Psi$, such that $\sup_{t \in [0,T]}W_1(\Psi_t,\Psi_{t_0}) \le C^\Psi$. This bound, combined with \eqref{eq:intermediate_sub_proof}, yields the claim.
\end{proof}

\begin{rem}\label{rem: about measures on curves} 
Proposition \ref{prop:MeanFieldLipSub-v-sigma-T} applies, in particular, when the curve $t \longmapsto \Psi_t$ consists of the time marginals of a fixed probability measure $\Psi \in \cP_p(\cC([0,T],\real^d \times \cP(U)))$, with $p \in [1,+\infty) $; namely, $\Psi_t = (\ev_t)_\sharp\Psi$ for all $t \in [0,T]$.
Indeed, as a consequence of Lebesgue Dominated Convergence Theorem it can be verified that 
the curve $[0,T] \ni t \longmapsto \Psi_t \in \cP_p(\real^d \times \cP(U))$ is continuous with respect to the Wasserstein metric $W_p$, hence, by Weierstrass Theorem, there exists $\tilde{C}^\Psi > 0$ such that $W_p(\Psi_t,\delta_{y_0}) \le \tilde{C}^\Psi$ for all $t\in[0,T]$, where $y_0$ is an arbitrary point in $\real^d \times \cP(U)$. 
By choosing for convenience $y_0$ of the form $(0,\lambda_0)$, and by using equation \eqref{eq:moments as W distance from delta}, we immediately conclude that
\begin{equation*}
\int_{\real^d \times \cP(U)} \norm{x}_{\real^d}\,\de\Psi_t(x,\lambda) \le W_1(\Psi_t, \delta_{y_0}) \le W_p(\Psi_t, \delta_{y_0}) \le \tilde{C}^\Psi,\quad \text{for all } t \in [0,T],
\end{equation*}
therefore condition \eqref{eq:boundedness of moment} is satisfied.
\end{rem} 

We prove the following \emph{a priori} estimate on the moments of the solutions to~\eqref{eq:MeanField}.
\begin{prop}[\emph{a priori} estimate on the moments of the solutions]\label{prop:a priori mean field}
Let us assume that \eqref{eq_fieldsprop} and \eqref{eq:T_prop_geometry} are satisfied and that the spatial initial datum $\barX_0$ belongs to $L^p(\Omega,\sF,\prob)$, with $p \ge 2$. Moreover, let $\barY$ be a strong solution to problem \eqref{eq:MeanField} such that the map $[0,T] \ni t \longmapsto \Sigma_t = \Law(\barY_t) \in \cP_1(\real^d \times \cP(U))$ satisfies condition \eqref{eq:boundedness of moment of solutions}.
Then the following estimate holds 
\begin{equation}\label{eq:APrioriEstimateMeanField}
\E\biggl[\sup_{t\in[0,T]} \norm{\barY_t}_{\real^d \times F(U)}^p\biggr] \le C^\Sigma \Bigl(1 + \E\bigl[\norm{\barX_0}_{\real^d}^p\bigr]\Bigr),
\end{equation}
where $C^\Sigma >0$ depends on $p$, $T$, $M_v^\Sigma$\,, and $M_\sigma^\Sigma$ (see estimates \eqref{eq:MeanFieldSub}, with $t\longmapsto\Psi_t$ replaced by $t\longmapsto\Sigma_t=\Law(\barY_t)$).
\end{prop}
\begin{proof} The proof is completely analogous to that of Proposition \ref{prop:stima a priori}. Here, the sublinearity estimates \eqref{eq:DiscreteEmpSub-v} and \eqref{eq:DiscreteEmpSub-sigma} are replaced by the sublinearity estimates (uniform in time) \eqref{eq:MeanFieldSub-v} and \eqref{eq:MeanFieldSub-sigma} respectively. Note that the latter are valid since, by hypothesis, the map $t \longmapsto \Sigma_t$ satisfies condition \eqref{eq:boundedness of moment}.
\end{proof}

\subsection[Construction of a solution to the mean-field problem]{Construction of a solution to the mean-field problem: outline of the proof}
\label{sec:MeanFieldProofOutline}
The argument will be articulated in two steps. First, we consider a filtered probability space $(\Omega,\sF,(\sF_t)_{t\in[0,T]},\prob)$, an initial datum $\barY_0 = (\barX_0,\barLambda_0)$ with $\barX_0 \in L^2(\Omega, \sF,\prob)$, 
and an $m$-dimensional standard Brownian motion $\barB$ adapted to the filtration $(\sF_t)_{t \in [0,T]}$. We then fix a probability measure $\Psi \in \cP_2(\cC([0,T],\real^d \times \cP(U)))$ and we consider the following \emph{auxiliary problem} 
\begin{equation}\label{eq:AuxiliaryMeanField}
\begin{dcases}
    \de\barX^\Psi_t = v_{\Psi_t}(\barX^\Psi_t,\barLambda^\Psi_t)\,\de t + \sigma_{\Psi_t}(\barX^\Psi_t,\barLambda^\Psi_t)\,\de \barB_t\,, \\
    \de\barLambda^\Psi_t = \cT_{\Psi_t}(\barX^\Psi_t,\barLambda^\Psi_t)\,\de t,
\end{dcases}
\end{equation}
where $\Psi_t \coloneqq (\ev_t)_\sharp\Psi$, for all $t \in [0,T]$. We shall prove that problem \eqref{eq:AuxiliaryMeanField} has a pathwise unique strong solution $\barY^\Psi = (\barX^\Psi,\barLambda^\Psi) \colon \Omega \longrightarrow \cC([0,T],\real^d \times \cP(U))$, which satisfies
\begin{equation*}
    \Law(\barY^\Psi) \in \cP_2\bigl(\cC([0,T],\real^d \times \cP(U))\bigr).
\end{equation*}
Therefore, it will be well-defined a map
\begin{equation}
\begin{aligned}
\cS \colon \cP_2\bigl(\cC([0,T],\real^d \times \cP(U))\bigr) &\longrightarrow \cP_2\bigl(\cC([0,T],\real^d \times \cP(U))\bigr) \\
\Psi &\longmapsto \cS(\Psi) \coloneqq \Law(\barr{Y}^\Psi)
\end{aligned}
\end{equation}
which associates to the measure $\Psi \in \cP_2\bigl(\cC([0,T],\real^d \times \cP(U))\bigr)$ the law of the stochastic process $\barY^\Psi$ solving the auxiliary problem (\ref{eq:AuxiliaryMeanField}). 
We explicitly note that pathwise uniqueness of the solution guarantees that the measure $\cS(\Psi)$ is unambiguously determined. 

In the second step of the proof, we will show that the map $\cS$ admits a unique fixed point. To achieve this, we will employ the Banach--Caccioppoli fixed-point theorem in the complete metric space $\bigl(\cP_2\bigl(\cC([0,T],\real^d \times \cP(U))\bigr), W_2\bigr)$. Therefore we will obtain the existence of a unique measure $\Sigma \in \cP_2\bigl(C([0,T],\real^d \times \cP(U))\bigr)$ satisfying the \emph{fixed-point} equation
\begin{equation}
\cS(\Sigma) = \Law(\barY^\Sigma) = \Sigma.
\end{equation}
By setting $\barY \coloneqq \barY^\Sigma$, we will find a strong solution to the mean-field problem \eqref{eq:MeanField}. Indeed, since, by construction, $\Sigma = \Law(\barY)$, and $\Sigma_t \coloneqq (\mathrm{ev}_t)_\sharp\Sigma \in \cP_2(\real^d \times \cP(U))$ for all $t \in [0,T]$, we have, by associativity of the \emph{push-forward} operation,
\begin{equation*}
\Law(\barr{Y}_t) = \Law((\mathrm{ev}_t) \circ \barr{Y}) = ((\mathrm{ev}_t) \circ \barr{Y})_\sharp \prob =  (\mathrm{ev}_t)_\sharp(\barr{Y}_\sharp \prob) = (\mathrm{ev}_t)_\sharp\Sigma = \Sigma_t,
\end{equation*}
for all $t \in [0,T]$, as prescribed in \eqref{eq:MeanFieldSelfReferential}. In the following paragraphs, we will present the construction in detail.

\subsection{Study of the auxiliary problem}
\label{sec:studyauxproblem}
We begin by defining the notion of solution to the \emph{auxiliary problem} \eqref{eq:AuxiliaryMeanField} and the main result about its well-posedness.

\begin{defn}\label{def:strongSolutionAux} 
Let us fix a measure $\Psi \in \cP_2\bigl(\cC([0,T],\real^d \times \cP(U))\bigr)$, and define $\Psi_t \coloneqq (\ev_t)_\sharp \Psi$ for all $t \in [0,T]$. Moreover, let 
$\barB \colon \Omega \longrightarrow \cC([0,T],\real^m)$ be an
$m$-dimensional standard Brownian motion defined on a filtered probability space $\bigl(\Omega, \sF, (\sF_t)_{t\in[0,T]},\prob\bigr)$, let $\barX_0 \colon\Omega\longrightarrow\real^d$ be an $\sF_0$-measurable random variable, and let $\barLambda_0\colon\Omega\longrightarrow\mathcal{P}(U)$ be an $\sF_0$-measurable random variable. We define \emph{strong solution} to problem \eqref{eq:AuxiliaryMeanField} a continuous, $(\sF_t)_t$-adapted stochastic process $\barY^\Psi = (\barX^\Psi,\barLambda^\Psi) \colon \Omega \longrightarrow \cC\bigl([0,T],\real^d\times\cP(U)\bigr)$ satisfying, $\prob$-almost surely, for all $t \in [0,T]$,
\begin{equation}
\begin{dcases}
    \barX^\Psi_t = \barX_0 + \int_0^t v_{\Psi_s}(\barX^\Psi_s,\barLambda^\Psi_s)\,\de s + \int_0^t \sigma_{\Psi_s}(\barX^\Psi_s,\barLambda^\Psi_s)\,\de \barB_s\,, \\
    \barLambda^\Psi_t = \barLambda_0 +\int_0^t \cT_{\Psi_s}(\barX^\Psi_s,\barLambda^\Psi_s)\,\de s.
\end{dcases}
\end{equation}
We say that such a solution is \emph{pathwise unique} if, given two strong solutions $\barY_1^\Psi$, $\barY_2^\Psi$ of \eqref{eq:AuxiliaryMeanField} (with the same Brownian motion and initial datum), we have
\begin{equation}
    \prob\Bigl(\barr{Y}_{1,t}^\Psi = \barr{Y}_{2,t}^\Psi\,, \textup{ for every } t\in[0,T]\Bigr) = 1.
\end{equation}
\end{defn}

\begin{prop}\label{prop:AuxiliaryProblemWellPosedness} Let us assume that \eqref{eq_fieldsprop} and \eqref{eq:T_prop_geometry} are satisfied and that the spatial initial datum $\barX_0$ belongs to $L^2(\Omega,\sF,\prob)$. 
Then the auxiliary problem \eqref{eq:AuxiliaryMeanField} admits a pathwise unique strong solution $\barY^\Psi$ according to Definition~\ref{def:strongSolutionAux}.
Moreover, the integrability of the spatial initial datum is inherited, uniformly in time, by the solution, namely, if  $\barX_0 \in L^p(\Omega,\mathscr{F},\prob)$ with $p \ge 2$, the process $\barY^\Psi$ satisfies  
\begin{equation*}
    \E\biggl[\sup_{t\in[0,T]} \norm{\barY^\Psi_t}_{\real^d \times F(U)}^p\biggr] \le C^\Psi\Bigl(1+\E\bigl[\norm{\barX_0}_{\real^d}^p\bigr]\Bigr),
\end{equation*}
where $C^\Psi > 0$ depends on $p$, $T$, $M_v^\Psi$\,, and $M_\sigma^\Psi$ (see estimates \eqref{eq:MeanFieldSub}, with $t\longmapsto\Psi_t=(\ev_t)_\sharp \Psi$; see also Remark~\ref{rem: about measures on curves}).
\end{prop}
The proof of Proposition \ref{prop:AuxiliaryProblemWellPosedness} will proceed along the lines of Section~\ref{sec:NParticleModel}, by exploiting the Lipschitz continuity and sublinearity properties \eqref{eq:MeanFieldLip}, \eqref{eq:MeanFieldSub}, and the geometric property \eqref{eq:T_prop_geometry}.
Unlike the $N$-particle system, the auxiliary problem \eqref{eq:AuxiliaryMeanField} features an explicit dependence on time. However, the techniques employed in the study of the $N$-particle system can be adapted to the present problem with almost no modifications, as we will see.

Let us begin with the proof of existence. Our strategy will be based again on the \emph{method of successive approximation}.
We first fix $\theta > 0$ such that property \eqref{eq:T_prop_geometry} is satisfied; this condition allows us to introduce the field $\cG \colon [0,T]\times\real^d\times\cP(U) \longrightarrow \cP(U)$, defined as
\begin{equation}\label{eq:MeanFieldDef-G}
    \cG_{\Psi_t}(x,\lambda) \coloneqq \lambda + \theta\cT_{\Psi_t}(x,\lambda)
\end{equation}
(notice that in the auxiliary problem (\ref{eq:AuxiliaryMeanField}) the measure $\Psi \in \cP_2\bigl(\cC([0,T],\real^d \times \cP(U))\bigr)$ is fixed once once and for all), so that the field $\cT$ can be rewritten as
\begin{equation*}
    \cT_{\Psi_t}(x,\lambda) = \frac{\cG_{\Psi_t}(x,\lambda) - \lambda}{\theta},
\end{equation*}
for all $(t,x,\lambda) \in [0,T]\times\real^d\times\cP(U)$, allowing us to employ the strategy of Brezis \cite[Corollaire 1.1]{Brezis-Operateur-Maximaux} in the definition of the Picard iterations. Before writing down the expression of the iterations, we state the structural properties of the field $\cG$.

\begin{prop}\label{prop:MeanFieldLipSub-G} 
Let $\Psi$ be as in Definition \ref{def:strongSolutionAux}. Then there exists a constant $L_\cG>0$, depending on $L_\cT$ and $\theta$, such that, for all $(t,x_1,\lambda_1), (t,x_2,\lambda_2) \in [0,T] \times \real^d \times \cP(U)$, the following inequality holds
\begin{subequations}\label{eq:MeanFieldLipSub-G}
\begin{equation}\label{eq:MeanFieldLip-G}
\norm{\cG_{\Psi_t}(x_1,\lambda_1) - \cG_{\Psi_t}(x_2,\lambda_2)}_{F(U)} \le L_{\cG} \bigl(\norm{x_1 - x_2}_{\real^d} + \norm{\lambda_1 - \lambda_2}_{F(U)}\bigr).
\end{equation}
Furthermore, there exists $M_\cG^\Psi > 0$, depending on $M_\cT^\Psi$ and $\theta$, such that
\begin{equation}\label{eq:MeanFieldSub-G}
\norm{\cG_{\Psi_t}(x,\lambda)}_{F(U)} \le M_\cG^\Psi \bigl(1 + \norm{x}_{\real^d}\bigr),
\end{equation}
\end{subequations}
for all $(t,x,\lambda) \in [0,T]\times\real^d\times\cP(U)$.
\end{prop}
\begin{proof} By Remark \ref{rem: about measures on curves}, we have that the map $t \longmapsto \Psi_t = (\ev_t)_\sharp \Psi$ satisfies condition \eqref{eq:boundedness of moment}, therefore, estimates \eqref{eq:MeanFieldLip-T} and \eqref{eq:MeanFieldSub-T} are valid, and it can easily be checked that they are inherited by $\cG$ as a consequence of its definition \eqref{eq:MeanFieldDef-G}. \end{proof}
\begin{notation}
In the ensuing part of this section, and in this section alone, we will denote by~$\barY_n$ and~$\barY$ the iterations and the solutions, respectively, of the auxiliary problem \eqref{eq:AuxiliaryMeanField} associated with the measure $\Psi \in \cP_2\bigl(\cC([0,T],\real^d \times \cP(U))\bigr)$, and the same convention will apply to their spatial components $\barX_n$\,, $\barX$, and mixed-strategy components $\barLambda_n$\,, $\barLambda$. Therefore, we will avoid the more complete (though heavier) notation $\barY_n^\Psi$ and $\barY^\Psi$.
This choice is justified by the fact that in this section the measure $\Psi$ is fixed.
\end{notation}


We are now in a position to write the expression of the iterations and prove their convergence to a solution to the auxiliary problem \eqref{eq:AuxiliaryMeanField}.
They are define recursively as follows:
\begin{equation}\label{eq:IterationsMeanField}
\begin{dcases}
\barX_{n+1,t} \coloneqq \barX_0 + \int_0^t v_{\Psi_s}(\barX_{n,s},\barLambda_{n,s})\,\de s + \int_0^t \sigma_{\Psi_s}(\barX_{n,s},\barLambda_{n,s})\,\de\barB_s\,, \\
\barLambda_{n+1,t} \coloneqq e^{-\frac{t}{\theta}}\barLambda_0 + \frac{1}{\theta}\int_0^t e^{\frac{s-t}{\theta}} \cG_{\Psi_s}(\barX_{n,s},\barLambda_{n,s})\,\de s,
\end{dcases}
\end{equation}
for all $t \in [0,T]$ and $n = 0,1,2,\dots$ The initial iteration in given by
\begin{equation}\label{eq:Iterations0MeanField}
\begin{dcases}
\barX_{0,t} \coloneqq \barX_0,\\
\barLambda_{0,t} \coloneqq \barLambda_0,
\end{dcases}\qquad\text{for all } t \in [0,T]. 
\end{equation}
We anticipate that many of the ensuing arguments are completely analogous to those followed in the proof of well-posedness of the $N$-particle system \eqref{eq:DicsreteProblem}. Therefore, to avoid redundancy, we will only sketch the proof of the following statements, pointing out only the properties and results that are needed in each case.

The proof of the following proposition is the same as that of Proposition \ref{prop:DiscreteIterationsContAdapt}. 
\begin{prop}\label{prop:MeanFieldIterateContAdatt}
The stochastic processes $\barX_n$ and $\barLambda_n$, defined by \eqref{eq:IterationsMeanField} and \eqref{eq:Iterations0MeanField}, are continuous and adapted to the filtration $(\sF_t)_{t \in [0,T]}$ for all $n = 0,1,2,\dots$
\end{prop}

Arguing like in 
Remark \ref{rem:trajectories are confined in P(U)}, it can readily be checked that, as a consequence of the convexity of $\cP(U)$ and of the geometric property \ref{eq:T_prop_geometry}, we have that $\barLambda_{n,t} \in \cP(U)$, for all $t \in [0,T]$ and for all $n = 0,1,2,\dots$.

Let us now introduce the four estimates which represent the analogous for the auxiliary problem \eqref{eq:AuxiliaryMeanField} of estimates \eqref{eq:12QuattroStime} and \eqref{eq:34QuattroStime}.

\begin{prop}\label{prop:QuattroStimeMeanField} Under the hypotheses of Proposition~\ref{prop:AuxiliaryProblemWellPosedness}, the following estimates hold:
\begin{subequations}\label{eq:12QuattroStimeMeanField}
\begin{align}
\begin{split}\label{eq:disuguaglianza(a)MeanField}
\E\bigg[\sup_{u\in[0,t]} \norm{\barX_{1,u} - \barX_0}_{\real^d}^2\bigg] \le &\, \Bigl(6 (M_v^\Psi)^2 t^2 + 24(M_\sigma^\Psi)^2 t \Bigr) \\
&\, \cdot \Bigl(1 + \E\bigl[\norm{\barX_0}_{\real^d}^2\bigr] + \E\bigl[\norm{\barLambda_0}_{F(U)}^2\bigr]\Bigr),
\end{split}\\
\begin{split}\label{eq:disuguaglianza(b)MeanField}
\E\biggl[\sup_{u\in[0,t]} \norm{\barLambda_{1,u} - \barLambda_0}_{F(U)}^2\biggr] \le&\, \frac{2t^2}{\theta^2}  \E\bigl[\norm{\barLambda_0}_{F(U)}^2\bigr] \\ &+ \frac{6(M_\cG^\Psi)^2t^2}{\theta^2} \Bigl(1 + \E\bigl[\norm{\barX_0}_{\real^d}^2\bigr] + \E\bigl[\norm{\barLambda_0}_{F(U)}^2\bigr]\Bigr) 
\end{split}
\end{align}
\end{subequations}
for all $t \in [0,T]$, and
\begin{subequations}\label{eq:34QuattroStimeMeanField}
\begin{align}
\begin{split}\label{eq:disuguaglianza(c)MeanField}
\E\bigg[\sup_{u\in[0,t]} \norm{\barX_{n+1,u} &\,- \barX_{n,u}}_{\real^d}^2\bigg] \le \bigl(4L_v^2t + 16L_\sigma^2 \bigr) \\ 
&\, \cdot\E\biggl[\int_0^t \bigl(\norm{\barX_{n,s} - \barX_{n-1,s}}_{\real^d}^2+ \norm{\barLambda_{n,s} - \barLambda_{n-1,s}}_{F(U)}^2\bigr)\,\de s \biggr],
\end{split}\\
\begin{split}\label{eq:disuguaglianza(d)MeanField}
\E\bigg[\sup_{u\in[0,t]} \norm{\barLambda_{n+1,u} &\,- \barLambda_{n,u}}_{(\real^d)^N}^2\bigg] \le \frac{2L_\cG^2 t}{\theta^2} \E\biggl[\int_0^t \bigl(\norm{\barX_{n,s} - \barX_{n-1,s}}_{\real^d}^2 \\
&\,\hspace{4.7cm}+ \norm{\barLambda_{n,s} - \barLambda_{n-1,s}}_{F(U)}^2\bigr)\,\de s \biggr]
\end{split}
\end{align}
\end{subequations}
for all $n = 1,2,\dots$ and $t \in [0,T]$, where $L_v,L_\sigma,L_\cG$ and $M_v^\Psi, M_\sigma^\Psi, M_\cG^\Psi$ indicate the Lipschitz and sublinearity constants of the fields  $v_\Psi$, $\sigma_\Psi$, and $\cG_\Psi$ respectively (see \eqref{eq:MeanFieldLip}, \eqref{eq:MeanFieldSub}, and \eqref{eq:MeanFieldLipSub-G}).
\end{prop}

\begin{proof} The proof is the same as that of Proposition \ref{prop:QuattroStime}. In the present case, the Lipschitz estimates \eqref{eq:DiscreteEmpLip-v}, \eqref{eq:DiscreteEmpLip-sigma}, \eqref{eq:DiscreteEmpLip-G} must be replaced by \eqref{eq:MeanFieldLip-v}, \eqref{eq:MeanFieldLip-sigma}, \eqref{eq:MeanFieldLip-G}, respectively, whereas sublinearity estimates \eqref{eq:DiscreteEmpSub-v}, \eqref{eq:DiscreteEmpSub-sigma}, \eqref{eq:DiscreteEmpSub-G}, must be replaced by \eqref{eq:MeanFieldSub-v}, \eqref{eq:MeanFieldSub-sigma}, \eqref{eq:MeanFieldSub-G}. The validity of the latter is ensured by Remark \ref{rem: about measures on curves}. The role of H\"{o}lder and Burkholder--Davis--Gundy \eqref{eq_BDG}  inequalities is unchanged.
\end{proof}

We now introduce the sequence of $\real^d \times \cP(U)$-valued continuous stochastic processes $\{\barY_n\}_{n=0}^{\infty}\,$, which we define, for all $n = 0,1,2,\dots$, as 
\begin{equation}\label{eq:BarYnDef}
\barY_{n,t} \coloneqq \bigl(\barX_{n,t},\barLambda_{n,t}\bigr)
\end{equation}
for all $t \in [0,T]$, where $\barX_n$, $\barLambda_n$ are defined in \eqref{eq:IterationsMeanField}. As we will see, the uniform limit of $\{\barY_n\}_{n=0}^{\infty}$ will provide a strong solution to the auxiliary problem \eqref{eq:AuxiliaryMeanField}.

\begin{prop}\label{prop:MeanFieldfactInequality_t} Let $\{\barY_n\}_{n=0}^{\infty}$ be the sequence of continuous $\real^d \times \cP(U)$-valued stochastic processes $\{\barY_n\}_{n=0}^{\infty}$ defined by \eqref{eq:BarYnDef}. Then, for all $n = 0,1,2,\dots$ and for all $t \in [0,T]$, the following inequality holds
\begin{equation}\label{eq:MeanFieldfactInequality_t}
\E\biggl[\sup_{u\in[0,t]} \norm{\barY_{n+1,u} - \barY_{n,u}}_{\real^d \times F(U)}^2\biggr] \le \frac{(\cR^\Psi t)^{n+1}}{(n+1)!}\,,
\end{equation}
where $\mathcal{R}^\Psi>0$ depends on $\theta,T$, on the constants $L_{v}\,, L_{\sigma}\,, L_{\cG}$\,, $M_{v}^\Psi\,, M_{\sigma}^\Psi\,, M_{\cG}^\Psi$ and on the second moment of the initial datum $\barX_0$.
\end{prop}

\begin{proof} The argument is identical to that of Proposition \ref{prop:factInequality_t}. Here, estimates \eqref{eq:12QuattroStime} and \eqref{eq:34QuattroStime} are replaced by \eqref{eq:12QuattroStimeMeanField} and \eqref{eq:34QuattroStimeMeanField}, respectively.
\end{proof}

\begin{prop}\label{prop:MeanFieldIterateCauchy} There exists a $\prob$-negligible set $\tilde{\cZ}\in\sF$ such that for all $\omega\in\tilde{\cZ}^c$,
$\{\barY_n(\omega)\}_{n=0}^{\infty} \subseteq \cC([0,T],\real^d \times \cP(U))$
is a Cauchy sequence in the space $(\cC([0,T],\real^d \times F(U)), \norm{\cdot}_\infty)$.
Therefore, letting $\tilde{y}$ be an arbitrary element of $\real^d \times \cP(U)$, the process $\barY \colon \Omega \longrightarrow \cC([0,T],\real^d \times \cP(U))$ given by
\begin{equation}\label{eq:MeanFieldDefSoluzione}
\barY(\omega) \coloneqq
\begin{dcases}
\lim_{n\to \infty} \barY_{n}(\omega), &\omega \in \tilde{\cZ}^c \\
\tilde{y}, &\omega \in \tilde{\cZ}
\end{dcases}
\end{equation}
is well defined.
\end{prop}

\begin{proof} The proof is the same that as Proposition \ref{prop:IterateCauchy}; estimate \eqref{eq:MeanFieldfactInequality_t} replaces \eqref{eq:factInequality_t}.
\end{proof}

The proof of the following proposition is analogous to that of Proposition \ref{prop:SolContAdatt}.
\begin{prop}\label{prop:MeanFieldSolContAdatt}The limiting stochastic process $\barY$ defined in Proposition~\ref{prop:MeanFieldIterateCauchy} is adapted to the filtration $(\sF_t)_{t \in [0,T]}$.
\end{prop}

We now show that the process $\barY$ is a strong solution to the auxiliary problem \eqref{eq:AuxiliaryMeanField}.
We write $\barY$ as
\begin{equation*}
\barY_t = (\barX_t,	\barLambda_t),
\end{equation*}
for all $t \in [0, T]$, where, by construction, $\barX$ and $\barLambda$ are the uniform limits of the sequences of processes $\{\barX_n\}_{n=0}^\infty$ and $\{\barLambda_n\}_{n=0}^\infty\,$, respectively.
We now observe that, by virtue of properties \eqref{eq:MeanFieldLip-v}, \eqref{eq:MeanFieldLip-sigma}, and \eqref{eq:MeanFieldLip-G}, we have that
\begin{align*}
\sup_{t \in [0,T]}\norm{v_{\Psi_t}(\barX_{n,t},\barLambda_{n,t}) - v_{\Psi_t}(\barr{X}_t,\barLambda_t)}_{\real^d} &\le L_v\sup_{t \in [0,T]}\norm{\barY_{n,t} - \barY_t}_{\real^d\times F(U)}, \\
\sup_{t \in [0,T]}\norm{\sigma_{\Psi_t}(\barX_{n,t},\barLambda_{n,t}) - \sigma_{\Psi_t}(\barr{X}_t,\barLambda_t)}_{\real^{d \times m}} &\le L_\sigma\sup_{t \in [0,T]}\norm{\barY_{n,t} - \barY_t}_{\real^d\times F(U)}, \\
\sup_{t \in [0,T]}\norm{\cG_{\Psi_t}(\barr{X}_{n,t},\barLambda_{n,t}) - \cG_{\Psi_t}(\barr{X}_t,\barLambda_t)}_{F(U)} &\le L_\cG\sup_{t \in [0,T]}\norm{\barY_{n,t} - \barY_t}_{\real^d\times F(U)};
\end{align*}
therefore we can replicate the argument of Section~\ref{sec_existenceN} (see the discussion just before Proposition~\ref{prop:SolDerivate}), by passing to the limit as $n \to \infty$ in the definition of the iterations \eqref{eq:IterationsMeanField}, and conclude that, $\prob$-almost surely, the process $\barY$ satisfies
\begin{equation}
\begin{dcases}
\barX_t = \barX_0 + \int_0^t v_{\Psi_s}(\barX_s,\barLambda_s)\,\de s + \int_0^t\sigma_{\Psi_s}(\barX_s,\barLambda_s)\,\de \barB_s\,,\\
\barLambda_t = e^{-\frac{t}{\theta}}\barLambda_0 + \frac{1}{\theta}\int_0^t e^{\frac{s-t}{\theta}} \cG_{\Psi_s}(\barX_s,\barLambda_s)\,\de s,
\end{dcases}
\label{eq:SoluzioneProbIntermedioMeanField}
\end{equation}
for all $t \in [0,T]$. Finally, equation \eqref{eq:SoluzioneProbIntermedioMeanField} allows us to deduce that $\barY$ is a strong solution to the auxiliary problem \eqref{eq:AuxiliaryMeanField}. Indeed, arguing as in Proposition \ref{prop:SolDerivate}, we have the following result.

\begin{prop}
The stochastic process $\barY$ defined in Proposition~\ref{prop:MeanFieldIterateCauchy} satisfies, $\prob$-a.s.,
\begin{equation}
\begin{dcases}
\barX_t = \barr{X}_0 + \int_0^t v_{\Psi_s}(\barX_s,\barLambda_s)\,\de s + \int_0^t\sigma_{\Psi_s}(\barX_s,\barLambda_s)\,\de\barB_s\,, \\
\barLambda_t = \barLambda_0 + \int_0^t \cT_{\Psi_s}(\barX_s,\barLambda_s)\,\de s,
\end{dcases}
\qquad\text{for all $t \in [0,T]$.}
\end{equation}
\end{prop}

The proof of existence of a solution to the auxiliary problem \eqref{eq:AuxiliaryMeanField} is concluded.

The following \emph{a priori} estimate on the moments of the strong solutions to \eqref{eq:AuxiliaryMeanField} holds true.
\begin{prop}\label{prop:stima a priori Aux}
Let $\barY$ be a strong solution to problem \eqref{eq:AuxiliaryMeanField}, and let $\barX_0$ a random variable in $L^p(\Omega,\sF,\prob)$ with $p \ge 2$. Then the following estimate holds
\begin{equation}
\label{eq:APrioriEstimateAux}
\E\biggl[\sup_{t\in[0,T]} \norm{\barY_t}_{\real^d \times F(U)}^p\biggr] \le C^\Psi\Bigl(1 + \E\bigl[\norm{\barX_0}_{\real^d}^p\bigr]\Bigr),
\end{equation}
where $C^\Psi > 0$ depends on $p$, $T$, $M_v^\Psi$\,, $M_\sigma^\Psi$\,.
\end{prop}

\begin{proof} The argument is analogous to that of Proposition \ref{prop:stima a priori} (and of Proposition \ref{prop:a priori mean field}). Here, the sublinearity estimates \eqref{eq:MeanFieldSub-v} and \eqref{eq:MeanFieldSub-sigma} must be used, which are valid by Remark~\ref{rem: about measures on curves}.
\end{proof}

Finally, the proof of pathwise uniqueness is identical to the one presented in Section \ref{sec_pathwiseuniquenessN}. The Lipschitz estimates \eqref{eq:DiscreteEmpLip} must be replaced by estimates \eqref{eq:MeanFieldLip},
and the concluding argument is performed through Gr\"{o}nwall inequality and the \emph{a priori} estimate \eqref{eq:APrioriEstimateAux}. This concludes the proof of Proposition \ref{prop:AuxiliaryProblemWellPosedness}.

\subsection{Construction of a solution to the mean-field problem: fixed-point argument} 
\label{sec:MeanFieldFixedPoint}
In this paragraph, we conclude the proof of existence of a solution to the mean-field problem \eqref{eq:MeanField} through a fixed-point argument.
\begin{prop}\label{prop:DefS} Let us assume that \eqref{eq_fieldsprop} and \eqref{eq:T_prop_geometry} are satisfied and that the initial datum $\barY_0$ is such that its spatial component $\barX_0$ belongs to $L^2(\Omega,\sF,\prob)$. Then the map 
\begin{equation}\label{eq:DefS}
\begin{aligned}
\cS \colon \cP_2\bigl(\cC([0,T],\real^d\times\cP(U))\bigr) &\longrightarrow \cP_2\bigl(\cC([0,T],\real^d\times\cP(U))\bigr) \\
\Psi &\longmapsto \cS(\Psi) \coloneqq \Law(\barY^\Psi),
\end{aligned}
\end{equation}
is well defined. Here, $\barY^\Psi \colon \Omega \longrightarrow \cC([0,T],\real^d\times\cP(U))$ is the pathwise unique strong solution of the auxiliary problem \eqref{eq:AuxiliaryMeanField} associated with the measure $\Psi \in \cP_2\bigl(\cC([0,T],\real^d\times\cP(U))\bigr)$ (see Definition \ref{def:strongSolutionAux} and Proposition \ref{prop:AuxiliaryProblemWellPosedness}).
\end{prop}

\begin{proof}
The statement is a consequence of Proposition \ref{prop:AuxiliaryProblemWellPosedness}. 
Indeed, for any fixed $\Psi \in \cP_2\bigl(\cC([0,T],\real^d\times\cP(U))\bigr)$, we can construct a strong solution $\barY^\Psi \colon \Omega \longrightarrow \cC([0,T],\real^d\times\cP(U))$ of the auxiliary problem \eqref{eq:AuxiliaryMeanField}, whose law $\Law(\barY^\Psi)$ is unambiguously determined owing to the pathwise uniqueness. 
In fact, given to strong solutions of \eqref{eq:AuxiliaryMeanField} $\barY^\Psi_1$ and $\barY^\Psi_2$, pathwise uniqueness implies that the $\cC([0,T],\real^d\times\cP(U))$-valued random variables $\barY^\Psi_1$ and $\barY^\Psi_2$ coincide $\prob$-almost surely. Since two random variables coinciding almost surely necessarily have the same law, we conclude that $\cS(\Psi)$ is well defined. 
Finally, to show that $\cS(\Psi) \in \cP_2\bigl(\cC([0,T],\real^d\times\cP(U))\bigr)$, we use estimate \eqref{eq:APrioriEstimateAux}, which ensures that the quantity
\begin{equation*}
    \E\biggl[\sup_{t\in[0,T]} \norm{\barY^\Psi_t}_{\real^d \times F(U)}^2\biggr] \le C^\Psi\Bigl(1+\E\bigl[\norm{\barX_0}_{\real^d}^2\bigr]\Bigr)
\end{equation*}
is finite. Therefore $\cS(\Psi) = \Law(\barY^\Psi)$ is a measure with finite second moment.
\end{proof}

For all $t \in [0,T]$, and for all measures $\Psi^1, \Psi^2 \in \cP_2\bigl(\cC([0,T],\real^d \times \cP(U))\bigr)$, we define the quantity 
\begin{equation*}
    W_{2,t}(\Psi^1,\Psi^2) \coloneqq \biggl(\inf_{\pi \in \Gamma(\Psi^1,\Psi^2)} \int_{\cC([0,T],\real^d \times \cP(U))^2} \sup_{u\in[0,t]} \norm{\varphi_1(u) - \varphi_2(u)}_{\real^d \times F(U)}^2 \, \de \pi(\varphi_1,\varphi_2)\biggr)^\frac{1}{2}.
\end{equation*}
\begin{rem}\label{rem:vuoifareunremark}
We highlight the following immediate facts about $W_{2,t}$\,.
\begin{enumerate}
\item[(a)] The infimum defining $W_{2,t}(\Psi^1,\Psi^2)$ is actually achieved for some optimal transport plan $\tilde{\pi} \in \Gamma(\Psi^1,\Psi^2)$, as a consequence of the general theory of Optimal Transport on complete and separable metric spaces \cite[Theorem~2.10]{Optimal-Transport-Ambrosio}.
\item[(b)] If $s \le t$, then $W_{2,s} \le W_{2,t}$\,; the quantity $W_{2,T}$ 
is the customary $W_2$ distance on $\cP_2\bigl(\cC([0,T],\real^d \times \cP(U))\bigr)$.
\end{enumerate}
\end{rem}

We now prove some estimates regarding the iteration of the map $\cS$. 

\begin{lemma}\label{lemma:4.20}
For all $\Psi^1,\Psi^2 \in \cP_2(\cC([0,T],\real^d \times \cP(U)))$, $n = 1,2,\dots$, and $t \in [0,T]$ we have that
\begin{subequations}\label{eq:4.32}
\begin{align}
    W_{2,t}^2(\cS(\Psi^1),\cS(\Psi^2)) &\le \cK \,t\, W_2^2(\Psi^1,\Psi^2), \label{eq:contractionLemma1} \\
    W_{2,t}^2(\cS^n(\Psi^1),\cS^n(\Psi^2)) &\le \cK \int_0^t W_{2,s}^2(\cS^{n-1}(\Psi^1),\cS^{n-1}(\Psi^2))\,\de s, \label{eq:contractionLemman}
\end{align}
\end{subequations}
where $\cK > 0$ depends only on $L_v$, $L_\sigma$, $L_\cT$, and $T$.
\end{lemma}
\begin{proof} In what follows, we shall denote by $\barY^1\coloneqq \barY^{\Psi^1}$ and $\barY^2\coloneqq \barY^{\Psi^2}$ the strong solutions to the auxiliary problem \eqref{eq:AuxiliaryMeanField} associated with the measures $\Psi^1$ and $\Psi^2$, respectively. We first observe that, for all fixed $t \in [0,T]$, we have 
\begin{equation}\label{eq:420i}
\begin{split}
    &\,W^2_{2,t}(\cS(\Psi^1),\cS(\Psi^2)) \\
    \le& \int_{\cC([0,T],\real^d \times \cP(U))^2} \sup_{u\in[0,t]} \norm{\varphi_1(u) - \varphi_2(u)}_{\real^d \times F(U)}^2\, \de P_{(\barY^1, \barY^2)}(\varphi_1,\varphi_2) \\
    =&\, \E\biggl[\sup_{u \in [0,t]} \norm{\barY^1_u -  \barY^2_u}_{\real^d \times F(U)}^2\biggr],
\end{split}
\end{equation}
where $P_{(\barY^1,\barY^2)} \coloneqq \Law(\barY^1,\barY^2) = (\barY^1,\barY^2)_\sharp\prob\,$, which is an element of $\Gamma(\cS(\Psi^1),\cS(\Psi^2))$. 
Owing to H\"{o}lder inequality, \eqref{eq_BDG}, \eqref{eq:v_prop}, and \eqref{eq:sigma_prop}, we observe that 
\begin{equation*}
\begin{split}
    &\,\E\biggl[\sup_{u\in[0,t]} \norm{\barX^1_u - \barX^2_u}_{\real^d}^2\biggr] \\
    \le&\, 2t\,\E\biggl[\int_0^t\norm{v_{\Psi^1_s}(\barY^1_s) - v_{\Psi^2_s}(\barY^2_s)}_{\real^d}^2\,\de s \biggr] + 8\E\biggl[\int_0^t\norm{\sigma_{\Psi^1_s}(\barY^1_s) - \sigma_{\Psi^2_s}(\barY^2_s)}_{\real^{d \times m}}^2\,\de s\biggr] \\
    \le&\, 2t\, \E\biggl[\int_0^t \bigl(L_v(\norm{\barY^1_s - \barY^2_s}_{\real^d \times F(U)} + W_1(\Psi^1_s,\Psi^2_s))\bigr)^2\,\de s\biggr] \\
    &\,+ 8\E\biggl[\int_0^t \bigl(L_\sigma(\norm{\barY^1_s - \barY^2_s}_{\real^d \times F(U)} + W_1(\Psi^1_s,\Psi^2_s))\bigr)^2\,\de s\biggr] \\
    \le&\, 4(tL_v^2 + 4L_\sigma^2)\cdot \E\biggl[\int_0^t \bigl(\norm{\barY^1_s - \barY^2_s}_{\real^d \times F(U)}^2 + W_1^2(\Psi^1_s,\Psi^2_s)\bigr)\,\de s\biggr]
\end{split}
\end{equation*}
and, owing to H\"{o}lder inequality and \eqref{eq:T_prop}, we analogously observe that
\begin{equation*}
\begin{split}
&\,\E\biggl[\sup_{u\in[0,t]}\norm{\barLambda^1_u - \barLambda^2_u}_{F(U)}^2\biggr] \le \E\biggl[t\int_0^t \norm{\cT_{\Psi^1_s}(\barY^1_s) - \cT_{\Psi^2_s}(\barY^2_s)}^2_{F(U)}\de s\biggr] \\
\le&\, t\, \E\biggl[\int_0^t \bigl(L_\cT(\norm{\barY^1_s - \barY^2_s}_{\real^d \times F(U)} + W_1(\Psi^1_s,\Psi^2_s))\bigr)^2\,\de s\biggr] \\
\le&\, 2t\,L_\cT^2\, \E\biggl[\int_0^t \bigl(\norm{\barY^1_s - \barY^2_s}_{\real^d \times F(U)}^2 + W_1^2(\Psi^1,\Psi^2)\bigr)\,\de s\biggr].
\end{split}
\end{equation*}
Combining the previous inequalities, we obtain the estimate
\begin{equation*}
\begin{split}
&\,\E\biggl[\sup_{u \in [0,t]} \norm{\barY^1_u -  \barY^2_u}_{\real^d \times F(U)}^2\biggr] \le 2\E\biggl[\sup_{u\in[0,t]} \norm{\barX^1_u - \barX^2_u}_{\real^d}^2\biggr] + 2\E\biggl[\sup_{u\in[0,t]}\norm{\barLambda^1_u - \barLambda^2_u}_{F(U)}^2\biggr] \\
\le&\,4 \bigl(2tL_v^2 + 8L_\sigma^2 + tL_\cT^2\bigr) \cdot \E\biggl[\int_0^t \bigl(\norm{\barY^1_s - \barY^2_s}_{\real^d \times F(U)}^2 + W_1^2(\Psi^1_s,\Psi^2_s)\bigr)\,\de s\biggr] \\
\le&\,4 \bigl(2TL_v^2 + 8L_\sigma^2 + TL_\cT^2\bigr)\cdot\biggl(\int_0^t \E\biggl[\sup_{u\in[0,s]} \norm{\barY^1_u - \barY^2_u}_{\real^d \times F(U)}^2\biggr] \,\de s + \int_0^t W_1^2(\Psi^1_s,\Psi^2_s)\,\de s\biggr).
\end{split}
\end{equation*}
To estimate the last integrand above, let us fix $s \le s'$ and let us consider a measure $\tilde{\pi} \in \Gamma(\Psi^1,\Psi^2)$ which is optimal for $W_{2,s'}$ (see Remark~\ref{rem:vuoifareunremark}(a)).
We also define the probability measure over $(\real^d \times \cP(U))^2$ given by $\tilde{\pi}_s \coloneqq (\ev_s,\ev_s)_\sharp \tilde{\pi} \in \Gamma(\Psi^1_s,\Psi^2_s)$ . We have that
\begin{equation*}
\begin{split}
    W_1^2(\Psi^1_s,\Psi^2_s) &\le \biggl(\int_{(\real^d \times \cP(U))^2}\norm{y_1 - y_2}_{\real^d \times F(U)}\,\de\tilde{\pi}_s(y_1,y_2)\biggr)^2 \\
    &\le \int_{(\real^d \times \cP(U))^2}\norm{y_1 - y_2}_{\real^d \times F(U)}^2\,\de\tilde{\pi}_s(y_1,y_2) \\
    &= \int_{\cC([0,T],\real^d \times \cP(U))^2}\norm{\varphi_1(s) - \varphi_2(s)}_{\real^d \times F(U)}^2\,\de\tilde{\pi}(\varphi_1,\varphi_2) \\
    &\le \int_{\cC([0,T],\real^d \times \cP(U))^2} \sup_{u\in[0,s']} \norm{\varphi_1(u) - \varphi_2(u)}_{\real^d \times F(U)}^2 \, \de \tilde{\pi}(\varphi_1,\varphi_2) = W_{2,s'}^2(\Psi^1,\Psi^2),
\end{split}
\end{equation*}
where we have used Jensen inequality in the second line and the optimality of $\tilde{\pi}$ in the last one. 
Using this last estimate with $s'=s$, we can write
\begin{equation*}
\E\biggl[\sup_{u \in [0,t]} \norm{\barY^1_u -  \barY^2_u}_{\real^d \times F(U)}^2\biggr]\! \le C_0 \biggl(\int_0^t \!\E\biggl[\sup_{u\in[0,s]} \norm{\barY^1_u - \barY^2_u}_{\real^d \times F(U)}^2\biggr] \,\de s + \int_0^t \!W_{2,s}^2(\Psi^1,\Psi^2)\,\de s\biggr),
\end{equation*}
with $C_0 \coloneqq 4 \bigl(2TL_v^2 + 8L_\sigma^2 + TL_\cT^2\bigr)$. 
Since the map $t \longmapsto \E\bigl[\sup_{u \in [0,t]} \norm{\barY^1_u -  \barY^2_u}_{\real^d \times F(U)}^2\bigr]$ is bounded over $[0,T]$ (see the \emph{a priori} estimate \eqref{eq:APrioriEstimateAux}) and the map $t \longmapsto \int_0^t W_{2,s}^2(\Psi^1,\Psi^2) \,\de s$ is non decreasing and bounded (indeed $\int_0^t W_{2,s}^2(\Psi^1,\Psi^2) \,\de s 
\le T W_2^2(\Psi^1,\Psi^2) < +\infty$, see Remark~\ref{rem:vuoifareunremark}(b)), by Gr\"{o}nwall inequality we conclude that, for all $t \in [0,T]$, estimate \eqref{eq:420i} becomes
\begin{equation}\label{eq:bottigliarotta}
    W_{2,t}^2(\cS(\Psi^1),\cS(\Psi^2)) \le \E\biggl[\sup_{u \in [0,t]} \norm{\barY^1_u -  \barY^2_u}_{\real^d \times F(U)}^2\biggr]  \le e^{C_0 t} C_0 \int_0^t W_{2,s}^2(\Psi^1,\Psi^2)\,\de s.
\end{equation}
By choosing $\cK \coloneqq e^{C_0 T}C_0$, inequality \eqref{eq:contractionLemma1} follows from Remark~\ref{rem:vuoifareunremark}(b), and inequality \eqref{eq:contractionLemman} follows by applying \eqref{eq:bottigliarotta} replacing $\Psi^i$ by $\cS^{n-1}(\Psi^i)$, for $i=1,2$.
\end{proof}

The following contractivity estimates holds true.
\begin{lemma}\label{lemma:4.21}
For all $\Psi^1,\Psi^2 \in \cP_2(\cC([0,T],\real^d \times \cP(U)))$ and $n = 1,2,\dots$, we have that
\begin{equation}\label{eq:contractionestimate}
    W_2(\cS^n(\Psi^1),\cS^n(\Psi^2)) \le \sqrt{\frac{(\cK T)^n}{n!}} W_2(\Psi^1,\Psi^2),
\end{equation}
where $\cK$ is the constant from Lemma~\ref{lemma:4.20}.
\end{lemma}
\begin{proof}
We proceed by induction to show that, for every $n=1,2,\ldots$ and for every $t\in[0,T]$, 
\begin{equation}\label{eq:perinduzione}
W_{2,t}^2(\cS^n(\Psi^1),\cS^n(\Psi^2))\leq \frac{(\cK\,t)^n}{n!} W_2^2(\Psi^1,\Psi^2).
\end{equation}
For $n=1$, \eqref{eq:perinduzione} is exactly \eqref{eq:contractionLemma1}. 
Let us assume that \eqref{eq:perinduzione} holds true for a given $n-1\ge0$.
By \eqref{eq:contractionLemman} and the inductive hypothesis, we can estimate
\begin{equation*}
\begin{split}
    W_{2,t}^2(\cS^n(\Psi^1),\cS^n(\Psi^2)) \le &\, \cK \int_0^t W_{2,s}^2(\cS^{n-1}(\Psi^1),\cS^{n-1}(\Psi^2))\,\de s \\
    \leq &\, \cK \int_0^t \frac{(\cK\,s)^{n-1}}{(n-1)!} W_2^2(\Psi^1,\Psi^2)\,\de s = \frac{(\cK\,t)^n}{n!}W_2^2(\Psi^1,\Psi^2),
\end{split}
\end{equation*}
which proves \eqref{eq:perinduzione}. 
Estimate~\eqref{eq:contractionestimate} follows by taking $t=T$ and by Remark~\eqref{rem:vuoifareunremark}(b).
\end{proof}

As a consequence of Lemma~\ref{lemma:4.21}, in the next proposition we prove that the map $\cS$ has a unique fixed point.
\begin{prop}\label{prop:fixedpoint}
Under the hypotheses of Proposition~\ref{prop:DefS}, the map $\cS$ defined in \eqref{eq:DefS} admits a unique fixed point in the space $\cP_2(\cC([0,T],\real^d\times\cP(U)))$.
\end{prop}
\begin{proof}
Since the coefficient $\sqrt{(\cK T)^n/n!}<1$ for $n$ sufficiently large, estimate~\eqref{eq:contractionestimate} implies that the correcponding iterate $\cS^n$ is a contraction in the complete metric space $\bigl(\cP_2(\cC([0,T],\real^d\times\cP(U))),W_2\bigr)$. 
A classical variant of the Banach--Caccioppoli fixed-point theorem (see, \emph{e.g.}, \cite[Esercizio~5, page~215]{PaganiSalsa}) allows us to conclude.
\end{proof}

\subsection{Uniqueness of the solution to the mean-field problem}\label{sec:MeanFieldConclusion}
Proposition~\ref{prop:fixedpoint} provides, in particular, pathwise uniqueness of the solution to problem \eqref{eq:MeanField} in the class of strong solutions whose law belongs to $\cP_2(\cC([0,T],\real^d\times\cP(U)))$.
We now prove that pathwise uniqueness, in fact, holds in the larger class of strong solutions satisfying \eqref{eq:boundedness of moment of solutions}.

To see this, notice that Proposition~\ref{prop:a priori mean field} implies that if $\barY$ is a solution to \eqref{eq:MeanField} and $\Sigma=\Law(\barY)$ satisfies \eqref{eq:boundedness of moment of solutions}, then the \emph{a priori} estimate \eqref{eq:APrioriEstimateMeanField} (with $p=2$) holds, which is equivalent to having $\Sigma\in\cP_2(\cC([0,T],\real^d\times\cP(U)))$. 
Uniqueness follows from Proposition~\ref{prop:fixedpoint}.
This concludes the proof of Theorem~\ref{thm:WellPosednessMeanField}.

\section{Propagation of chaos}\label{sec:propagation_of_chaos}
In this section, we prove the propagation of chaos, that is, the solution to the $N$-particle problem~\eqref{eq:NDicsreteProblem} converges to the solution to the mean-field problem~\eqref{eq:MeanField}, as $N\to\infty$.

We start by recalling that we work in the filtered probability space $(\Omega,\sF,(\sF_t)_{t\in[0,T]},\prob)$, which is fixed once and for all, and that we assume that the fields $v$, $\sigma$, and $\cT$ satisfy the structural assumptions~\eqref{eq_fieldsprop} and~\eqref{eq:T_prop_geometry}.

This section contains two main results.
In Proposition~\ref{prop:independence}, we prove that if we consider a sequence $\{(Y^n_0,B^n)\}_{n=1}^{\infty}$ of i.i.d.~pairs of initial datum and Brownian motion, then the corresponding solutions to the mean-field problem \eqref{eq:MeanField} maintain the i.i.d.~property.
In Theorem~\ref{thm:propagation_of_chaos}, we prove the pathwise propagation of chaos through synchronous coupling \cite[Section~3]{ChaosReviewII}. 
\begin{prop}[independence and identical distribution]
\label{prop:independence}
Let $\{(Y^n_0,B^n)\}_{n=1}^{\infty}$ be a sequence of independent and identically distributed random variables, where $\{Y^n_0\}_{n=1}^{\infty}$ is a sequence of $\sF_0$-measurable, $\real^d \times \cP(U)$-valued random variables with $Y^n_0 = (X^n_0,\Lambda^n_0)$ and $\{X_0^n\}_{n=1}^\infty \subseteq L^2(\Omega,\mathscr{F},\prob)$, and where $\{B^n\}_{n=1}^{\infty}$ is a sequence of $m$-dimensional standard Brownian motions.
Then, denoting with $\barY^n$ the solution to the mean-field problem \eqref{eq:MeanField} constructed from the initial datum $Y^n_0$ and the Brownian motion $B^n$ (for every $n \in \nat^+$), the processes $\{\barY^n\}_{n=1}^\infty$ are independent and identically distributed.
\end{prop}
The proof is obtained from the following technical result concerning the structure of the map~$\cS$ defined in~\eqref{eq:DefS}.

\begin{prop}[properties of the map $\cS$]
\label{prop:propertiesS}
Under the hypotheses of Proposition~\ref{prop:DefS}, the map~$\cS$ defined in~\eqref{eq:DefS}
is determined only by the law of the initial datum $\barY_0$.
\end{prop}

\begin{proof} Let us consider two identically distributed initial data $\barY^1_0$, $\barY^2_0$ and two $m$-dimensional Brownian motions $\barB^1$, $\barB^2$. 
We want to show that, for a fixed measure $\Psi\in \cP_2(\cC([0,T],\real^d\times\cP(U)))$, the laws of the solutions $\barY^1$ and $\barY^2$ to the auxiliary problems
\begin{equation}\label{eq:aux12}
\begin{dcases}
\barX^i_t = \barX_0^i + \int_0^t v_{\Psi_s}(\barX^i_s,\barLambda^i_s)\,\de s + \int_0^t \sigma_{\Psi_s}(\barX^i_s,\barLambda^i_s)\,\de \barB^i_s\,, \\
\barLambda^i_t = \barLambda_0^i +\int_0^t \cT_{\Psi_s}(\barX^i_s,\barLambda^i_s)\,\de s\,,
\end{dcases}
\end{equation}
($i = 1,2$) actually coincide.
From the proof of well-posedness of the auxiliary problem~\eqref{eq:aux12} (see Proposition~\ref{prop:AuxiliaryProblemWellPosedness}), we know that, for $i = 1,2$, the solutions $\barY^i = (\barX^i,\barLambda^i)$ are $\prob$-almost surely the uniform limit of the iterations $\{\barY^i_k\}_{k=0}^\infty$\,, which we recall to be defined as
\begin{equation}
\begin{dcases}
\barX^i_{k,t} = \barX^i_0 + \int_0^t v_{\Psi_s}(\barX^i_{k-1,s},\barLambda^i_{k-1,s})\,\de s + \int_0^t \sigma_{\Psi_s}(\barX^i_{k-1,s},\barLambda^i_{k-1,s})\,\de \barB^i_s\,, \\
\barLambda^i_{k,t} = e^{-\frac{t}{\theta}}\barLambda^i_0 + \frac{1}{\theta}\int_0^t e^{\frac{s-t}{\theta}}\Bigl(\underbrace{\barLambda^i_{k-1,s} + \theta\cT_{\Psi_s}(\barX^i_{k-1,s},\barLambda^i_{k-1,s})}_{= \cG_{\Psi_s}(\barX^i_{k-1,s},\barLambda^i_{k-1,s})}\Bigr)\,\de s,
\end{dcases}
\end{equation}
for $k = 0,1,2,\dots$ and 
for all $t\in[0,T]$.

Let us begin by proving that the two processes $(\barY^1_k,\barB^1)$ and $(\barY^2_k,\barB^2)$ (considered as random variables from $\Omega$ with values in $\cC([0,T],\real^d\times\cP(U)\times\real^m)$) are identically distributed for all $k = 0,1,2,\dots$ 
We proceed by induction on $k$: the case $k = 0$ follows from the fact that, by hypothesis, $\barY^1_0$ and $\barY^2_0$ are $\mathscr{F}_0$-measurable, hence (see, \emph{e.g.}, Exercise~3.4c in \cite[page~75]{Stochastic-Calculus}) they are independent of the Brownian motions $\barB^1$ and $\barB^2$ respectively. Therefore the random variables $(\barY^1_0,\barB^1)$ and $(\barY^2_0,\barB^2)$ are identically distributed since their joint law is the product of the laws of their components. 

We consider now $k\ge1$ and we suppose that $(\barY^1_{k-1},\barB^1)$ and $(\barY^2_{k-1},\barB^2)$ are identically distributed; to show that $(\barY^1_k,\barB^1)$ and $(\barY^2_k,\barB^2)$ have in turn the same law, we prove that their \emph{finite-dimensional distributions coincide}, that is, for any finite collection of times $t_1,\dots,t_r \in [0,T]$, the random vectors
$(\barY^i_k(t_1),\barB^i(t_1),\dots,\barY^i_k(t_r),\barB^i(t_r))$, for $i=1,2$, 
are identically distributed (the fact that we are working with continuous processes taking values in a separable Banach space ensures the applicability of the criterion). We observe that, for any time $t_j \in \{t_1,\dots,t_r\}$, we can construct a sequence of partitions $\pi^{k,j}_h\coloneqq \{0=\tau_0<\tau_1<\cdots<\tau_{M_h^{k,j}}=t_j\}$\footnote{To be precise, each node $\tau_\eta\in\pi_h^{k,j}$ should be labeled with $k$, $j$, and $h$; we omit this explicit dependence to keep the notation lighter.} of the interval $[0,t_j]$ with vanishing amplitude 
as $h \to \infty$, and such that
\begin{align*}
    &\int_0^{t_j} v_{\Psi_s}(\barY^i_{k-1,s})\,\de s = \lim_{h\to\infty} \sum_{\eta=0}^{M_h^{k,j}-1} v_{\Psi_{\tau_{\eta}}}(\barY^i_{k-1,\tau_{\eta}})(\tau_{\eta+1}-\tau_{\eta}),\quad \prob\text{-a.s.}\\
    &\int_0^{t_j} \sigma_{\Psi_s}(\barY^i_{k-1,s})\,\de \barB^i_s = \lim_{h\to\infty} \sum_{\eta=0}^{M_h^{k,j}-1} \sigma_{\Psi_{\tau_{\eta}}}(\barY^i_{k-1,\tau_\eta})(\barB^i_{\tau_{\eta+1}}-\barB^i_{\tau_\eta})\,,\quad \text{in $\prob$-probability,} \\
    &\int_0^{t_j} e^{\frac{s-t_j}{\theta}}\cG_{\Psi_s}(\barY^i_{k-1,s})\,\de s = \lim_{h\to\infty} \sum_{\eta=0}^{M_h^{k,j}-1} e^{\frac{\tau_\eta-t_j}{\theta}} \cG_{\Psi_{\tau_{\eta}}}(\barY^i_{k-1,\tau_\eta})(\tau_{\eta+1}-\tau_{\eta}),\quad \prob\text{-a.s.}
\end{align*}
for $i = 1,2$; this is possible thanks to the properties of Lebesgue, It\={o} (see \cite[Proposition~7.4]{Stochastic-Calculus}), and Bochner integrals, owing to the continuity of the integrands. Therefore, from the sequences of partitions $\{\pi^{k,j}_h\}_{h=1}^\infty$, $j = 1\dots,r$, we can extract suitable subsequences of partitions (which we will denote with the same set of symbols $\pi^{k,j}_h$) such that all limits in $\prob$-probability are $\prob$-almost sure limits.

Therefore, for any $j=1,\dots,r$ and for $i = 1,2$, we can write that, $\prob$-almost surely,
\begin{equation*}
    \barY^i_k(t_j) = \lim_{h\to\infty}\Phi^{k,j}_h\bigl(\barY^i_{k-1}(\tau_0),\barB^i(\tau_0),\dots,\barY^i_{k-1}\bigl(\tau_{M^{k,j}_h}\bigr),\barB^i\bigl(\tau_{M^{k,j}_h}\bigr)\bigr), 
\end{equation*}
where 
$\Phi^{k,j}_h \colon (\real^d\times\cP(U)\times\real^m)^{M^{k,j}_h+1}\longrightarrow \real^d\times F(U)$ is a suitable measurable map.
Hence, we conclude that, $\prob$-almost surely and for $i = 1,2$, we have
\begin{equation*}
(\barY^i_k(t_1),\barB^i(t_1),\dots,\barY^i_k(t_r),\barB^i(t_r)) = \lim_{h\to\infty} \Phi^{k}_h \bigl(\barY^i_{k-1}(\tau_0),\barB^i(\tau_0),\dots,\barY^i_{k-1}\bigl(\tau_{M^k_h}\bigr),\barB^i\bigl(\tau_{M^k_h}\bigr)\bigr) 
\end{equation*}
where
\begin{equation*}
    \tau_0,\dots,\tau_{M^k_h} \in \pi^k_h \coloneqq \bigcup_{j=1}^r \pi^{k,j}_h
\end{equation*}
and $\Phi^k_h \colon (\real^d\times\cP(U)\times\real^m)^{M^k_h+1}\longrightarrow (\real^d\times F(U)\times\real^m)^r$ is a suitable measurable map, common to $i=1,2$. Since, by inductive hypothesis, $(\barY^1_{k-1},\barB^1)$ and $(\barY^2_{k-1},\barB^2)$ have the same law, and hence the same finite-dimensional distributions, we conclude that for all fixed $h \in \nat$ we have
\begin{equation*}
\begin{split}
&\Law\bigl(\Phi^k_h(\barY^1_{k-1}(\tau_0),\barB^1(\tau_0),\dots,\barY^1_{k-1}(\tau_{N^k_h}),\barB^1(\tau_{M^k_h}))\bigr) \\
=&\Law\bigl(\Phi^k_h(\barY^2_{k-1}(\tau_0),\barB^2(\tau_0),\dots,\barY^2_{k-1}(\tau_{N^k_h}),\barB^2(\tau_{M^k_h})) \bigr).
\end{split}
\end{equation*}
By recalling that almost sure convergence implies weak convergence (that is, in duality with continuous and bounded functions) of the laws, we conclude that, for $i=1,2$,
\begin{equation*}
\begin{split}
&\,\Law\bigl(\Phi^k_h(\barY^i_{k-1}(\tau_0),\barB^i(\tau_0),\dots,\barY^i_{k-1}(\tau_{M^k_h}),\barB^i(\tau_{M^k_h}))\bigr) \\ &\, \xrightharpoonup[h\to\infty]{} \Law\bigl(\barY^i_k(t_1),\barB^i(t_1),\dots,\barY^i_k(t_r),\barB^i(t_r)\bigr), 
\end{split}
\end{equation*}
but the previous argument shows that the two sequences on the left-hand side are actually the same sequence:
by uniqueness of the weak limit of probability measures we deduce that
\begin{equation*}
\Law\bigl(\barY^1_k(t_1),\barB^1(t_1),\dots,\barY^1_k(t_r),\barB^1(t_r)\bigr) = \Law\bigl(\barY^2_k(t_1),\barB^2(t_1),\dots,\barY^2_k(t_r),\barB^2(t_r)\bigr),
\end{equation*}
and from the arbitrariness of times $t_1,\dots,t_r$ we finally conclude that $(\barY^1_k,\barB^1)$ and $(\barY^2_k,\barB^2)$ are identically distributed.

To conclude the proof, we need to show that $(\barY^1,\barB^1)$ and $(\barY^2,\barB^2)$ have the same law. 
We proceed similarly to the last lines of the first part of the proof. We know that
\begin{equation*}
\begin{split}
&(\barY^i,\barB^i) = \lim_{k\to\infty}(\barY^i_k,\barB^i),\quad \prob\text{-a.s., for $i=1,2$;}  \\
\end{split}
\end{equation*}
moreover, as shown earlier, for every $k=0,1,2,\dots$, we have
$
    \Law(\barY^1_k,\barB^1) = \Law(\barY^2_k,\barB^2),
$ 
therefore, relying again on the uniqueness of the weak limit, we conclude that
$
\Law(\barY^1,\barB^1) = \Law(\barY^2,\barB^2). $ 
Finally, by projecting onto the first component, it follows that
$
\Law(\barY^1) = \Law(\barY^2) 
$ 
and the proof is complete.
\end{proof}


\begin{proof}[Proof of Proposition~\ref{prop:independence}]
Let us begin by showing that the stochastic processes $\{\barY^n\}_{n=1}^\infty$ are identically distributed. 
From the proof of the well-posedness of the mean-field problem (see Theorem~\ref{thm:WellPosednessMeanField}), we know that $\Law(\barY^n)$ is the unique fixed point of the map $\cS_n$ defined as in \eqref{eq:DefS}
associating to the measure $\Psi \in \cP_2\bigl(\cC([0,T],\real^d \times \cP(U))\bigr)$ the law of the stochastic process solving the auxiliary problem \eqref{eq:AuxiliaryMeanField} with data $(Y_0^n,B^n)$. 
Proposition~\ref{prop:propertiesS} implies that all maps $\{\cS_n\}_{n=1}^\infty$ actually coincide, since the random variables $\{(Y^n_0, B^n)\}_{n=1}^{\infty}$\,, and therefore the initial data $\{Y_0^n\}_{n=1}^\infty$\,, are identically distributed. 
By uniqueness of the fixed point of the maps $\{\mathcal{S}_n\}_{n=1}^\infty$ (see Proposition~\ref{prop:fixedpoint}), we conclude that
$
    \Law(\barY^1) = \Law(\barY^2) = \cdots = \Law(\barY^n) = \cdots,
$ 
thus proving the claim.

We now want to show that the family of sub-$\sigma$-algebras $\{\sigma(\barY^n_t,t\in[0,T])\}_{n=1}^\infty$ of $\sF$ is independent. 
By hypothesis, the family of sub-$\sigma$-algebras
$    \bigl\{\sigma(Y^n_0;B^n_t,t\in[0,T])\bigr\}_{n=1}^\infty$
is independent.
Letting $\sN$ be the collection of the $\prob$-negligible elements of $\sF$, Lemma~\ref{lemma:augmented_sub_sigma_algebras} implies that also the family
$\bigl\{\sigma(Y^n_0;B^n_t,t\in[0,T];\mathscr{N})\bigr\}_{n=1}^\infty$
is independent. 
By arguing as in Proposition~\ref{prop:propertiesS}, we can see that, for all $t\in [0,T]$, each random variable $\barY^n_t$ is measurable with respect to the $\sigma$-algebra $\sigma(Y^n_0;B^n_t,t\in[0,T];\mathscr{N})$. Indeed, $\barY^n_t$ is the $\prob$-almost sure limit of the sequence of iterations $\{\barY^n_{k,t}\}_{k=0}^\infty$, which are, in turn, $\prob$-almost sure limits of $\sigma(Y^n_0;B^n_t,t\in[0,T])$-measurable functions. As a consequence, we have that, for all $n \in \nat^+$,
\begin{equation*}
        \sigma(\barY^n_t,t\in[0,T]) \subseteq \sigma(Y^n_0;B^n_t,t\in[0,T];\mathscr{N})
\end{equation*}
and since the family $\{\sigma(Y^n_0;B^n_t,t\in[0,T];\mathscr{N})\}_{n=1}^\infty$ is independent, the same holds true for the family
$\{\sigma(\barY^n_t,t\in[0,T])\}_{n=1}^\infty$\,, and the proof is concluded.
\end{proof}

We are now in a position to state and prove the main result of this section, concerning the propagation of chaos as the number of particles $N\to\infty$.
\begin{thm}[propagation of chaos]\label{thm:propagation_of_chaos}
Let $\{(Y^n_0,B^n)\}_{n=1}^{\infty}$ be as in Proposition~\ref{prop:independence}. 
For every fixed $N \in \nat^+$, denote with $Y^{1,N},\dots,Y^{N,N}$ the stochastic processes solving the $N$-particle problem~\eqref{eq:NDicsreteProblem}, namely
\begin{equation*}
\begin{dcases}
X^{n,N}_t = X^n_0 + \int_0^t v_{\Sigma^{n,N}_s}(X^{n,N}_s,\Lambda^{n,N}_s)\,\de s + \int_0^t \sigma_{\Sigma^{n,N}_s}(X^{n,N}_s,\Lambda^{n,N}_s)\,\de B^n_s, \\
\Lambda^{n,N}_t = \Lambda^n_0 +\int_0^t \cT_{\Sigma^{n,N}_s}(X^{n,N}_s,\Lambda^{n,N}_s)\,\de s,
\end{dcases}
\quad\text{($n=1,\ldots,N$)}
\end{equation*}
where $\Sigma_t^{n,N}\coloneqq \frac1{N}\sum_{j=1}^N \delta_{(X^{j,N}_t,\Lambda^{j,N}_t)}$\,,
and denote with $\barY^1,\barY^2,\ldots$ the stochastic processes solving the mean field problems \eqref{eq:MeanField}
\begin{equation}\label{eq:diamogli_un_nome}
\begin{dcases}
\barX^n_t = X^n_0 + \int_0^t v_{\Sigma_s}(\barX^n_s,\barLambda^n_s)\,\de s + \int_0^t \sigma_{\Sigma_s}(\barX^n_s,\barLambda^n_s)\,\de B^n_s, \\
\barLambda^n_t = \Lambda^n_0 + \int_0^t \cT_{\Sigma_s}(\barX^n_s,\barLambda^n_s)\,\de s,
\end{dcases}
\quad\text{($n=1,2,\ldots$)}
\end{equation}
where $\Sigma_t \coloneqq \Law(\barY^1_t) = \Law(\barY^2_t) = \cdots$, for all $t\in[0,T]$.
Then 
\begin{equation}\label{eq:propagation_of_chaos}
\lim_{N\to\infty}\biggl(\max_{1\le n\le N} \E\biggl[\sup_{t \in [0,T]} \norm{Y^{n,N}_t - \barY^n_t}_{\real^d \times F(U)}^2\biggr]\biggr) = 0.
\end{equation}
\end{thm}

The proof relies on the following two technical propositions.
\begin{prop}\label{prop:5.4}
Under the hypotheses of Theorem~\ref{thm:propagation_of_chaos}, for all $N\in \nat^+$, define the empirical measures
\begin{equation}
\barSigma^N \coloneqq \frac{1}{N}\sum_{j=1}^N \delta_{\barY^j} \colon \Omega \longrightarrow \cP_2(\cC([0,T],\real^d \times \cP(U))),
\end{equation}
associated with the $N$ processes $\barY^1,\dots,\barY^N$, solving the first $N$ instances of the mean-field problems \eqref{eq:diamogli_un_nome}.  
Define $\Sigma \coloneqq \Law(\barY^1) = \Law(\barY^2) = \cdots$.
Then the following inequality holds:
\begin{equation}\label{eq:PropChaosPreliminaryEstimate}
    \max_{1\le n\le N} \E\biggl[\sup_{t \in [0,T]} \norm{Y^{n,N}_t - \barY^n_t}_{\real^d \times F(U)}^2\biggr] \le \mathcal{C}\, \E\Bigl[W_2^2(\barSigma^N,\Sigma)\Bigr],
\end{equation}
where $\mathcal{C}>0$ is a constant depending on $L_v$\,, $L_\sigma$\,, $L_\cT$\,, and $T$.
\end{prop}
\begin{proof}
For a fixed $N \in \nat^+$ and $t \in [0,T]$, by applying H\"{o}lder and Burkholder--Davis--Gundy \eqref{eq_BDG} inequalities, combined with the Lipschitz properties \eqref{eq:v_prop} and \eqref{eq:sigma_prop}, we have that for all $n = 1,\dots,N$
\begin{equation*}
\begin{split}
&\,\E\biggl[\sup_{u \in [0,t]} \norm{X^{n,N}_u - \barX^n_u}_{\real^d}^2\biggr]  \\
\le&\, 2t\,\E\biggl[\int_0^t\norm{v_{\Sigma^N_s}(Y^{n,N}_s) - v_{\Sigma_s}(\barY^n_s)}_{\real^d}^2\,\de s\biggr] + 8 \E\biggl[\int_0^t\norm{\sigma_{\Sigma^N_s}(Y^{n,N}_s) - \sigma_{\Sigma_s}(\barY^n_s)}^2_{\real^{d\times m}}\,\de s\biggr]  \\
\le& \,2t\,\E\biggl[\int_0^t\bigl(L_v\bigl(\norm{Y^{n,N}_s - \barY^n_s}_{\real^d\times F(U)} + W_1(\Sigma^N_s,\Sigma_s)\bigr)\bigr)^2\,\de s\biggr] +  \\
&\,+ 8 \E\biggl[\int_0^t\bigl(L_\sigma\bigl(\norm{Y^{n,N}_s - \barY^n_s}_{\real^d\times F(U)} + W_1(\Sigma^N_s,\Sigma_s)\bigr)\bigr)^2\,\de s\biggr] \\
\le&\, 4(tL_v^2 + 4L_\sigma^2)\cdot \E\biggl[\int_0^t \bigl(\norm{Y^{n,N}_s - \barY^n_s}_{\real^d\times F(U)}^2 + W_1^2(\Sigma^N_s,\Sigma_s)\bigr)\,\de s \biggr].
\end{split}
\end{equation*} 
Analogously, by applying H\"{o}lder inequality and the Lipschitz property \eqref{eq:T_prop}, we have that
\begin{equation*}
\begin{split}
&\,\E\biggl[\sup_{u \in [0,t]} \norm{\Lambda^{n,N}_u - \barLambda^n_u}_{F(U)}^2\biggr] \le \E\biggl[t\int_0^t\norm{\cT_{\Sigma^N_s}(Y^{n,N}_s) - \cT_{\Sigma_s}(\barY^n_s)}_{F(U)}^2\,\de s\biggr] \\
\le&\, t\,\E\biggl[\int_0^t\bigl(L_\cT\bigl(\norm{Y^{n,N}_s - \barY^n_s}_{\real^d\times F(U)} + W_1(\Sigma^N_s,\Sigma_s)\bigr)\bigr)^2\,\de s\biggr] \\
\le& \,2t\,L_\cT^2 \E\biggl[\int_0^t \bigl(\norm{Y^{n,N}_s - \barY^n_s}_{\real^d\times F(U)}^2 + W_1^2(\Sigma^N_s,\Sigma_s)\bigr)\,\de s \biggr].
\end{split}
\end{equation*} 
Combining the previous inequalities, we obtain that, for all $t \in [0,T]$ and $n = 1,\dots,N$,
\begin{equation*}
\begin{split}
&\E\biggl[\sup_{u \in [0,t]} \norm{Y^{n,N}_u - \barY^n_u}_{\real^d \times F(U)}^2\biggr] \le 2\E\biggl[\sup_{u \in [0,t]} \norm{X^{n,N}_u - \barX^n_u}_{\real^d}^2\biggr] + 2\E\biggl[\sup_{u \in [0,t]} \norm{\Lambda^{n,N}_u - \barLambda^n_u}_{F(U)}^2\biggr] \\
&\le 4 (2tL_v^2 + 8L_\sigma^2 + tL_\cT^2)\cdot\E\biggl[\int_0^t \bigl(\norm{Y^{n,N}_s - \barY^n_s}_{\real^d\times F(U)}^2 + W_1^2(\Sigma^N_s,\Sigma_s)\bigr)\,\de s \biggr] \\
&\le 4 (2TL_v^2 + 8L_\sigma^2 + TL_\cT^2)\cdot\biggl(\E\biggl[\int_0^t\norm{Y^{n,N}_s - \barY^n_s}_{\real^d\times F(U)}^2\,\de s\biggr] + \E\biggr[\int_0^t W_1^2(\Sigma^N_s,\Sigma_s)\,\de s\biggr]\biggr). 
\end{split}
\end{equation*} 
We shall now estimate the last integral in the previous inequality. We note that, by triangle inequality, for all $s \in [0,t]$, we have
\begin{equation*}
    W_1(\Sigma^N_s,\Sigma_s) \le  W_1(\Sigma^N_s,\barSigma^N_s) + W_1(\barSigma^N_s,\Sigma_s),
\end{equation*}
where $\barSigma_t^N \coloneqq (\ev_t)_\sharp \barSigma^N$; 
therefore we obtain that, for all $t\in[0,T]$,
\begin{equation*}
    \int_0^t W_1^2(\Sigma^N_s,\Sigma_s)\,\de s \le 2\int_0^t W_1^2(\Sigma^N_s,\barSigma^N_s)\,\de s + 2\int_0^t W_1^2(\barSigma^N_s,\Sigma_s)\,\de s \eqqcolon 2\mathrm{I}_t + 2\mathrm{II}_t\,.
\end{equation*}
We observe now that, by definition of the Wasserstein distance and by convexity of the square function, we have
\begin{equation*}
\begin{split}
\mathrm{I}_t &\le \int_0^t \biggl(\int_{(\real^d \times \cP(U))^2}\norm{x-y}_{\real^d \times F(U)}\,\de\biggl(\frac{1}{N}\sum_{i=1}^N \delta_{(Y^{n,N}_s,\barY^n_s)}\biggr)(x,y)\biggr)^2\,\de s \\
&\le \int_0^t \biggl(\frac{1}{N} \sum_{i=1}^N \norm{Y^{n,N}_s - \barY^n_s}_{\real^d\times F(U)} \biggr)^2\,\de s \le \int_0^t \frac{1}{N} \sum_{i=1}^N \norm{Y^{n,N}_s - \barY^n_s}_{\real^d\times F(U)}^2 \,\de s
\end{split}
\end{equation*}
and, since $W_1(\barSigma_s^N,\Sigma_s)\leq W_2(\barSigma^N_s,\Sigma_s) \le W_2(\barSigma^N,\Sigma)$, we immediately obtain that 
$ 
    \mathrm{II}_t \le 
    T W_2^2(\barSigma^N,\Sigma).
$ 
Therefore, by letting $C_0 \coloneqq 4 (2TL_v^2 + 8L_\sigma^2 + TL_\cT^2)$, we can write that, for all $t\in [0,T]$ and $n = 1,\dots,N$,
\begin{equation*}
\begin{split}
&\,\E\biggl[\sup_{u \in [0,t]} \norm{Y^{n,N}_u - \barY^n_u}_{\real^d \times F(U)}^2\biggr] \le C_0 \biggl(\E\biggl[\int_0^t\norm{Y^{n,N}_s - \barY^n_s}_{\real^d\times F(U)}^2\,\de s\biggr] \\
&\,+ 2\E\biggr[\int_0^t \frac{1}{N} \sum_{i=1}^N \norm{Y^{n,N}_s - \barY^n_s}_{\real^d\times F(U)}^2\,\de s\biggr] +  \E\Bigl[2T W_2^2(\barSigma^N,\Sigma)\Bigr] \biggr) \\
\le&\, C_0 \biggl(\int_0^t \max_{1\le n \le N} \E\biggl[\sup_{u \in [0,s]} \norm{Y^{n,N}_u - \barY^n_u}_{\real^d \times F(U)}^2\biggr]\,\de s \\
&\,+ 2\int_0^t \frac{1}{N}\cdot N \max_{1\le n \le N} \E\biggl[\sup_{u \in [0,s]} \norm{Y^{n,N}_u - \barY^n_u}_{\real^d \times F(U)}^2\biggr]\,\de s + 2T \E\biggl[W_2^2(\barSigma^N,\Sigma)\biggr] \biggr) \\
\le&\, 3C_0 \int_0^t \max_{1\le n \le N} \E\biggl[\sup_{u \in [0,s]} \norm{Y^{n,N}_u - \barY^n_u}_{\real^d \times F(U)}^2\biggr]\,\de s + 2TC_0  \E\Bigl[W_2^2(\barSigma^N,\Sigma)\Bigr].
\end{split}
\end{equation*} 
Now we can apply Gr\"{o}nwall inequality to the function
\begin{equation*}
    v(t) \coloneqq \max_{1\le n \le N} \E\biggl[\sup_{u \in [0,t]} \norm{Y^{n,N}_u - \barY^n_u}_{\real^d \times F(U)}^2\biggr],
\end{equation*}
which 
is bounded on $[0,T]$ as a consequence of the \emph{a priori} estimates \eqref{eq:APrioriEstimate} and \eqref{eq:APrioriEstimateMeanField} (with $p=2$), to conclude that, for all $t \in [0,T]$,
\begin{equation*}
    v(t) \le e^{3C_0 t} 2TC_0\,\E\Bigl[W_2^2(\barSigma^N,\Sigma)\Bigr],
\end{equation*}
and by substituting $t = T$ we finally obtain estimate \eqref{eq:PropChaosPreliminaryEstimate} with
$\mathcal{C} \coloneqq e^{3C_0 T} 2TC_0$.
\end{proof}

The next proposition, which is a consequence of the strong law of large numbers, can be found in \cite[Lemma~4.7.1]{Invitation-To-Wasserstein}; we present its proof in Section~\ref{proof:Panaretos} for the reader's convenience.
\begin{prop}\label{prop:Panaretos}
Let $Z_n \colon \Omega \longrightarrow \cX$ be a sequence of random variables defined on the probability space $(\Omega,\mathscr{F},\prob)$ taking values in a complete and separable metric space $(\cX,d)$. Let us assume that the random variables $\{Z_n\}_{n=1}^\infty$ are i.i.d., and, for a certain $p \in [1,+\infty)$, they all have finite $p$-th moment.
Then, by denoting with $\mu \in \cP_p(\cX)$ the common law of the $Z_n$'s and with $\barr{\mu}^N$ the empirical measure
\begin{equation*}
    \barr{\mu}^N \coloneqq \frac{1}{N}\sum_{n=1}^N\delta_{Z_n} \colon \Omega \longrightarrow \cP_p(\cX),\qquad N\in\nat^+,
\end{equation*}
we have
\begin{equation}\label{eq:Panaretos}
    \lim_{N \to \infty}\E\bigl[W_p^p(\barr{\mu}^N,\mu)\bigr] = 0.
\end{equation}
\end{prop}

Thanks to Propositions~\ref{prop:5.4} and~\ref{prop:Panaretos}, we can now obtain our propagation of chaos result.
\begin{proof}[Proof of Theorem~\ref{thm:propagation_of_chaos}]
By Proposition~\ref{prop:independence}, the processes $\{\barY^n\}_{n=1}^\infty$ are i.i.d.~with common law $\Sigma\in \cP_2(\cC([0,T],\real^d\times\cP(U)))$, so that \eqref{eq:propagation_of_chaos} holds by combining \eqref{eq:PropChaosPreliminaryEstimate} and \eqref{eq:Panaretos} in Proposition~\ref{prop:Panaretos}, which we apply in the complete metric space $\cX=\cC([0,T],\real^d\times\cP(U))$ to the random variables $Z_n=\barY^n$ and with $p=2$.
\end{proof}

\smallskip 




\appendix
\section{Proofs of some technical propositions}\label{app}
In this appendix, we present the proofs of some propositions with a more probabilistic flavour, which we left behind in the preceding text.
To do so, we start by recalling some known facts on the measurability of maps with values in metric and Banach spaces: the statements of the next two propositions are borrowed from \cite{ProbDistBanachSpaces} and adapted to our needs.

\begin{prop}[{\cite[Chapter I, Proposition 1.1]{ProbDistBanachSpaces}}]\label{prop:measurelim}
Let $(\cA,\sA)$ be a measurable space and let~$\cX$ be a metric space endowed with its Borel $\sigma$-algebra $\sB(\cX)$. Let $f_n \colon \cA \longrightarrow \cX$, $n \in \nat$, be $(\sA,\sB(\cX))$-measurable maps. If the map $f \colon \cA \longrightarrow \cX$ is the pointwise limit of the sequence $\{f_n\}_{n=0}^\infty$\,, then it is $(\sA,\sB(\cX))$-measurable.
    
\end{prop}

\begin{prop}[{\cite[Chapter II, Theorem 1.1]{ProbDistBanachSpaces}}]\label{prop:nonPettis} 
Let $(\cA,\sA)$ be a measurable space and $E$ be a separable Banach space endowed with its Borel $\sigma$-algebra $\sB(E)$. If $f \colon \cA \longrightarrow E$ then the following are equivalent:
\begin{enumerate}
    \item[(a)] the map $f$ is $(\sA,\sB(E))$-measurable;
    \item[(b)] for every continuous linear functional $\varphi \in E^*$, the function $\langle\varphi,f\rangle \colon \cA \longrightarrow \real$ is $(\sA,\sB(\real))$-measurable.
\end{enumerate}
\end{prop}
\subsection{Proofs of Propositions~\ref{prop:DiscreteIterationsContAdapt} and~\ref{prop:SolContAdatt}}\label{app_adapted}
\begin{proof}[Proof of Proposition \ref{prop:DiscreteIterationsContAdapt}]
We proceed by induction on $n$. By definition \eqref{eq:DiscreteIterations0}, we have that $\bX_{0,t} = \bX_0$ for all $t \in [0,T]$, therefore the initial (spatial) iteration is clearly continuous, and it is $(\sF_t)_{t\in[0,T]}$-adapted since, by hypothesis, the initial datum $\bX_{0}$ is an $\sF_0$-measurable random variable. We argue identically for the process $\bLambda_0$\,.

We now consider $n \ge 0$, and assume that the $n$-th iterations $\bX_n$ and $\bLambda_n$ are continuous and $(\sF_t)_{t\in[0,T]}$-adapted. We first prove the continuity of the trajectories of $\bX_{n+1}$ and $\bLambda_{n+1}$. Let $\omega \in \Omega$ be fixed; we observe that the maps $[0,T] \ni s \longmapsto \bv_{\Sigma^N_{n,s}}(\bX_{n,s},\bLambda_{n,s})(\omega) \in (\real^d)^N$ and $s \longmapsto \bcG_{\Sigma^{N}_{n,s}}(\bX_{n,s},\bLambda_{n,s})(\omega) \in (F(U))^N$ are continuous, since they are obtained through the composition of continuous maps, therefore they are integrable (in the sense of Lebesgue and Bochner respectively) and  the two maps
\begin{equation*}
     t \longmapsto \int_0^t \bv_{\Sigma^{N,n}_s}(\bX_{n,s},\bLambda_{n,s})(\omega)\,\de s,\qquad 
     t \longmapsto \frac{1}{\theta}\int_0^t e^{\frac{s-t}{\theta}} \bcG_{\Sigma^{N,n}_s}(\bX_{n,s},\bLambda_{n,s})(\omega)\,\de s
\end{equation*}
are in turn continuous over $[0,T]$. The inductive hypothesis also implies that the $\real^{dN\times mN}$-valued process $(\bsigma_{\Sigma^{N,n}_t}(\bX_{n,t},\bLambda_{n,t}))_{t\in[0,T]}$ is continuous and $(\sF_t)_{t\in[0,T]}$-adapted, hence it is progressively measurable as a consequence of \cite[Proposition 2.1]{Stochastic-Calculus}, and satisfies the condition $\int_0^T \norm{\bsigma_{\Sigma^{N,n}_t}(\bX_{n,t},\bLambda_{n,t})}_{\real^{dN\times dm}}^2\,\de t < \infty$ everywhere on $\Omega$. Therefore, it is integrable in the sense of It\={o}, and the map 
\begin{equation*}
    t \longmapsto \biggl(\int_0^t \bsigma_{\Sigma^{N,n}_s}(\bX_{n,s},\bLambda_{n,s})\,\de \bB_s\biggr)(\omega)
\end{equation*}
is continuous over $[0,T]$ by virtue of the continuity properties of stochastic integrals (see, \emph{e.g.}, \cite[Theorem 7.3]{Stochastic-Calculus}). Recalling definition \eqref{eq:DiscreteIterations}, we conclude that the maps $[0,T] \ni t \longmapsto \bX_{n+1,t}(\omega) \in (\real^d)^N$ and 
$[0,T] \ni t \longmapsto \bLambda_{n+1,t}(\omega) \in (F(U))^N$ are continuous, thus proving the claim concerning continuity.

We now show that both $\bX_{n+1}$ and $\bLambda_{n+1}$ are $(\sF_t)_{t \in [0,T]}$-adapted. 
We start from the spatial components: owing to the progressive measurability of $\bX_n$ and $\bLambda_n$ (coming, again, from \cite[Proposition~2.1]{Stochastic-Calculus}), we deduce that, for every $t \in [0,T]$, the map $\Omega \times [0,t]\ni (\omega,s) \longmapsto \bv_{\Sigma^{N,n}_s}(\bX_{n,s},\bLambda_{n,s})(\omega)$ is $\mathscr{F}_t\otimes\mathscr{B}([0,t])$-measurable; by the Fubini--Tonelli Theorem (see, \emph{e.g.}, \cite[Theorem~1.2(a)]{Stochastic-Calculus}), we can conclude that, for every $t \in [0,T]$, the map
$\Omega \ni \omega \longmapsto \int_0^t \bv_{\Sigma^{N,n}_s}(\bX_{n,s},\bLambda_{n,s})(\omega)\,\de s$
is $\mathscr{F}_t$-measurable. 
By recalling that the random variable $\int_0^t \bsigma_{\Sigma^{N,n}_s}(\bX_{n,s},\bLambda_{n,s})\de \bB_s$ is $\mathscr{F}_t$-measurable, we conclude that so is $\bX_{n+1,t}$, whence the process $\bX_{n+1}$ is adapted.
Since $\bLambda_{n+1,t}$ takes valued in the separable Banach space $((F(U))^N,\norm{\cdot}_{(F(U))^N})$, by virtue of Proposition~\ref{prop:nonPettis}, we conclude if we show that, for every continuous linear functional $\varphi \in ((F(U))^N)^*$, the real-valued map
$\Omega\ni \omega \longmapsto \langle\varphi, \bLambda_{n+1,t}(\omega)\rangle$ is $\sF_t$-measurable. By taking $\varphi \in ((F(U))^N)^*$, we have that 
\begin{equation}
\langle\varphi, \bLambda_{n+1,t}(\omega)\rangle = 
e^{-\frac{t}{\theta}} \langle\varphi, \bLambda_0(\omega)\rangle + \frac{1}{\theta}\int_0^t e^{\frac{s-t}{\theta}} \Bigl\langle \varphi, \bcG_{\Sigma^{N,n}_s}(\bX_{n,s},\bLambda_{n,s})(\omega)\Bigr\rangle\,\de s,
\label{eq:duality}
\end{equation}
where we have used \cite[Chapter~5.5, Corollary~2]{Yosida} to swap the time integral and the duality action. 
Since the integral in~\eqref{eq:duality} is a Lebesgue integral, we can argue as for the spatial components
to obtain that the map
$\Omega \ni \omega \longmapsto \frac{1}{\theta}\int_0^t e^{\frac{s-t}{\theta}} \bigl\langle \varphi, \bcG_{\Sigma^{N,n}_s}(\bX_{n,s},\bLambda_{n,s})(\omega)\bigr\rangle\,\de s$
is $\mathscr{F}_t$-measurable. 
We conclude by arbitrariness of $\varphi$.
\end{proof}

\begin{proof}[Proof of Proposition \ref{prop:SolContAdatt}] 
Letting $t \in [0,T]$ and  $n = 0,1,2,\dots$, by defining the random variable $\widetilde{\bY}_{n,t} \colon \Omega \longrightarrow (\real^d \times F(U))^N$ as
\begin{equation*}
\widetilde{\bY}_{n,t}(\omega) \coloneqq
\begin{dcases}
 \bY_{n,t}(\omega), &\omega \in \cZ^c, \\
\tilde{\by}, &\omega \in \cZ, 
\end{dcases}
\end{equation*}
we can write $\bY_t(\omega)=\lim_{n\to\infty}\widetilde{\bY}_{n,t}(\omega)$\,, for every $\omega\in\Omega$. 
Proposition~\ref{prop:DiscreteIterationsContAdapt} grants that, for every $n = 0,1,2,\dots$, the random variables $\bY_{n,t}$ are $\sF_t$-measurable. 
Since the filtration $(\mathscr{F}_t)_{t \in [0,T]}$ is standard, $\cZ\in\sF_t$\,, so that the random variables $\widetilde{\bY}_{n,t}$ are $\mathscr{F}_t$-measurable. 
Since $\bY_t$ is the pointwise limit of a sequence of  $\mathscr{F}_t$-measurable random variables, by Proposition~\ref{prop:measurelim}, we conclude that it is $\sF_t$-measurable. The thesis follows from the arbitrariness of $t$.
\end{proof}

\subsection{On the independence of sub-$\sigma$-algebras augmented by null sets}\label{augmented_sub_sigma_algebras}
\begin{lemma}\label{lemma:augmented_sub_sigma_algebras}
Let $(\Omega, \sF, \prob)$ be a probability space, and let $\sN \subseteq \sF$ be the collection of its $\prob$-negligible subsets. If $\{\sA_\alpha\}_{\alpha\in I}$ is an independent family of sub-$\sigma$-algebras of $\sF$, so is the family $\{\barsA_\alpha\}_{\alpha\in I}$\,, where $\barsA_\alpha \coloneqq \sigma(\sA_\alpha \cup \sN)$, for all $\alpha \in I$.
\end{lemma}
\begin{proof}
We want to show that, for any finite collection of indices $\alpha_1,\dots,\alpha_r \in I$, and for any choice of $\barA_1 \in \barsA_{\alpha_1}$\,, \dots, $\barA_r \in \barsA_{\alpha_r}$\,, we have
\begin{equation}\label{eq:independenceDef}
    \prob(\barA_1\cap\dots\cap \barA_r) = \prod_{i=1}^r\prob(\barA_i).
\end{equation}
By a repeated application of the coincidence criterion for finite measures \cite[Theorem~1.1 and Remark~1.1]{Stochastic-Calculus}, it is straightforward to check that it suffices to verify condition \eqref{eq:independenceDef} for $\barA_1 \in \sC_{\alpha_1}$\,, \dots, $\barA_r \in \sC_{\alpha_r}$\,, where  $\sC_{\alpha_i}\subseteq\barsA_{\alpha_i}$ (for $i = 1,\ldots,r$) is a subclass which is stable with respect to finite intersection and such that $\sigma(\sC_{\alpha_i}) = \barsA_{\alpha_i}$\,. In the present case, we may choose the subclasses
\begin{equation*}
    \sC_{\alpha_i} \coloneq \{A \cap N : \text{$A \in \sA_{\alpha_i}$ and either $N = \Omega$ or $N \in \sN$}\},\qquad \text{for $i = 1,\dots,r$,}
\end{equation*}
which satisfy the desired properties. We observe that two cases can arise: either all the sets $\barA_1,\dots,\barA_r$ are of the form $A_i \cap \Omega$ with $A_i \in \sA_{\alpha_i}$\,, and in this case \eqref{eq:independenceDef} is satisfied owing to the assumed independence of the family $\{\sA_\alpha\}_{\alpha\in I}$\,, or one of the sets $\barA_1,\dots,\barA_r$ (say $\barA_1$) is of the form $A_1 \cap N_1$ with $A_1 \in \sA_{\alpha_1}$ and $N_1\in \sN$, in this case we have that
\begin{equation*}
    \prob(\barr{A}_1\cap\dots\cap\barr{A}_r) = \prob(A_1\cap N_1\cap\dots\cap A_r\cap N_r) = 0 = \prod_{i=1}^r \prob(A_i\cap N_i) = \prod_{i=1}^r \prob(\barr{A}_i),
\end{equation*}
since both $A_1\cap N_1\cap \dots\cap A_r\cap N_r$ and $A_1\cap N_1$ are contained in the negligible set $N_1$.
\end{proof}

\subsection{Proof of Proposition~\ref{prop:Panaretos}}\label{proof:Panaretos}
To obtain the proof of Proposition~\ref{prop:Panaretos}, the following generalization of Lebesgue's Dominated Convergence Theorem is needed.
\begin{prop}[{\cite[Theorem~1.20]{EvansGariepy}}]\label{prop:EG}
Let $(E,\sE,\nu)$ be a measure space and, for all $N\in\nat^+$, let $g_N\,,g$ be $\nu$-integrable functions, and $f_N\,,f$ be measurable functions. 
Assume that 
\begin{itemize}
    \item[(a)] $\lim_{N\to\infty}f_N=f$,  $\nu$-a.e.;
    \item[(b)] $\abs{f_N} \le g_N$, for every $N\in\nat^+$;
    \item[(c)] $\lim_{N\to\infty} g_N = g$, $\nu$-a.e.;
    \item[(d)] $\displaystyle \lim_{N\to\infty} \int_{E} g_N \,\de\nu= \int_{E}g\,\de\nu$.
\end{itemize}
Then the functions $\{f_N\}_{N=1}^\infty$ and $f$ are $\nu$-integrable and it holds
\begin{equation*}
\lim_{N\to\infty} \int_{E} f_N \,\de\nu = \int_{E}f\,\de\nu.
\end{equation*}
\end{prop}

\begin{proof}[Proof of Proposition~\ref{prop:Panaretos}]
We start by showing that for almost every $\omega \in \Omega$ we have that $W_p^p(\barmu^N(\omega),\mu) = 0$. By \cite[Theorem~6.9]{Optimal-Transport-Villani}, this is equivalent to showing that there exists a $\prob$-negligible set $\cN \in \mathscr{F}$ such that, for every $\omega \in \cN^c$, the following conditions hold:
\begin{itemize}
    \item[$(i)$](convergence of $p$-th moments) for some $z_0\in \cX$,
    \begin{equation*}
        \int_{\cX} d(z,z_0)^p \,\de (\barmu^N(\omega))(z) \longrightarrow \int_{\cX} d(z,z_0)^p \,\de \mu(z), \qquad \text{as $N\to\infty$,}
    \end{equation*} 
    \item[$(ii)$](weak convergence, in duality with $C_b(\cX)$) $\barmu^N(\omega) \xrightharpoonup{\phantom{M}} \mu$, as $N\to\infty$.
\end{itemize}
To show $(i)$, we observe that the real-valued random variables $\{d(Z_n,z_0)^p\}_{n=1}^\infty$ are i.i.d.~and have finite first moment. 
Thus, we can apply the strong law of large numbers \cite[Theorem~22.1]{Billingsley-PandM} to deduce the existence of a $\prob$-negligible set $\cN_0 \in \mathscr{F}$ such that, for every $\omega \in \cN_0^c$, 
\begin{equation*}
        \int_{\cX} d(z,z_0)^p \,\de (\barmu^N(\omega))(z) = \frac{1}{N}\sum_{n=1}^N d(Z_n(\omega),z_0)^p \xlongrightarrow[N\to\infty]{} \E[d(Z_1,z_0)^p]= \int_{\cX} d(z,z_0)^p \,\de \mu(z).
\end{equation*}
To show $(ii)$, we use the fact that, being $\cX$ a separable space, it suffices to test the convergence condition on a suitable countable collection $\{\varphi_k\}_{k=1}^\infty \subseteq C_b(\cX)$ (see the discussion at the beginning of \cite[Section~2.2.2]{Invitation-To-Wasserstein}). 
By applying, for every fixed $k$, the strong law of large numbers to the real-valued random variables $\{\varphi_k(Z_n)\}_{m=1}^\infty$ (which turn out to be i.i.d.~and with finite first moment), we conclude that for a certain $\prob$-negligible set $\cN_k\in\sF$\,, we have that, for every $\omega \in \cN_k^c$, 
\begin{equation*}
        \int_{\cX} \varphi_k(z) \,\de (\barmu^N(\omega))(z) = \frac{1}{N}\sum_{n=1}^N \varphi_k(Z_n(\omega)) \xlongrightarrow[N\to\infty]{} \E[\varphi_k(Z_1)]= \int_{\cX} \varphi_k(z) \,\de \mu(z).
\end{equation*}
Now the set $\cN\coloneqq \cup_{k=0}^\infty \cN_k$\, belongs to $\sF$, is $\prob$-negligible, and, for every $\omega\in \cN^c$, we have $W_p^p(\barmu^N(\omega),\mu) \to 0$, as $N\to\infty$, as desired. 

To obtain \eqref{eq:Panaretos}, we need to pass to the limit under the expectation, so that 
$$\lim_{N\to\infty}\E\bigl[W_p^p(\barmu^N,\mu)\bigr]=\E\Bigl[\lim_{N\to\infty} W_p^p(\barmu^N,\mu)\Bigr]=0.$$
To show the first equality above, we apply Proposition~\ref{prop:EG} with $(E,\sE,\nu)=(\Omega,\sF,\prob)$,
\begin{equation*}
g_N(\omega) \coloneqq \frac{1}{N}\sum_{n=1}^N\int_{\cX}d(Z_n(\omega),z)^p\,\de\mu(z), \;\;\text{for all $N$,}\quad\text{and}\quad
g(\omega) \coloneqq \E\biggl[\int_{\cX}d(Z_1,z)^p\,\de\mu(z)\biggr],
\end{equation*}
$f_N(\omega) \coloneqq W_p^p(\barmu^N(\omega),\mu)$, for all $N\in\nat^+$, and $f(\omega) \coloneqq 0$.
Indeed, condition (a) follows from what we have just proved; condition (b) is a consequence of the definition of $W_p$ by taking $\barmu^N(\omega)\otimes\mu$ as a competitor in \eqref{eq:WasserDef}; condition $(c)$ follows from the strong law of large numbers applied to the i.i.d.~random variables $\int_{\cX}d(Z_n,z)^p\,\de\mu(z)$; finally, condition $(d)$ is a consequence of the identical distribution of $Z_n$\,, so that $\E[g_N] = \frac{1}{N}\sum_{i=1}^N\E[\int_{\cX}d(Z_n,z)^p\,\de\mu(z)] = \E[\int_{\cX}d(Z_1,z)^p\,\de\mu(z)] = \E[g]$. 
The proof is concluded.
\end{proof}

\bibliographystyle{siam}
\bibliography{bibliography}

\end{document}